\documentclass[reqno, 11pt]{amsart} 

\usepackage[utf8]{inputenc}
\usepackage[T1]{fontenc}

\DeclareUnicodeCharacter{03A6}{\ensuremath{\Phi}}
\DeclareUnicodeCharacter{03B1}{\ensuremath{\alpha}}
\DeclareUnicodeCharacter{03B2}{\ensuremath{\beta}}
\DeclareUnicodeCharacter{03B3}{\ensuremath{\gamma}}
\DeclareUnicodeCharacter{03C0}{\ensuremath{\pi}}
\DeclareUnicodeCharacter{03C8}{\ensuremath{\psi}}
\DeclareUnicodeCharacter{03C9}{\ensuremath{\omega}}
\DeclareUnicodeCharacter{03B6}{\ensuremath{\zeta}}
\DeclareUnicodeCharacter{03C3}{\ensuremath{\sigma}}
\DeclareUnicodeCharacter{03C6}{\ensuremath{\phi}}
\DeclareUnicodeCharacter{03A8}{\ensuremath{\Psi}}
\DeclareUnicodeCharacter{0394}{\ensuremath{\Delta}}
\DeclareUnicodeCharacter{03B7}{\ensuremath{\eta}}
\DeclareUnicodeCharacter{03BC}{\ensuremath{\mu}}
\DeclareUnicodeCharacter{03C1}{\ensuremath{\rho}}
\DeclareUnicodeCharacter{03BB}{\ensuremath{\lambda}}
\DeclareUnicodeCharacter{03B9}{\ensuremath{\iota}}
\DeclareUnicodeCharacter{03BE}{\ensuremath{\xi}}
\DeclareUnicodeCharacter{03C4}{\ensuremath{\tau}}
\DeclareUnicodeCharacter{03C7}{\ensuremath{\chi}}
\DeclareUnicodeCharacter{03B5}{\ensuremath{\varepsilon}}
\DeclareUnicodeCharacter{03B4}{\ensuremath{\delta}}
\DeclareUnicodeCharacter{2202}{\ensuremath{\partial}}

\usepackage[english]{babel}

\usepackage[margin=1.2in]{geometry}
\usepackage{varwidth}
\usepackage{enumitem}
\usepackage{xcolor}

\usepackage{soul} 

\usepackage{amsmath, amsthm, amssymb}
\usepackage[all]{xy}

\usepackage[bookmarks = true]{hyperref}

\usepackage{cleveref}
\usepackage{tikz-cd}

\usepackage{mathtools}

\usepackage{listings}
\usepackage{textcomp} 

\tikzset{
symbol/.style={
draw=none,
every to/.append style={
edge node={node [sloped, allow upside down, auto=false]{$#1$}}}
}
}

\numberwithin{equation}{section}

\theoremstyle{plain}
\newtheorem{theorem}{Theorem}[section]
\newtheorem{lemma}[theorem]{Lemma}
\newtheorem{corollary}[theorem]{Corollary}
\newtheorem{proposition}[theorem]{Proposition}
\newtheorem{conjecture}[theorem]{Conjecture}
\newtheorem{question}[theorem]{Question}

\theoremstyle{definition}
\newtheorem{definition}[theorem]{Definition}

\theoremstyle{remark}
\newtheorem{remark}[theorem]{Remark}
\newtheorem{example}[theorem]{Example}

\newcommand{\overbar}[1]{\mkern 1.5mu\overline{\mkern-1.5mu#1\mkern-1.5mu}\mkern 1.5mu}

\newcommand{\ov}{\overline}

\newcommand{\mr}{\mathrm}
\newcommand{\wt}{\widetilde}

\newcommand{\tr}{\mathrm{tr}}

\newcommand{\diag}{\operatorname{diag}}

\newcommand{\Herm}{\mr{Herm}}
\newcommand{\Sp}{\operatorname{Sp}}

\newcommand{\tensor}{\otimes}

\newcommand{\iso}{\cong}

\newcommand{\lr}{\longrightarrow}

\newcommand{\mbA}{\mathbb{A}}

\newcommand{\mbC}{\mathbb{C}}

\newcommand{\mbR}{\mathbb{R}}

\newcommand{\mbZ}{\mathbb{Z}}

\newcommand{\mcA}{\mathcal{A}}

\newcommand{\mcC}{\mathcal{C}}
\newcommand{\mcD}{\mathcal{D}}

\newcommand{\mcF}{\mathcal{F}}
\newcommand{\mcG}{\mathcal{G}}

\newcommand{\mcP}{\mathcal{P}}
\newcommand{\mcQ}{\mathcal{Q}}

\newcommand{\mcS}{\mathcal{S}}
\newcommand{\mcT}{\mathcal{T}}

\newcommand{\mfA}{\mathfrak{A}}

\newcommand{\mfF}{\mathfrak{F}}

\newcommand{\mfb}{\mathfrak{b}}

\newcommand{\mfg}{\mathfrak{g}}

\newcommand{\mfl}{\mathfrak{l}}
\newcommand{\mfm}{\mathfrak{m}}

\newcommand{\mfs}{\mathfrak{s}}
\newcommand{\mft}{\mathfrak{t}}
\newcommand{\mfu}{\mathfrak{u}}

\newcommand{\Sym}{\mr{Sym}}
\newcommand{\Skew}{\mr{Skew}}

\newcommand{\rs}{\mr{rs}}

\newcommand{\mfgl}{\mfg\mfl}
\newcommand{\mfsl}{\mfs\mfl}
\newcommand{\mfsu}{\mfs\mfu}
\newcommand{\simto}{\overset{\sim}{\to}}
\newcommand{\simlr}{\overset{\sim}{\lr}}

\newcommand{\bbA}{\mathbb A}

\newcommand{\bbC}{\mathbb C}

\newcommand{\bbR}{\mathbb R}

\newcommand{\calA}{\mathcal A}

\newcommand{\calF}{\mathcal F}
\newcommand{\calG}{\mathcal G}

\newcommand{\calQ}{\mathcal Q}

\newcommand{\lie}[1]{ \mathfrak{#1}}

\newcommand{\isomto}{\stackrel{\sim}{\longrightarrow}}

\DeclareMathOperator{\Spec}{Spec}

\DeclareMathOperator{\SL}{SL}
\DeclareMathOperator{\GL}{GL}
\DeclareMathOperator{\SO}{SO}

\DeclareMathOperator{\End}{End}

\DeclareMathOperator{\Aut}{Aut}

\DeclareMathOperator{\Mp}{Mp}

\newcommand{\Orb}{\mathrm{Orb}}

\newcommand{\inv}{\mathrm{inv}}

\newcommand{\mn}{\medskip \noindent}

\title{Jacquet--Rallis transfer for $\GL_2$}

\date{\today}

\author{Andreas Mihatsch}
\address{School of Mathematical Sciences, Zhejiang University, 866 Yuhangtang Rd, Hangzhou, 310058, P. R. China.}
\email{mihatsch@zju.edu.cn}

\author{Siddarth Sankaran}
\address{Department of Mathematics, University of Manitoba, Winnipeg, Manitoba, Canada.}
\email{siddarth.sankaran@umanitoba.ca}

\author{Tonghai Yang}
\address{Department of Mathematics, University of Wisconsin, Madison, WI 53706, USA.}
\email{thyang@math.wisc.edu}

\date{\today}

\begin{document}

\begin{abstract}
We study archimedean smooth transfer for the Jacquet--Rallis relative trace formula comparison,  in particular, identities between orbital integrals on $\GL_n$ and its unitary forms. We work with Lie algebras and a $(\mfg, K)$-module setting. Our main result states that for $n = 2$, meaning $\GL_2$ acting on $\mfgl_3$, every polynomial type Schwartz function has a transfer to the unitary side and vice versa. Our proof relies on the Weil representation and the study of invariant differential operators. It also suggests certain structural properties of the $(\mfg, K)$-modules in question which we formulate as conjectures.
\end{abstract}

\maketitle
\tableofcontents

\section{Introduction}

In this paper, we study archimedean smooth transfer for the Jacquet--Rallis relative trace formula comparison. We work with Lie algebras and a $(\mfg, K)$-module setting. Our main theorem states that for $n = 2$, meaning $\GL_2$ acting on $\mfgl_3$, every polynomial type Schwartz function has a transfer to the unitary side and vice versa. We begin by formulating this result more precisely.

\subsection{Main result}
We first recall the general setting which was introduced by Jacquet--Rallis \cite{JR}. We view $\GL_n$ embedded in $\GL_{n+1}$ via $g\mapsto \left(\begin{smallmatrix} g & \\ & 1\end{smallmatrix}\right)$, which in turn acts by conjugation on $\lie{gl}_{n+1}$. For an $n$-dimensional hermitian $\mbC$-vector space $V$, we similarly consider $U(V)$ acting by conjugation on $\mfu(V\oplus \mbC)$. The categorical quotients of these representations can both be identified with $\mcQ' := \mbA^{2n+1}_\mbR$. Taking $\mbR$-points, this leads to the \emph{matching bijection} for regular semi-simple orbits,
\begin{equation}\label{eq:intro_matching}
[\mfgl_{n+1}(\mbR)_\rs] \ = \ \mcQ'(\mbR)_\rs\ =\ \coprod_{p+q = n} [\mfu(V_{p,q}\oplus \mbC)(\mbR)_\rs],
\end{equation}
where $V_{p,q}$ is a choice of hermitian space with signature $(p,q)$. We write $\mcS(W)$ for the space of Schwartz functions on a real vector space $W$. Given any $Φ\in \mcS(\mfgl_{n+1}(\mbR))$, its $\GL_n(\mbR)$-orbital integral is a smooth function
$$\Orb(-,Φ) \in \mcC^\infty(\mcQ'(\mbR)_\rs).$$
Similarly, the $U(V)(\mbR)$-orbital integral of some $Ψ \in \mcS(\mfu(V\oplus \mbC)(\mbR))$ defines a smooth function
$$\Orb(-,Ψ) \in \mcC^\infty(\mcQ'(\mbR)_\rs)$$
by extending it as zero to the complement of $[\mfu(V\oplus \mbC)(\mbR)_\rs]$ in \eqref{eq:intro_matching}. We say a function $Φ\in \mcS(\lie{gl}_{n+1}(\bbR))$ and a tuple  $(Ψ_{p,q})_{p+q = n}$, where $Ψ_{p,q}\in \mcS(\mfu(V_{p,q}\oplus \mbC)(\mbR))$, are \emph{transfers} if $\Orb(-,Φ) =\sum \Orb(-,Ψ_{p,q})$. The smooth transfer conjecture from \cite{JR}, \cite[\S3.1]{Zhang_GGP} or \cite[Conjecture 3.1]{Xue} states that every $Φ$ should have a transfer $(Ψ_{p,q})_{p+q = n}$ and vice versa:
\begin{conjecture}\label{conj:intro_transfer_all_Schwartz}
There is an equality of spaces of smooth functions on $\mcQ'(\mbR)_\rs$,
$$\Orb\big(-,\mcS(\mfgl_{n+1}(\mbR))\big) = \bigoplus_{p + q = n} \Orb\big(-, \mcS(\mfu(V_{p,q}\oplus \mbC)(\mbR))\big).$$
\end{conjecture}
The strongest known general statement about this conjecture, due to Hang Xue \cite[Theorem 3.3]{Xue}, states that on each side, the set of functions admitting transfers is dense. The analog of Conjecture \ref{conj:intro_transfer_all_Schwartz} in the non-archimedean setting is a theorem due to W. Zhang \cite{Zhang_GGP}. These two (purely local) results play an important role in the proof of the global Gan--Gross--Prasad Conjecture \cite{Zhang_GGP, Xue, BLZZ, BCZ} which is also the application that motivated \cite{JR}.

In this paper, we take a $(\mfg, K)$-module perspective by restricting to polynomial type Schwartz functions as described below. Our results suggest that this is a natural perspective from the representation theory point of view. We also expect that it will prove useful for the study of derivatives of orbital integrals, extending results on ``partial Gaussian test functions'' of Wei Zhang \cite{Zhang_AFL}. This should ultimately expand the arithmetic intersection theory from \cite{Zhang_AFL, MZ, ZZ, LMZ} in the complex multiplication cycle direction.

We endow $\mfgl_{n+1}$ with the quadratic form $Y\mapsto \tr(Y^2)$. The standard choice of Siegel Gaussian for this form is
$$φ(Y) := \exp(-2π\tr(YY^t)) \ \in \ \mcS(\mfgl_{n+1}(\mbR)),$$
leading us to consider the space of \emph{polynomial type} Schwartz functions
$$\mcS'_n := \mbC[\mfgl_{n+1}]\cdot φ.$$
A similar construction on the unitary side defines spaces $\mcT'_{p,q}$ of polynomial type Schwartz functions on $\mfu(V_{p,q}\oplus \mbC)(\mbR)$. We conjecture that transfer is already defined at the level of polynomial Schwartz functions:
\begin{conjecture}[see Conjecture \ref{conj:polynomial_Schwartz}]\label{conj:intro_transfer_poly_Schwartz}
There is an equality of spaces of smooth functions on $\mcQ'(\mbR)_\rs$
$$\Orb\big(-,\mcS'_n\big) = \bigoplus_{p + q = n} \Orb\big(-, \mcT'_{p,q}\big).$$
\end{conjecture}
The main result of this article is the following theorem:
\begin{theorem}\label{thm:intro_main}
Conjecture \ref{conj:intro_transfer_poly_Schwartz} holds when $n \leq 2$.
\end{theorem}
The case $n = 1$ is quick to establish, see \S\ref{s:n_equal_1}. Moreover, the main result from \cite{MSY} shows that every function $Ψ_{n,0}\in \mcT'_{n,0}$ from the positive definite hermitian space has a transfer to $\mcS'_n$ (Theorem \ref{thm:pos_def}). A sign switch argument extends this to $\mcT'_{0,n}$. So Theorem \ref{thm:intro_main} adds the converse direction for signatures $(2,0)$ and $(0,2)$, as well as establishes the $(1,1)$-case.

\subsection{Strategy of proof}

Our approach is based on the Weil representation and the study of invariant differential operators. We begin by explaining this on the $\GL_n$-side. The structure of the problem is to describe $\Orb(-,\mcS'_n)$, which amounts to starting from the reductive dual pair
$$\big(\mr{O}(\mfgl_{n+1}(\mbR),\,\tr(Y^2)),\ \Mp_2\!\big),$$
restricting along the representation $\GL_n(\mbR)\to \mr{O}(\mfgl_{n+1}(\mbR))$, and taking $\GL_n(\mbR)$-coinvariants. This is similar to a seesaw situation, with the difference that $\GL_n(\mbR) \subset \Sp_{2(n+1)^2}(\mbR)$ does not give rise to a reductive dual pair because it does not equal its double centralizer. More precisely, the ring of invariant polynomials
$$\mcA' := \mbC[\mfgl_{n+1}]^{\GL_n}$$
is a polynomial ring in $2n+1$ variables, but these variables do not correspond to tensors defining a linear algebraic group dual to $\GL_n$. We instead pass to the infinitesimal Weil representation
\begin{equation}\label{eq:intro_Weil}
ω:\mfsl_2(\mbC) \lr \End_{\GL_n(\mbR)}\big(\mcS(\mfgl_{n+1}(\mbR))\big).
\end{equation}
The motivation for this is that we can view both $ω(\mfsl_2(\mbC))$ and $\mcA'$ as subsets of the ring $\mbC\{\mfgl_{n+1}\}^{\GL_n}$ of invariant algebraic differential operators on $\mfgl_{n+1}$. The slogan is that we study the ``non-linear seesaw''
\begin{equation}\label{eq:intro_seesaw_GLn}
\xymatrix{O(\mfgl_{n+1}) \ar@{-}[rrd] \ar@{-}[d] & & \mbC\langle ω(\mfsl_2(\mbC)),\,\, \mcA'\rangle \ar@{-}[d]\\
\GL_n \ar@{-}[rru] & & \ \mfsl_2.}
\end{equation}
Let $\mfA'$ denote the upper right corner, i.e. the (non-commutative) subalgebra of $\bbC\{ \lie{gl}_{n+1} \}^{\GL_n}$ generated by $\omega(\lie{sl}_2(\bbC))$ and $\calA'$. Note that $\mbC\{\mfgl_{n+1}\}$ acts on $\mcS'_n$ naturally by differentiation. Since $\calQ' = \Spec(\calA')$, any invariant differential operator on $\lie{gl}_{n+1}$ restricts to a differential operator on $\calA'$, inducing a ring homomorphism
\begin{equation}\label{eq:intro_restrict_diffop}
\mbC\{\mfgl_{n+1}\}^{\GL_n}\ \lr\ \ \mbC\{\mcQ'\}.
\end{equation}
In particular, the space $\Orb(-,\mcS'_n)$ is an $\lie A'$-module. Our first result for $n=2$ exhibits concrete generators:
\begin{theorem}[see Corollary \ref{cor:main_S_2}]\label{thm:intro_GL2_generators} There exist three explicit functions $\Phi_{2,0}, \Phi_{1,1}, \Phi_{0,2} \in \mcS'_2$ such that $\Orb(-,\mcS'_2)$ is generated as $\mfA'$-module by
$$\Orb(-,Φ_{2,0}),\quad \Orb(-,Φ_{1,1}),\quad\text{and}\quad \Orb(-,Φ_{0,2}).$$
\end{theorem}
As previously discussed, the functions $Φ_{2,0}$ and $\Phi_{0,2}$ arise from our prior work \cite{MSY}; for details as well as the construction of $\Phi_{1,1}$, see the discussion preceding \Cref{cor:main_S_2}.

On the unitary side, we start with the quadratic spaces $\big(\mfu(V_{p,q}\oplus \mbC),\ -\tr(X^2)\big)$ and consider analogous commutator diagrams
\begin{equation}\label{eq:intro_seesaw_U}
\xymatrix{O(\mfu(V_{p,q}\oplus \mbC)) \ar@{-}[rrd] \ar@{-}[d] & & \mfA' \ar@{-}[d]\\
U(V_{p,q}) \ar@{-}[rru] & & \ \mfsl_2.}
\end{equation}
 Here, we have written $\mfA'$ again because the subalgebra of $\mbC\{\mfu(V_{p,q}\oplus \mbC)\}$ generated by $\mbC[\mfu(V_{p,q}\oplus \mbC)]^{U(V_{p,q})}$ and the infinitesimal Weil representation on $\mcS(\mfu(V_{p,q}\oplus \mbC)(\mbR))$ is naturally identified with the ring $\mfA'$ from before,  see \Cref{prop:inv_diff_op_isomorphic}. Our second result is the following unitary analog of Theorem \ref{thm:intro_GL2_generators}.
\begin{theorem}[see Corollary \ref{cor:main_T_11}]\label{thm:intro_U_generators}
For every signature $(p,q)$ with $p + q = 2$, $\Orb(-,\mcT'_{p,q})$ is a cyclic $\mfA'$-module generated by the orbital integral $\Orb(-,ψ_{p,q})$ of the Siegel Gaussian $ψ_{p,q}$.
\end{theorem}
Our third result is that the generators from Theorems \ref{thm:intro_GL2_generators} and \ref{thm:intro_U_generators} are mutual transfers.
\begin{theorem}[{\Cref{thm:main_signature_11,thm:pos_def}}]\label{thm:intro_explicit_transfer}
For every $(p,q)$ with $p+q = 2$ and with appropriate normalizations of Haar measures, $Φ_{p,q}$ and $ψ_{p,q}$ are mutual transfers.
\end{theorem}
We have already noted that the ring $\mfA'$ occurs on both the $\GL_n$-side and the unitary side. This phenomenon is compatible with smooth transfer in the sense that if $\Phi$ and $(\Psi_{p,q})_{p+q=n}$ are transfers, then for every $\mcD \in \mfA'$ also $\mcD(Φ)$ and $(\mcD(Ψ_{p,q}))_{p+q=n}$ are transfers; while this follows from a theorem of Xue \cite[Theorem 9.1]{Xue} about the Weil representation of $\Mp_2$, the infinitesimal version here can be shown directly, see \Cref{cor:Weil_rep_transfer}. The conclusion is that Theorems \ref{thm:intro_GL2_generators} and \ref{thm:intro_U_generators} allow us to extend Theorem \ref{thm:intro_explicit_transfer} to the identity
$$\Orb(-,\mcS'_2) = \bigoplus_{p+q = 2}\Orb(-,\mcT'_{p,q}).$$
This is how we prove Theorem \ref{thm:intro_main}.

\subsection{Coinvariant modules}

We conjecture that Theorems \ref{thm:intro_GL2_generators} and \ref{thm:intro_U_generators} hold in general. That is, $\Orb(-,\mcT'_{p,q})$ should always be generated by the orbital integral of the Siegel Gaussian $ψ_{p,q}$ as $\mfA'$-module, and similarly $\Orb(-,\mcS'_n)$ should be generated by the orbital integrals of $n+1$ transferring functions. These two statements are equivalent if Conjecture \ref{conj:intro_transfer_poly_Schwartz} is true.

However, our proofs of Theorems \ref{thm:intro_GL2_generators} and \ref{thm:intro_U_generators} are purely algebraic and do not refer to orbital integrals. We instead consider the coinvariant modules
\begin{equation}\label{eq:intro_def_coinvariants}
\begin{aligned}
\ov{\mcS_n'} & := \big(\mcS_n'\ /\ \langle \mfgl_n\cdot \mcS_n'\rangle \big)_{(O(n),η)}\\[1mm]
\ov{\mcT'_{p,q}} & := \mcT'_{p,q}\ /\ \langle\mfu(V_{p,q})\cdot \mcT'_{p,q}\rangle,
\end{aligned}
\end{equation}
where $η(g) = \mr{sign}(\det(g))$ is the quadratic character that goes into the definition of the $\GL_n$-orbital integral. Then $Φ\mapsto \Orb(-,Φ)$ and $Ψ\mapsto \Orb(-,Ψ)$ factor over $\ov{\mcS'_n}$ resp. $\ov{\mcT'_{p,q}}$. The $\mfA'$-action descends to $\ov{\mcS'_n}$ and $\ov{\mcT'_{p,q}}$, and we show that $\ov{\mcS'_2}$ is generated by the images of $Φ_{2,0},Φ_{1,1}$ and $Φ_{0,2}$, while $\ov{\mcT'_{p,q}}$ is generated by the image of $ψ_{p,q}$ $(p+q = 2)$.
\begin{conjecture}\label{conj:intro_coinvariants}
(1) For every $n\geq 1$, the coinvariants $\ov{\mcS'_n}$ are generated by $n+1$ elements as $\mfA'$-module.

\mn (2) For every signature $(p,q)$, the coinvariants $\ov{\mcT'_{p,q}}$ are a cyclic $\mfA'$-module generated by the image of the Siegel Gaussian $ψ_{p,q}$.
\end{conjecture}
It would be interesting to find natural candidates $Φ_{p,q}\in \mcS'_n$ for the generators of $\ov{\mcS'_n}$ resp. transfers of the $ψ_{p,q}\in \mcT'_{p,q}$. For example, any such $Φ_{p,q}$ should be $(O(n),η)$-invariant and of the same biweight as $ψ_{p,q}$ for the two Weil representations from \S\ref{sec:Weil rep transfer} (cf. Conjecture \ref{conj:cyclic_coinvariants}).

Another interesting question is whether it could be true that the natural quotient maps
$$\ov{\mcS'_n}\ {\relbar\joinrel\twoheadrightarrow}\ \Orb(-,\mcS'_n),\qquad \ov{\mcT'_{p,q}}\ {\relbar\joinrel\twoheadrightarrow}\ \Orb(-,\mcT'_{p,q})$$
are in fact isomorphisms. This would be in line with the expected density principle for group orbital integrals from \cite[\S6.2.2]{CZ}.

\subsection{Transfer for $n = 2$}

The proofs of Theorems \ref{thm:intro_GL2_generators} and \ref{thm:intro_U_generators} are straightforward once suitable elements of $\mfA'$ and $\mfA'$ have been written down. However, we found it less immediate to obtain the transfer identity $\Orb(-,Φ_{1,1}) = \Orb(-,ψ_{1,1})$ in Theorem \ref{thm:intro_explicit_transfer}. Its two sides are given by concrete triple (resp. double) integrals parametrized by an open subset $\mcQ(\mbR)_{\rs, (1,1)}$ of $\mbR^3$ (see Part III for notation), but it seems hard to identify them directly. Our proof ultimately combines the following three techniques:

\mn (1) There is a certain hyperplane $\{β = 0\}\subset \mcQ(\mbR)_{\rs,(1,1)}$ over which we show $\Orb(-,Φ_{1,1}) = \Orb(-,ψ_{1,1})$ by standard integral substitutions.

\mn (2) Both $Φ_{1,1}$ and $ψ_{1,1}$ are annihilated by one of the weight operators from \S\ref{sec:Weil rep transfer}. It follows that $\Orb(-,Φ_{1,1})$ and $\Orb(-,ψ_{1,1})$ are solutions to a certain differential equation on $\mcQ(\mbR)_{\rs,(1,1)}$. We prove that an analytic solution to this equation is uniquely determined by its restriction to $\{β = 0\}$.

\mn (3) We finally show that $\Orb(-,Φ_{1,1})$ and $\Orb(-,ψ_{1,1})$ are sufficiently analytic. The verification for $ψ_{1,1}$ is elementary, while the one for $Φ_{1,1}$ is done with a short computer program due to the number of terms involved.

\mn We are wondering if it is true that for every $Φ \in \mcS'_n$ (resp.\ $Ψ\in \mcT'_{p,q}$) the orbital integral $\Orb(-,Φ)$ (resp.\ $\Orb(-,Ψ)$) is a real-analytic function on $\mcQ(\mbR)_\rs$. This is easily seen to be true for $Ψ\in \mcT'_{n,0}$ or $Ψ\in \mcT'_{0,n}$ (compare the proof of Theorem \ref{thm:pos_def}). But already for $\Orb(-,\mcT'_{1,1})$ or $\Orb(-,\mcS'_2)$ we only prove a partial result that suffices to apply (2). It would be desirable to improve this step with a structural argument.

We are also wondering if Conjecture \ref{conj:intro_transfer_all_Schwartz} for $n = 2$ can be proved from Theorems \ref{thm:intro_GL2_generators}, \ref{thm:intro_U_generators} and \ref{thm:intro_explicit_transfer}. Polynomial type Schwartz functions are dense in all Schwartz functions, and the information of the pairs $Φ_{p,q} \leftrightarrow ψ_{p,q}$ and the $\mfA'$-action might be sufficient for a suitable approximation argument.

\subsection{Structure of the article}

In Part I, we introduce the general setting for \eqref{eq:intro_seesaw_GLn} and \eqref{eq:intro_seesaw_U}. We note that the notation in the body of the paper will be slightly more refined than here in the introduction; it will be fixed in \S\ref{ss:reduction_trace_0}. Part II is devoted to the structure of $\Orb(-,\mcS'_2)$ and $\Orb(-,\mcT'_{1,1})$, and in particular, the proofs of Theorems \ref{thm:intro_GL2_generators} and \ref{thm:intro_U_generators}. Part III is more analytic and establishes Theorem \ref{thm:intro_explicit_transfer}.

\subsection{Acknowledgements}

        We thank Weixiao Lu and Binyong Sun for interesting conversations on the topic of this article. We thank the American Institute of Mathematics, the Max Planck Institute for Mathematics in Bonn, and the Morningside Center of Mathematics for their hospitality. Parts of this work were done during visits at these institutions.

SS is supported by a Discovery Grant from the National Science and Engineering Research Council (NSERC) of Canada. TY is partially supported by UW-Madison's Kellet Mid-Career award and a National Science Foundation grant (MS-2501617).

\part{Representation-theoretic setup}
\section{Differentiation and orbital integrals}

\subsection{Differential operators}

Let $R$ be a finite type $\mbC$-algebra. The $R$-module of differential operators of order $\leq n$ of $R$ is defined inductively by $D_0(R) = R$ and
$$D_n(R) := \{\mcD\in \End_\text{$\mbC$-vect}(R)\ \mid\ [\mcD, f]\in D_{n-1}(R)\text{ for all $f\in R$}\}.$$
Let $D(R) = \bigcup_{n\geq 0} D_n(R)$ be the ring of all differential operators. If $R = \mbC[x_1,\ldots,x_n]$ is a polynomial ring, then this is the Weyl algebra
\begin{equation}\label{eq:notation_Weyl}
\mbC\{x_1,\ldots,x_n\} := \mbC[x_1,\ldots,x_n, \partial/\partial x_1,\ldots, \partial/\partial x_n].
\end{equation}
\begin{lemma}\label{lem:restrict_diff_op}
Let $R_0\subseteq R$ be a $\mbC$-subalgebra and let $\mcD\in D_n(R)$ satisfy $\mcD(R_0)\subseteq R_0$. Then $\mcD\vert_{R_0}$ lies in $D_n(R_0)$.
\end{lemma}
\begin{proof}
For every $f_0\in R_0$, we have $[\mcD,f_0](R_0)\subseteq R_0$. By induction on $n$, we find that $[\mcD,f_0]\vert_{R_0} \in D_{n-1}(R_0)$. This commutator equals $[\mcD\vert_{R_0}, f_0]$ which shows that $\mcD\vert_{R_0}\in D_n(R_0)$ as claimed.
\end{proof}

\subsection{Invariant differential operators}

Let $G$ be a connected reductive group over $\mbR$ and let $V$ be a finite dimensional algebraic representation of $G$. Let $\mbC[V]$ be the $\mbC$-algebra of polynomial functions on $V$ and let $\mbC\{V\} = D(\mbC[V])$ be the Weyl algebra. On the one hand, both $\mbC[V]$ and $\mbC\{V\}$ can be viewed as real vector spaces and are naturally locally finite algebraic representations of $G$. We denote by $\mcA := \mbC[V]^G$ and $\mbC\{V\}^G$ the subspaces of $G$-invariant elements. By Lemma \ref{lem:restrict_diff_op}, there is a restriction map
\begin{equation}\label{eq:restriction_diff_op_map}
\begin{aligned}
\mbC\{V\}^G & \,\lr\, D(\mcA)\\
\mcD\, & \,\longmapsto\, \mcD\vert_\mcA.
\end{aligned}
\end{equation}
On the other hand, we also obtain a representation of the Lie group $G(\mbR)$ on the real vector space $V(\mbR)$. Then $G(\mbR)$ acts on $\mcC^\infty(V(\mbR))$ by translation, $(g\cdot f)(v) = f(g^{-1}v)$, and on differential operators by conjugation, $(g\cdot \mcD) = g\circ \mcD \circ g^{-1}$, while $\mbC\{V\}$ acts on smooth functions by differentiation. These actions are compatible via
$$(g\cdot \mcD)(g\cdot f) = g\cdot (\mcD(f)).$$
Let $\mcS(V(\mbR))\subset \mcC^\infty(V(\mbR))$ be the subspace of Schwartz functions. It is stable under the $\mbC\{V\}$-action and an element $\mcD \in \mbC\{V\}$ is uniquely determined by its restriction to $\mcS(V)$.
\begin{lemma}\label{lem:invariant_diff_ops}
An operator $\mcD\in \mbC\{V\}$ is $G$-invariant if and only if it is $G(\mbR)$-invariant as operator on $\mcS(V(\mbR))$.
\end{lemma}
\begin{proof}
This is clear because $G(\mbR)$ is Zariski dense in $G$ under our assumption on $G$ being connected.
\end{proof}

\subsection{Orbital integrals}
\label{ss:orb_ints_abstractly}
 Recall that a $G(\mbC)$ orbit in $V(\mbC)$ is said to be regular semi-simple if it is Zariski closed and if its stabilizer has minimal possible dimension. Also recall that the regular semi-simple orbits form an open subvariety $V_\rs \subseteq V$ defined over $\mbR$. We assume from now on that $V_\rs$ is non-empty and that generic stabilizers are trivial. Let
$$π:V\lr \mcQ := \Spec(\mcA)$$
denote the categorical quotient. Let $\mcQ_\rs\subseteq \mcQ$ be the image of $V_\rs$. Recall that $\mcQ_\rs$ is open and that $V_\rs = π^{-1}(\mcQ_\rs)$. Fix a left Haar measure on $G(\mbR)$ and a continuous finite order character $η:G(\mbR)\to \mbC^\times$. Let $\wt{η}:V(\mbR)_\rs\to \mbC^\times$ be a locally constant function that satisfies $\wt{η}(h^{-1}x) = η(h)\wt{η}(x)$ for all $h\in G(\mbR)$, the so-called \emph{transfer factor}. Then, for every Schwartz function $Φ\in \mcS(V(\mbR))$ and every $x\in V(\mbR)_\rs$, there is an absolutely convergent orbital integral
$$\Orb(x, Φ) := \wt{η}(x)\int_{G(\mbR)} η(g)Φ(g^{-1}\cdot x)\,dg$$
which only depends on the orbit $G(\mbR)\cdot x$. The image of $V_\rs(\mbR)\to \mcQ_\rs(\mbR)$ is a union of connected components, so we can use extension by zero to view it as a smooth function
$$\Orb(-,Φ):\mcQ_\rs(\mbR)\lr \mbC.$$
Elements of $D(\mcA)$ act on $\mcC^\infty(\mcQ_\rs(\mbR))$ by differentiation.
\begin{proposition}\label{prop:diff_op_orb_int_generic}
Let $\mcD\in \mbC\{V\}^G$ be an invariant differential operator. Then, for every $Φ\in \mcS(V(\mbR))$,
$$\Orb(-, \mcD(Φ)) = (\mcD\vert_\mcA)(\Orb(-, Φ)).$$
\end{proposition}
\begin{proof}
The restriction $\mcD\vert_\mcA$ is defined by Lemma \ref{lem:restrict_diff_op}. By our assumption on stabilizers, the morphism $π^{-1}(\mcQ_\rs(\mbR))\to \mcQ_\rs(\mbR)$ is a $G(\mbR)$-bundle over its image. Working locally on open subsets $U$ of the image, we can assume that there exists a section $U\to V(\mbR)_\rs$ and hence a trivialization
\begin{equation}\label{eq:trivialize_pi}
G(\mbR)\times U \simlr π^{-1}(U).
\end{equation}
Let $X_1,\ldots,X_r \in \mfg$ be a basis for the Lie algebra of $G(\mbR)$ viewed as left-invariant vector fields on $G(\mbR)$. Working locally on $U$, we may further assume that there exist smooth coordinate functions $x_1,\ldots,x_m$ on $U$. Let $\partial_1,\ldots, \partial_m$ denote the corresponding partial derivatives. Then $\mcD$ being $G(\mbR)$-invariant implies that there exist smooth functions $f_{α,β}$ on $U$, where $α\in \mbZ_{\geq 0}^r$ and $β\in \mbZ_{\geq 0}^m$ are multi-indices, such that, with respect to \eqref{eq:trivialize_pi},
$$\mcD\vert_{π^{-1}(U)} = \sum_{α, β} π^*(f_{α,β}) X^α\partial^β.$$
We obtain that
$$\mcD\vert_{\mcC^\infty(U)} = \sum_{β} f_{(0,\ldots,0),\, β} \partial^β.$$
This restriction agrees with $\mcD\vert_{\mcA}$ because every element of $\mbC[V]$ is uniquely determined by the function it defines on $π^{-1}(U)$.

Let $X\in \mfg$ be an element viewed as left-invariant vector field on $G(\mbR)$, let $Φ\in \mcS(V(\mbR))$ be a Schwartz function, and let $x\in V(\mbR)_\rs$ be a regular semi-simple element. Let $Ψ(g) := η(g)Φ(g^{-1}\cdot x)$ be the Schwartz function obtained from $η$ and by restricting $Φ$ to the orbit $G(\mbR)\cdot x$. Since $η$ is locally constant as function on $G(\mbR)$, we find
\begin{equation}\label{eq:orb_int_vanishing_derivative}
\begin{aligned}
\Orb(x, X(Φ)) &\ =\ \wt{η}(x)\int_{G(\mbR)} η(g) (X(Φ))(g^{-1}\cdot x)\, dg\\[1mm]
&\ =\ \wt{η}(x)\int_{G(\mbR)} (X(Ψ))(g)\,dg\\
&\ =\ 0
\end{aligned}
\end{equation}
because the integral of the derivative of a Schwartz function is zero. Hence, as functions in $u\in U$ and using the coordinates from \eqref{eq:trivialize_pi}, we obtain
\begin{equation}\label{eq:orb_int_restriction_diffop}
\begin{aligned}
\Orb(u, \mcD(Φ)) &\ =\ \wt{η}(1,u) \int_{G(\mbR)} η(g) \mcD(Φ)(g,u)\,dg\\[1mm]
&\ =\ \wt{η}(1,u) \sum_{β} f_{(0,\ldots,0),\, β}(u) \Big(\partial^β \int_{G(\mbR)} η(g)Φ(g,-)\,dg\Big)(u).
\end{aligned}
\end{equation}
The transfer factor is locally constant by assumption, hence commutes with the $\partial^β$, so \eqref{eq:orb_int_restriction_diffop} agrees with $\mcD\vert_{\mcA}(\Orb(-, Φ))$ as claimed.
\end{proof}

\section{The Weil representation}

In this section, we recall the infinitesimal Weil representation for $\mfsl_2$ and fix our conventions for Fock models. A general reference is \cite[Appendix A]{FunkeMillson}, from which our presentation can be obtained by specialization (see Remark \ref{rmk:compatib_FM}). References for $\Mp_2$ are \cite[(1.6), (1.7)]{Kudla_Annals} or \cite[(11.1)]{Zhang_AFL}.

\subsection{The metaplectic group}

The fundamental group of $\SL_2(\mbR)$ is cyclic of infinite order, so there exists a unique connected double cover
\begin{equation}\label{eq:metaplictic_cover}
1\lr μ_2 \lr \Mp_2 \lr \SL_2(\mbR)\lr 1
\end{equation}
called the metaplectic group. Recall that $\Mp_2$ is a Lie group but not an algebraic group; it has no finite-dimensional faithful representations. An explicit description of $\Mp_2$ is given by the Rao cocycle \cite[p. 363]{Kudla_splitting}. We identify $\mbR^\times/\mbR^{\times,2}$ with $\{\pm 1\}$ and define $χ:\SL_2(\mbR)\lr \{\pm 1\}$ by
$$χ\left(\begin{pmatrix}
a & b \\ c & d
\end{pmatrix}\right)\ =\ \begin{cases} \text{$c\ $ mod $\mbR^{\times,2}$} & \text{if $c\neq 0$}\\
\text{$d\ $ mod $\mbR^{\times, 2}$} & \text{if $c = 0$.}
\end{cases}$$
Let $\langle\ ,\ \rangle$ denote the Hilbert symbol on $\mbR^\times/\mbR^{\times,2}$, meaning $\langle ε_1, ε_2\rangle = -1$ if and only if  $\epsilon_1$ and $\epsilon_2$ are negative. The Rao cocycle is given by
\begin{equation}\label{eq:Rao_cocycle}
c(g_1, g_2) := \langle χ(g_1g_2), -χ(g_1)χ(g_2)\rangle\cdot \langle χ(g_1), χ(g_2)\rangle.
\end{equation}
There exists a set-theoretic section $\SL_2(\mbR) \to \Mp_2$ that induces an identification $\SL_2(\mbR)\times μ_2\simto \Mp_2$, for which multiplication is given by
\begin{equation}\label{eq:mult_metaplectic}
[g_1,\, ε_1]\cdot [g_2,\, ε_2] = [g_1g_2, \ ε_1ε_2c(g_1,g_2)].
\end{equation}
We use the standard notation
$$\begin{aligned}
& A = \{m(a) \mid a\in \mbR_{>0}\}, \quad m(a) = \begin{pmatrix}
a & \\ & 1/a
\end{pmatrix},\\[1mm]
& N = \{n(b) \mid b\in \mbR\}, \mkern 31mu \quad n(b) = \begin{pmatrix}
1 & b \\ & \mkern 9mu 1 \mkern 9mu
\end{pmatrix},
\end{aligned}$$
and  $B^+ = AN$. We  also consider the lower triangular group
$$B^- = \left.\left\{\begin{pmatrix}a & \\ b & 1/a\end{pmatrix} \in \SL_2(\mbR)\ \right \vert\ a > 0\right\}.$$
It is clear from \eqref{eq:Rao_cocycle} and \eqref{eq:mult_metaplectic} that the section $g\mapsto [g,1]$ restricts to continuous group homomorphisms
$$B^{\pm}\lr \Mp_2$$
describing the unique splittings of the metaplectic cover over the simply connected subgroups $B^+$ and $B^-$. We write $[B^\pm, 1] \subset \Mp_2$ for their images.

The Lie algebra of $\Mp_2$ is naturally identified with $\mfsl_2$ via the covering map to $\SL_2(\mbR)$. The Lie algebras $\mfb^{\pm}$ of $B^{\pm}$ satisfy $\mfb^+ + \mfb^- = \mfsl_2$. Since $\Mp_2$ is connected, this implies that $[B^+, 1]$ and $[B^-,1]$ generate $\Mp_2$. Let
$$w = \begin{pmatrix} & 1\\ -1 & \end{pmatrix}$$
be the Weyl element in $\SL_2(\mbR)$. Then $[w,1]^{-1} = [w^{-1},1]$ and
$$[w^{-1},1]\cdot [B^+,1]\cdot [w,1] = [B^-,1].$$
We clearly also have $\mr{ad}([w,1])\cdot \mfb^+ = \mfb^-$. The following lemma is an immediate consequence.
\begin{lemma}\label{lem:gen_Mp2}
The metaplectic group is generated by either of $[w,1]$ and $[B^+,1]$ or $[w,1]$ and $[B^-,1]$.
\end{lemma}

\subsection{The Weil representation}
\label{ss:Weil_rep}
Consider a real quadratic space $(V,Q)$ of dimension $n$ and signature $(p,q)$. Denote by $(\cdot, \cdot)$ the corresponding symmetric bilinear form; it is determined by $2 Q(v) = (v,v)$. Recall that $\mcS(V)$ denotes the space of Schwartz functions on $V$. Fix the additive character $ψ(t) = e^{2πit}$ on $\mbR$. Recall that the Fourier transform for $V$ is the $\mr{O}(V)$-linear map
$$\mfF:\mcS(V)\lr \mcS(V),\quad (\mfF φ)(x) = \int_V φ(y)ψ\big((x,y)\big)\,dy$$
where the measure is normalized such that $(\mfF \mfF φ)(x) = φ(-x)$. The \emph{Weil representation} of $\Mp_2$ on $\mcS(V)$ with respect to $ψ$ is the representation
$$\Omega:\Mp_2 \lr \End_{\mr{O}(V)}(\mcS(V))$$
which is given by the following formulas. Let $γ = \exp((p-q)πi/4)$ be the Weil index of $V$. Then
\begin{equation}\label{eq:Weil_group}
\begin{aligned}
\big(\Omega[m(a),1]\cdot φ\big)(x) & \ =\ |a|^{n/2}φ(ax)\quad \mkern 32mu (a > 0),\\[1mm]
\big(\Omega[n(b),1]\cdot φ\big)(x) & \ =\ ψ(bQ(x))φ(x)\quad \ \ (b\in \mbR),\\[1mm]
\big(\Omega[w,1]\cdot φ\big)\mkern 24mu & \ =\ γ\cdot \mfF φ.
\end{aligned}
\end{equation}
Note that $[w,1]^4 = [1, -1]$ which shows that $[1, -1]$ acts by $γ^4 = (-1)^{p-q}$. Thus, if $n$ is even, then $\Omega$ factors through $\SL_2(\mbR)$. If $n$ is odd, then $μ_2$ acts non-trivially.

\subsection{The infinitesimal Weil representation}
\label{ss:infinitesimal_Weil_rep}

The \emph{infinitesimal Weil representation} on $\mcS(V)$ with respect to $ψ$ is the Lie algebra representation
$$ω:\mfsl_2(\mbC) \lr \End_{\mr{O}(V)}(\mcS(V))$$
which is obtained by differentiation from \eqref{eq:Weil_group}. In order to give an explicit description, we choose standard coordinate functions $x_1,\ldots,x_n$ on $V$. That is, we assume that $Q$ is given by
$$Q = \frac{1}{2}\Big(\sum_{α = 1}^p x_α^2 - \sum_{μ = p+1}^n x_μ^2\Big).$$
In terms of such coordinates, $ω$ is defined by the differential operators
\begin{equation}\label{eq:Weil_Lie_algebra}
\begin{aligned}
ω\begin{pmatrix}
1 & \\ & -1
\end{pmatrix} &\ =\ \frac{n}{2} + \sum_{i = 1}^n x_i \frac{\partial}{\partial x_i} \\[1mm]
ω\begin{pmatrix}
& 1 \\ \ &
\end{pmatrix} &\ =\ 2πiQ\\[1mm]
ω\begin{pmatrix}
&\  \\ 1 &
\end{pmatrix} &\ =\ \frac{i}{4π}\Big(\sum_{α = 1}^p \frac{\partial^2}{\partial x_α^2} - \sum_{μ = p+1}^n \frac{\partial^2}{\partial x_μ^2}\Big).
\end{aligned}
\end{equation}
Indeed, the first two identities are obtained directly from the formulas for the action of $B$ in \eqref{eq:Weil_group}. The third identity is obtained from $\left(\begin{smallmatrix} &\ \\1 &\end{smallmatrix}\right) = \mr{ad}([w,1])\left(\left(\begin{smallmatrix} \ & -1 \\ &\end{smallmatrix}\right)\right)$ and hence
$$ω\left(\begin{smallmatrix} &\ \\1 &\end{smallmatrix}\right) = \mfF \circ ω\left(\begin{smallmatrix} \ & -1 \\ &\end{smallmatrix}\right)\circ \mfF^{-1}.$$

\subsection{Algebraicity of $ω$}
\label{ss:infinitesimal_Weil_rep_algebraic}

We observe that the infinitesimal Weil representation is a purely algebraic construction. Let $k$ be a field of characteristic different from $2$, let $V$ be an $n$--dimensional $k$-vector space, and let $Q\in k[V]$ be a non-degenerate quadratic form. That is, $Q$ is an element of $\mr{Sym}^2(V^*)$ of the form $Q(v) = B(v,v)/2$ for a non-degenerate symmetric bilinear form $B$. In particular, $B$ is an isomorphism $B:V\simto V^*$ so we also have a dual element $Q^* = B^{-1}(Q) \in \mr{Sym}^2(V)$. Let $k \{ V\}$ denote the Weyl algebra, and view $V \subset k\{ V\}$ by associating a vector $v$ to the partial derivative $\partial_v$.  Given any $λ\in k^\times$, we define a representation $\mfsl_2 \to k\{V\}$ by the formulas
\begin{equation}\label{eq:Weil_Lie_algebra_algebraic}
\begin{aligned}
ω\left(\begin{smallmatrix}
1 & \\ & -1
\end{smallmatrix}\right) &\ =\ \frac{n}{2} + \text{(degree operator)}\\[1mm]
ω\left(\begin{smallmatrix}
& 1 \\ \ &
\end{smallmatrix}\right) &\ =\ λ Q\\[1mm]
ω\left(\begin{smallmatrix}
&\  \\ 1 &
\end{smallmatrix}\right) &\ =\ -Q^*/λ.
\end{aligned}
\end{equation}
When specialized to $k = \mbR$, $λ = 2πi$, and after choosing a standard basis for $(V, Q)$, this recovers \eqref{eq:Weil_Lie_algebra}. The point of \eqref{eq:Weil_Lie_algebra_algebraic} is that these formulas are obviously invariant under field extension and preserved under isometries $(V, Q_V)\simto (W, Q_W)$.

\subsection{The Fock model} \label{ss:Fock model}

Recall that $\mbC[V] = \mbC[x_1,\ldots,x_n]$ is the ring of polynomial functions on $V$ and that
$$\mbC\{x_1,\ldots,x_n\} := \mbC[x_1,\ldots,x_n, \partial/\partial x_1,\ldots, \partial/\partial x_n]$$
denotes the Weyl algebra. We denote by $φ_S$ the \emph{Siegel Gaussian} with respect to our chosen standard basis,
\begin{equation} \label{eq:Siegel Gaussian general}
\varphi_S = e^{- \pi(x_1^2\,+\, \ldots\, +\, x_n^2)},
\end{equation}
and call $\mcP(V) := \mbC[V]\cdot φ_S$ the space of \emph{polynomial Schwartz functions}. It is clearly stable under the differentiation action of $\mbC\{V\}$. In particular, it is preserved by the infinitesimal Weil representation.

The \emph{lowering}, \emph{raising}, and \emph{weight operators} in $\mfsl_2(\mbC)$ are the three elements
\begin{equation}\label{eq:LRH}
L := \frac12 \begin{pmatrix}  1 & -i \\ -i  & -1 \end{pmatrix}, \qquad R := \frac12 \begin{pmatrix} 1 &i \\ i & -1 \end{pmatrix} , \quad\text{and}\quad H := \begin{pmatrix} 0 & -i \\ i & 0 \end{pmatrix}.
\end{equation}
They form a basis of $\mfsl_2(\mbC)$ and are related by $H = [R,L]$. The space of polynomial Schwartz functions is a direct sum of its $H$-eigenspaces, with eigenvalues in $(p-q)/2 + \mbZ$. These eigenvalues are called weights. However, the weight decomposition of a polynomial Schwartz function $f\cdot φ_S$ is not immediate to read off from $f$. For this reason, we next recall the Fock model for $\mcP(V)$ in which weights agree with the monomial degree up to the shift $(p-q)/2$.
\begin{definition}[Fock model]\label{def:Fock}
Let $z_1,\ldots,z_n$ be a second set of coordinates. The \emph{Fock model} for $\mcP(V)$ with respect to these coordinates is the vector space
$$\mcF(V) := \mbC[z_1,\ldots,z_n].$$
It is a cyclic module for the action of $\mbC\{z_1,\ldots,z_n\}$. Consider the ring isomorphism
$$ι:\mbC\{x_1,\ldots,x_n\} \simlr \mbC\{z_1,\ldots,z_n\}$$
given by the substitution rules
\begin{equation}\label{eq:Fock_model_intertwiner}
ι\left(x_i - \frac{1}{2π}\frac{\partial}{\partial x_i}\right)\, =\, \frac{1}{2π} z_i, \ \qquad\ ι\left(x_i + \frac{1}{2π}\frac{\partial}{\partial x_i}\right)\, =\, 2 \frac{\partial}{\partial z_i}.
\end{equation}
The intertwining map between polynomial Schwartz functions and Fock model is the unique $ι$-linear isomorphism
$$ι:\mcP(V)\simlr \mcF(V)$$
that sends $φ_S$ to $1$. In other words, $ι(f\cdot φ_S) = ι(f)(1)$ and $ι^{-1}(h) = ι^{-1}(h)(φ_S)$ for all $f\in \mbC[V]$ and $h\in \mcF(V)$.
\end{definition}
One can now apply $ι$ to the three operators in \eqref{eq:Weil_Lie_algebra} to express $ω(L)$, $ω(R)$ and $ω(H)$ as elements of $\mbC\{z_1,\ldots,z_n\}$. This results in
\begin{equation}\label{eq:LRH_Fock}
\begin{split}
\omega(L) &= 2 \pi \sum_{\alpha = 1}^{p} \frac{\partial^2}{\partial z_{\alpha}^2} - \frac{1}{8 \pi} \sum_{\mu = p+1}^{n} z_{\mu}^2 \\[1mm]
\omega(R) &=  -\frac1{8 \pi} \sum_{\alpha = 1}^{p}  z_{\alpha}^2  + 2 \pi \sum_{\mu = p+1}^{n} \frac{\partial^2}{\partial z_{\mu}^2}\\[1mm]
\omega(H) &= \sum_{\alpha = 1}^p z_{\alpha} \frac{\partial}{\partial z_{\alpha}} - \sum_{\mu = p+1}^n z_{\mu} \frac{\partial}{\partial z_\mu} +  \frac{p-q}{2}.
\end{split}
\end{equation}
In particular, we see that monomials in $\mcF(V)$ are eigenvectors for $ω(H)$, and that the weight of $z_1^{m_1}\cdots z_n^{m_n}$ is given by
\begin{equation}\label{eq:weight_monomial}
\frac{p-q}{2} + \sum_{α = 1}^p m_α - \sum_{μ = p+1}^n m_μ.
\end{equation}
We end this section with a simple identity relating $L, R, H$ and $Q$. Observe for this (see \eqref{eq:LRH}) that
$$R - H - L = 2i\begin{pmatrix}
& 1\\ &
\end{pmatrix}.$$
By the second identity in \eqref{eq:Weil_Lie_algebra}, we obtain
\begin{equation}\label{eq:LRHQ}
ω(R) - ω(H) - ω(L) = -4πQ
\end{equation}
as elements of $\mbC\{V\}$. Moreover, $R$, $H$ and $L$ are homogeneous of weights $2$, $0$ and $-2$, respectively. So we see that these operators are (up to sign) precisely the homogeneous components of $-4πQ$. We see from \eqref{eq:Fock_model_intertwiner} that for every pair $(x,z) = (x_i, z_i)$,
$$\begin{aligned}
-2π x^2 & = -2π (z/4π + \partial_z)^2\\
& = -z^2/8π - (1/2 + z\partial_z) - 2π\partial_z^2.
\end{aligned}$$
Taking homogeneous components in the sense of \eqref{eq:weight_monomial} and taking the sum according to $-4πQ = -2π(x_1^2 + \ldots +x_p^2 - x_{p+1}^2 -\ldots-x_n^2)$ recovers \eqref{eq:LRH_Fock} from \eqref{eq:LRHQ}.

\begin{remark}\label{rmk:compatib_FM}
Let $z_α'$ and $z'_μ$ denote the Fock model coordinates in \cite[Appendix A]{FunkeMillson}. Our conventions are obtained from those there by specialization to $\mfsl_2$ and $λ = 2πi$, and by substituting
$$z'_α = iz_α,\quad z'_μ = -iz_μ,\quad 1\leq α \leq p,\ p+1\leq μ \leq n.$$
\end{remark}

\section{The Jacquet--Rallis setting}

\subsection{The $\GL_n$-side}

Fix an integer $n\geq 1$. The group $\GL_n$ acts on $\mfgl_{n+1}$ by conjugation via $g\mapsto \diag(g,1)$. The regular semi-simple locus $\mfgl_{n+1,\rs}$ is non-empty and its elements have trivial stabilizer, so this setting fits into the framework from \S \ref{ss:orb_ints_abstractly}. We usually use the block matrix notation
\begin{equation}\label{eq:standard_notation_GL}
\begin{pmatrix}
y & v\\ w & d
\end{pmatrix},\quad y\in \mfgl_n,\ \ v\in \mbA^n,\ \ w\in (\mbA^n)^t,\ \ d\in \mbA^1
\end{equation}
to denote its elements. The ring of invariants $\mcA' = \mbC[\mfgl_{n+1}]^{\GL_n}$ is a polynomial ring in $2n+1$ variables generated by
\begin{equation}\label{eq:invariant_functions_GL}
\tr(y),\ \ \tr(\wedge^2 y),\ \ \ldots,\ \ \tr(\wedge^n y),\ \ wv,\ \ wyv,\ \ \ldots,\ \ wy^{n-1}v,\ \ d.
\end{equation}
The categorical quotient $\mcQ' = \Spec(\mcA')$ is hence $\mbA^{2n+1}$ via \eqref{eq:invariant_functions_GL}, and $\mcQ'(\mbR) = \mbR^{2n+1}$. Let
\begin{equation}\label{eq:quot_map_gl}
π:\mfgl_{n+1} \lr \mcQ'
\end{equation}
be the quotient morphism. An element $Y = (y, v, w, d) \in \mfgl_{n+1}(\mbR)$ as in \eqref{eq:standard_notation_GL} is regular semi-simple if and only if
$$\mbR[y]\cdot v = \mbR^n\quad \text{and}\quad w\cdot \mbR[y] = \mbR_n.$$
Here and later, we often write $\mbR_n := (\mbR^n)^t$ for the space of row vectors. Let $$η:= \mathrm{sgn} \circ \det \colon \GL_n(\mbR)\to \{\pm 1\}$$ be the non-trivial continuous quadratic character. We use the transfer factor already used in \cite{Zhang_GGP, MSY, Xue} which is defined as
\begin{equation}\label{eq:transfer_factor}
\begin{aligned}
\wt{η}:\mfgl_{n+1,\rs}(\mbR) & \lr \{\pm 1\}\\
(y, v, w, d) & \longmapsto \mathrm{sgn}\big(\det(v, yv, \ldots, y^{n-1}v)\big).
\end{aligned}
\end{equation}

 Fix a Haar measure $dg$ on $\GL_n(\bbR)$.
For a Schwartz function $Φ\in \mcS(\mfgl_{n+1}(\mbR))$ and a regular semi-simple element $Y\in \mfgl_{n+1,\rs}(\mbR)$, we then have the orbital integral
$$\Orb(Y, Φ) = \wt{η}(Y)\int_{\GL_n(\mbR)} Φ(g^{-1}\cdot Y) η(g)\, dg.$$
It only depends on the orbit $\GL_n(\mbR)\cdot Y$, and we have $π(\mfgl_{n+1,\rs}(\mbR)) = \mcQ'_\rs(\mbR)$, so we may view it as a smooth function
$$\Orb(-,Φ):\mcQ'_\rs(\mbR) \lr \mbC.$$

\subsection{The unitary side}
Let $V$ be an $n$-dimensional hermitian $\mbC$-vector space. Denote by $V\oplus \mbC$ its orthogonal direct sum with the standard one-dimensional hermitian space $(\mbC, \overbar x y)$. Let $\mfu(V)$ and $\mfu(V\oplus \mbC)$ denote the Lie algebras of $U(V)$ and $U(V\oplus \mbC)$. Then $U(V)$ acts on $\mfu(V\oplus \mbC)$ by conjugation, which is another incarnation of the setting in \S\ref{ss:orb_ints_abstractly}. We usually use the block matrix notation
\begin{equation}\label{eq:standard_notation_U}
\begin{pmatrix}
x & u\\ -u^* & e
\end{pmatrix},\quad x\in \mfu(V),\ \ u \in V,\ \ u^* = (u,-),\ \ e\in i\cdot \mbA^1_\mbR
\end{equation}
to denote elements of $\mfu(V\oplus \mbC)$. An element $(x,u,e)\in \mfu(V\oplus \bbC)(\mbR)$ is regular semi-simple if and only if $\mbC[x]\cdot u = V$. The ring of invariant functions $\mbC[\mfu(V\oplus \mbC)]^{U(V)}$ is a polynomial ring in $2n+1$ variables generated by
\begin{equation}\label{eq:invariant_functions_U}
\tr(x)/i,\ \tr(\wedge^2 x)/i^2,\ \ldots,\ \tr(\wedge^n x)/i^n,\ (u,u),\ (u,xu)/i,\ \ldots,\ (u,x^{n-1}u)/i^{n-1},\ e/i.
\end{equation}
The normalization by powers of $i$ is done to ensure that all these functions take values in the base ring (which is $\mbR$ for $\mbR$-points). Matching \eqref{eq:invariant_functions_GL} and \eqref{eq:invariant_functions_U} term by term provides an identification
\begin{equation}\label{eq:ident_inv_functions}
\mbC[\mfgl_{n+1}]^{\GL_n} \simlr \mbC[\mfu(V\oplus \mbC)]^{U(V)}.
\end{equation}
In this way, we have defined a quotient map
$$π_V:\mfu(V\oplus \mbC)\lr \mcQ'$$
to the same variety $\mcQ'$ as in \eqref{eq:quot_map_gl}.  Fix a Haar measure $dg$ on $U(V)(\bbR)$. Given a Schwartz function $Ψ\in \mcS(\mfu(V\oplus \mbC)(\mbR))$ and a regular semi-simple element $X\in \mfu(V\oplus \mbC)_\rs(\mbR)$, we have the orbital integral
$$\Orb(X, Ψ) = \int_{U(V)(\mbR)} Ψ(g^{-1}\cdot X) dg$$
which only depends on the orbit $U(V)(\mbR)\cdot X$. The image $π_V(\mfu(V\oplus \mbC)_\rs(\mbR))$ is open and closed in $\mcQ'_\rs(\mbR)$. We use extension by $0$ to view $\Orb(-,Ψ)$ as a smooth function
$$\Orb(-,Ψ):\mcQ'_\rs(\mbR) \lr \mbC.$$

\subsection{Transfer of orbital integrals}
We have just constructed a quotient variety $\mcQ'$ (isomorphic to $\mbA^{2n+1}$) as well as quotient morphisms
\begin{equation}
\xymatrix{
\mfgl_{n+1,\rs} \ar[rd]_{π} & & \mfu(V\oplus \mbC)_\rs \ar[ld]^{π_V}\\
& \mcQ'_\rs. &
}
\end{equation}
Elements $Y\in \mfgl_{n+1,\rs}(\mbR)$ and $X\in \mfu(V\oplus \mbC)_\rs(\mbR)$ are said to \emph{match} if $π(Y) = π_V(X)$. This sets up a bijection of orbits
\begin{equation}\label{eq:matching_bijection}
π(\mfgl_{n+1,\rs}(\mbR)) = \coprod_{p + q = n} π_{V_{p,q}}\big(\mfu(V_{p,q}\oplus \mbC)_\rs(\mbR)\big)
\end{equation}
where $V_{p,q}$ denotes a choice of hermitian $\mbC$-vector space of signature $(p,q)$, see \cite[Proposition 2.2.4.1]{Chaudouard}. We obtain the subspaces
$$\Orb(-,\mcS(\mfgl_{n+1}(\mbR))),\ \Orb(-,\mcS(\mfu(V_{p,q}\oplus \mbC)(\mbR)))\ \subset\ \mcC^\infty(\mcQ'_\rs(\mbR))$$
of those smooth functions that come as orbital integral from a Schwartz function on $\mfgl_{n+1}$ resp. $\mfu(V_{p,q}\oplus \mbC)$. This is the context of Conjecture \ref{conj:intro_transfer_all_Schwartz}.

\subsection{Polynomial Schwartz functions}
\label{ss:polynomial_Schwartz_functions}
We now give the details of the more algebraic Conjecture \ref{conj:intro_transfer_poly_Schwartz}. The quadratic form $Q(Y) = \tr(Y^2)$ on $\mfgl_{n+1}$ is $\GL_n$-invariant, i.e. lies in $\mbC[\mfgl_{n+1}]^{\GL_n}$.

We have the orthogonal decomposition
\[
\lie{gl}_{n+1} = \Sym_{n+1} \oplus \Skew_{n+1}
\]
into maximal positive (resp.\ negative) definite subspaces. The quadratic form $\bar Q(Y) = \tr(Y Y^t)$ is a Siegel majorant for $Q$, and the corresponding Siegel Gaussian, cf.\ \Cref{ss:Fock model}, can be expressed concretely as
\begin{equation}\label{eq:Siegel_GL}
φ(Y) = e^{-2π\tr(YY^t)} \in \mcS(\mfgl_{n+1}(\mbR)).
\end{equation}

\begin{definition}\label{def:pol_Schwartz_gl}
The space of \emph{polynomial Schwartz functions} on $\mfgl_{n+1}$ is the subspace
$$\mathcal S'_n := \mbC[\mfgl_{n+1}]\cdot φ \subset \mcS(\mfgl_{n+1}(\mbR)).$$
\end{definition}
Given an $n$-dimensional hermitian space $(V,h)$, we use \eqref{eq:ident_inv_functions} to view $Q$ as a quadratic form on $\mfu(V\oplus \mbC)$. It is given by $Q(X) = -\tr(X^2)$. Let $V = V_+ \oplus V_-$ be an orthogonal decomposition into a positive definite summand $V_+$ and a negative definite one $V_-$. Then $H_n := h\vert_{V_+} \oplus -h\vert_{V_-}$ is a positive definite hermitian form on $V$. We extend $H_n$ to $V\oplus \mbC$ as $H := H_n\oplus 1$. This gives rise to an orthogonal decomposition
 	\[
\lie u (V \oplus \bbC) =  \mathrm{SkewHerm}(H) \oplus \Herm(H)
\]
where $\Herm(H)$ and $\mr{SkewHerm}(H)$ denote the elements of $\mfu(V\oplus \mbC)$ which are (skew-)hermitian for $H$. (Note that $\Herm(H)$ is negative definite with respect to $-\tr(X^2)$.) Let $X \mapsto X^{\dagger}$ denote the adjoint with respect to $H$. Then the form $\ov Q(X)= \tr(X X^{\dagger})$ is a Siegel majorant for $Q$, and the resulting Siegel Gaussian is
$$ψ_V(X) = e^{-2π\tr(XX^\dagger)}.$$
\begin{definition}	The space of \emph{polynomial Schwartz functions} on $\lie u(V \oplus \bbC)$ is the subspace
$$  \mathcal T'_{V} := \mbC[\lie u(V \oplus \bbC)]\cdot \psi_V \subset \mcS(\lie u(V \oplus \bbC)(\mbR)).$$
\end{definition}
Note that the decomposition $V = V_+ \oplus V_-$ is unique up to the $U(V)(\mbR)$-action, so $\Orb(-,\mcT'_V)$ is independent of choices.

For every signature $(p,q)$, $p+q = n$, we denote by $ψ_{p,q}$ a choice of Siegel Gaussian on $\mfu(V_{p,q}\oplus \mbC)$ as just explained. Our conjecture is that smooth transfer is already defined for polynomial Schwartz functions.
\begin{conjecture}\label{conj:polynomial_Schwartz}
There is an identity of spaces of smooth functions on $\mcQ'_\rs(\mbR)$:
$$\Orb\big(-,\mathcal S'_n\big) = \bigoplus_{p + q = n} \Orb\big(-, \mathcal T'_{p,q}\big),$$
where $\mathcal T'_{p,q} = \mathcal T'_{V_{p,q}}$.
\end{conjecture}

\subsection{Reduction to $\mfsl_n\times \mbA^n\times (\mbA^n)^t$ and $\mfsu(V)\times V$}\label{ss:reduction_trace_0}
We will now eliminate a two-dimensional ``trivial'' factor of our setting. Along the way, we will also set up our notation for the rest of the article.

First, we define a subspace $\lie s_n \subset \lie{gl}_{n+1}$ by
\begin{equation}\label{eq:intro_sl_setting}
\lie s_n :=  \left\{\left. \begin{pmatrix}
y & v \\ w &
\end{pmatrix}\ \right\vert\ y \in \lie{sl}_n, \, v \in \mbA^n , \, w \in  (\mbA^n)^t \right\} \simeq 	\mfsl_n \times \mbA^n \times (\mbA^n)^t .
\end{equation}
The orthogonal complement of $\lie s_n$ is the two-dimensional subspace
\begin{equation}\label{eq:trivial_two_dim_sub}
\left.\left\{\begin{pmatrix}
a \cdot I_n & \\ & d
\end{pmatrix}\ \right\vert\ a,d\in \mbA^1 \right\}\ \subset\ \mfgl_{n+1}.
\end{equation}
Note that $\GL_n$ acts trivially on this subspace, and the two linear invariant polynomials $(\frac{2}{n})^{\frac{1}2}\tr(y), \sqrt 2\,  d\in \mcA'$, see \eqref{eq:invariant_functions_GL}, are the coordinate functions for the standard basis vectors $\frac{1}{\sqrt{2n}}\diag(I_n, 0)$ and $ \frac{1}{\sqrt 2}\diag(0_n,1)$, cf.\ \Cref{ss:infinitesimal_Weil_rep}.

We obtain the decomposition
\begin{equation*}
\begin{aligned}
\mcA & := \mbC[\lie s_n]^{\GL_n}, & \qquad \mcA' &= \mbC[\tr(y), d] \tensor_{\bbC} \mcA\\[1mm]
\mcQ & := \Spec(\mcA),&  \mcQ' &= \mbA^2 \times \mcQ.
\end{aligned}
\end{equation*}
The ring $\mcA$ is a polynomial ring in $2n-1$ variables generated by
\begin{equation}\label{eq:invariant_functions_SL}
\tr(\wedge^2 y),\ \ \ldots,\ \ \tr(\wedge^n y),\ \ wv,\ \ wyv,\ \ \ldots,\ \ wy^{n-1}v.
\end{equation}
The Siegel Gaussian is compatible with the decomposition of $\mfgl_{n+1}$ into \eqref{eq:intro_sl_setting} and \eqref{eq:trivial_two_dim_sub} in the sense that
$$
φ\left(\begin{pmatrix} y & v \\ w & d\end{pmatrix}\right) = e^{-2π(\tr(y)^2/n + d^2)}\cdot φ\left(\begin{pmatrix} y - \tr(y)/n & \ v \\ w & \end{pmatrix}\right).
$$
The first factor is $\GL_n(\mbR)$-invariant and comes by pullback from $\mcQ'(\mbR)$. We continue to write $φ$ for the restriction of the Siegel Gaussian to $\lie s_n$ and define
\begin{equation}\label{eq:def_S_n}
\mcS_n := \mbC[\lie s_n]\cdot φ.
\end{equation}
The orbital integral functional on $\mcS_n$ is now a map
$$\Orb(-,-)\colon \mcS_n \lr \mcC^\infty(\mcQ(\mbR)_\rs)$$
and we have the mutually compatible factorizations
\begin{equation}\label{eq:S_n_different_notation}
\begin{aligned}
\mcS'_n & = \big(\mbC[\tr(y), d]\cdot e^{-2π(\tr(y)^2/n+ d^2)}\big) \tensor \mcS_n\\
\Orb(-,\mcS'_n) & = \big(\mbC[\tr(y), d]\cdot e^{-2π(\tr(y)^2/n + d^2)}\big) \tensor \Orb(-,\mcS_n).
\end{aligned}
\end{equation}

The  same kind of factorizations exist on the unitary side: in the notation of \eqref{eq:standard_notation_U}, the complement of the two-dimensional trivial representation $\left.\left\{\left(\begin{smallmatrix} a \, I_n & \\ & e \end{smallmatrix}\right)\ \right\vert\ a,e\in i\cdot \mbA^1\right\} \subset \mfu(V\oplus \mbC)$ is
\begin{equation}\label{eq:intro_su_setting}
\lie t_{V}: = \left\{\left.\begin{pmatrix}
x & u \\ -u^* &
\end{pmatrix} \in \mfu(V\oplus \mbC) \ \right\vert\ \tr(x) = 0, \ u \in V \right\} \simeq \lie{su}(V) \oplus V
\end{equation}
The invariant polynomials $\mbC[\lie t_V]^{U(V)}$ are generated by
\begin{equation}\label{eq:invariant_functions_SU}
\tr(\wedge^2 x)/i^2,\ \ \ldots,\ \ \tr(\wedge^n x)/i^n,\ \ (u,u),\ \ (u,xu)/i,\ \ \ldots,\ (u,x^{n-1}u)/i^{n-1}.
\end{equation}
Matching these with \eqref{eq:invariant_functions_SL} identifies $\mbC[\lie t_V]^{U(V)}$ with $\mcA$ and provides the matching bijection
$$[\lie s_n(\mbR)]_{\rs} = \coprod_{p + q = n} [\lie t_{p,q}(\mbR)]_{\rs}.$$
of regular semisimple orbits, where we abbreviate $\lie t_{p,q} = \lie t_{V_{p,q}}$. We still use $ψ_{p,q}$ to denote the restriction of $ψ_{p,q}$ to $\lie t_{p,q}(\mbR)$ and define
\begin{equation}\label{eq:def_Tpq}
\mcT_{p,q} := \mbC[\lie t_{p,q}]\cdot ψ_{p,q}.
\end{equation}
The orbital integral on $\mcT_{p,q}$ is a functional
$$\Orb(-,-):\mcT_{p,q}\lr \mcC^\infty(\mcQ(\mbR)_\rs),$$
and $\mcT'_{p,q}$ is related to $\mcT_{p,q}$ in the same way as in \eqref{eq:S_n_different_notation}. The following conjecture is hence immediately equivalent to Conjecture \ref{conj:polynomial_Schwartz}.
\begin{conjecture}\label{conj:polynomial_Schwartz_SL}
There is an identity of spaces of orbital integral functions on $\mcQ_\rs(\mbR)$:
$$\Orb(-,\mcS_n) = \bigoplus_{p + q = n} \Orb(-,\mcT_{p,q}).$$
\end{conjecture}

For general $n$, we currently only know the following result which pertains to signatures $(n,0)$ and $(0,n)$.

\begin{theorem}[\protect{\cite[Theorems 4.15 and 6.2]{MSY}}]\label{thm:pos_def}
Every function $Ψ\in \mcT_{n,0} \oplus \mcT_{0,n}$ has a transfer to $\mcS_n$.
\end{theorem}
\begin{proof}
First \cite[Theorem 4.15]{MSY} constructs a Schwartz function $Φ_{n,0}\in \mbC[\mfgl_{n+1}]\cdot φ$
such that
$$\Orb(Y, Φ_{n,0}) = \begin{cases}
e^{-2π\tr(Y^2)} & \text{if $Y$ matches to signature $(n,0)$}\\
0 & \text{otherwise.}
\end{cases}$$
Normalizing the Haar measure on the compact group $U(V_{n,0})(\mbR)$ by requiring its total volume to be $1$, this expresses that the restriction of $Φ_{n,0}$ to $\mfs_n(\mbR)$ is a transfer of $ψ_{n,0}$.

Next, a comparison of the invariants \eqref{eq:invariant_functions_SL} and \eqref{eq:invariant_functions_SU} reveals that if a regular semisimple tuple $(y,v,w)$ matches with a tuple $(x, u) \in \lie t_V = \lie{su}(V) \oplus V$ for a hermitian space $V$, then $(y,v,-w)$ matches with the same element $(x,u)$ viewed for the hermitian space $-V$. In particular, the map $(y,v,w) \mapsto (y, v , -w)$ interchanges orbits matching to signature $(p,q)$ with those matching to signature $(q,p)$. It also leaves the Siegel majorant $Y\mapsto \tr(YY^t)$ invariant. Defining $Φ_{0,n}$ by $Φ_{0,n}(y, v, w) := Φ_{n,0}(y, v, -w)$ hence gives a transfer of $ψ_{0,n}$. Here, we also normalized the Haar measure on $U(V_{0,n})(\mbR)$ to have total volume $1$.

Finally, given any $f \in \mbC[\mfsu(V_{n,0}) \times V_{n,0}]$, the averaged polynomial $f_0 = \int_{U(n)(\mbR)} g\cdot f\, dg$ satisfies $\Orb(-,fψ_{n,0}) = \Orb(-,f_0ψ_{n,0})$ because $ψ_{n,0}$ is $U(n)$-invariant. As $f_0$ is also $U(n)$-invariant, we can view it as an element of $\mbC[\lie s_n]^{\GL_n}$ via \eqref{eq:ident_inv_functions}. (Note that there is no difference between invariance for the algebraic groups or their $\mbR$-points here.) Then $f_0 Φ_{n,0}$ is a transfer of $f ψ_{n,0}$. The same argument applies to $\mcT_{0,n}$.
\end{proof}

\subsection{Orbital integrals and Weil representation}
\label{sec:Weil rep transfer}

Let $V$ be an $n$-dimensional hermitian $\mbC$-vector space. Recall that the identification $\mbC[\mfs_n]^{\GL_n}\simto \mbC[\mft_V]^{U(V)}$ was defined by matching the generators in \eqref{eq:invariant_functions_SL} and \eqref{eq:invariant_functions_SU}. We will now show that this identification can also be obtained by a linear transformation over $\mbC$, which will have consequences for invariant differential operators and orbital integrals.

\begin{proposition}\label{prop:inv_diff_op_isomorphic}
There exists a $\mbC$-linear isomorphism $\mfs_n(\mbC)\simto \mft_V(\mbC)$ which is equivariant for an isomorphism of algebraic groups $\GL_{n,\mbC}\simto U(V)_\mbC$. The resulting isomorphism
$$ρ:\mbC\{\mfs_n\}^{\GL_n} \simlr \mbC\{\mft_V\}^{U(V)}$$
extends the identification $ρ:\mbC[\mfs_n]^{\GL_n} \simlr \mbC[\mft_V]^{U(V)}$ given by \eqref{eq:invariant_functions_SL} and \eqref{eq:invariant_functions_SU}. In particular, for every invariant differential operator $\mcD\in \mbC\{\mfs_n\}^{\GL_n}$,
$$ρ(\mcD)\vert_{\mbC[\mft_V]^{U(V)}} = ρ\big(\mcD\vert_{\mbC[\mfs_n]^{\GL_n}}\big)$$
and for any pair of transferring data $Φ \in \mcS(\mfs_n(\mbR))$ and $(Ψ_{p,q})_{p+q = n} \in \bigoplus_{p+q = n} \mcS(\mft_{p,q}(\mbR))$,
\begin{equation}\label{eq:transfer_inv_diff_op}
\Orb(-,\mcD(Φ)) = \sum_{p + q = n} \Orb(-,ρ(\mcD)(Ψ_{p,q})).
\end{equation}
\end{proposition}
\begin{proof}
Let $(\ ,\ )$ be the hermitian form on $V$ and let $\langle\ ,\ \rangle = \mr{tr}_{\mbC/\mbR}\circ (\ ,\ )$ be its trace. Then
$$\langle λv, w\rangle = \langle v, σ(λ)w\rangle$$
where $λ\in \mbC$ and where $σ$ denotes complex conjugation. We decompose $\mbC\tensor_\mbR\mbC\simto \mbC\times \mbC$ via $a\tensor b \mapsto (ab, a\ov b)$. The $\mbC$-linear extension $σ = 1\tensor σ$ of $σ$ is then given by $σ(a, b) = (b, a)$. We obtain a corresponding decomposition $\mbC\tensor_\mbR V \simto V_0\oplus V_1$ and both $V_0$ and $V_1$ are Lagrangian. Choose $\mbC$-bases for $V_0$ and $V_1$ such that the symmetric bilinear form $\langle\ ,\ \rangle$ is given by
$$\langle\ \,,\ \rangle \longleftrightarrow \begin{pmatrix}
&1_n \\ 1_n &
\end{pmatrix}.$$
Recall that $\mfsu(V)$ is the space of those trace $0$ complex linear endomorphisms $x$ of $V$ such that $\langle x v, w\rangle = -\langle v, x w\rangle$ for all $v,w\in V$. For the above choice of basis of $\mbC\tensor_\mbR V$, this gives
$$\mfsu(V)(\mbC) = \big\{(x^t, -x)\ \mid\ x \in \mfsl_n(\mbC)\big\}.$$
Similarly, $U(V)_\mbC$ is identified with $\{(g^{-1,t}, g)\ \mid\ g\in \GL_{n,\mbC}\}$. In these coordinates, we define $ρ:\mfs_n(\mbC)\to \mft_V(\mbC)$ by
\begin{equation}\label{eq:def_ρ_linear_identification}
ρ(y, v, w)\,:=\,\big((iy^t, -iy),\ (w^t, v)\big)
\end{equation}
and this map is equivariant for the identification $\GL_{n,\mbC}\simto U(V)_\mbC$, $g\mapsto (g^{-1,t}, g)$. The normalizing factors $i^k$ in \eqref{eq:invariant_functions_SU} are given by $1\tensor i^k \mapsto (i^k, (-i)^k)$ in $\mbC\tensor_\mbR\mbC$. So it is immediate that
$$\mr{char}(y;T) = \mr{char}((i,-i)^{-1}\cdot (iy^t, -iy); T)$$
under the diagonal embedding $\mbC\to \mbC\times \mbC$. In the same way, we have
$$\begin{aligned}
wy^kv & = \big((w^t,v), (i, -i)^{-k} ((iy^t)^kw^t, (-iy)^kv)\big)\\
& = \big\langle (w^t,v),\ ((y^t)^kw^t, y^kv)\big\rangle/2\\
& = (wy^kv + v^t(y^t)^kw^t)/2.
\end{aligned}$$
This shows that the isomorphism
$$ρ:\mbC[\mfs_n]^{\GL_n}\simlr \mbC[\mft_V]^{U(V)}$$
induced from \eqref{eq:def_ρ_linear_identification} agrees with the isomorphism given by identifying the invariant polynomials in \eqref{eq:invariant_functions_SL} and \eqref{eq:invariant_functions_SU}. It follows formally from definitions that for all invariant differential operators $\mcD\in \mbC\{\mfs_n\}^{\GL_n}$,
$$ρ(\mcD)\vert_{\mbC[\mft_V]^{U(V)}} = ρ\big(\mcD\vert_{\mbC[\mfs_n]^{\GL_n}}\big).$$
Finally, if $Φ\in \mcS(\mfs_n(\mbR))$ and $(Ψ_{p,q})\in \bigoplus \mcS(\mft_{p,q}(\mbR))$ are Schwartz functions which are mutual transfers, then
$$\begin{aligned}
\Orb(-,\mcD(Φ)) & \ \overset{\text{Prop. \ref{prop:diff_op_orb_int_generic}}}{=}\ (\mcD\vert_{\mbC[\mfs_n]^{\GL_n}})\Orb(-,Φ)\\[1mm]
& \ \mkern 5mu \overset{\text{transfer}}{=}\ \sum_{p + q = n} ρ(\mcD\vert_{\mbC[\mfs_n]^{\GL_n}})\Orb(-,Ψ_{p,q})\\[1mm]
& \ \overset{\text{Prop. \ref{prop:diff_op_orb_int_generic}}}{=}\ \sum_{p+q=n} \Orb(-,ρ(\mcD)(Ψ_{p,q})).
\end{aligned}$$
All claims of the proposition are now proved.
\end{proof}

Examples of invariant differential operators are provided by the Weil representation. Consider again the quadratic form $Q(Y) = \tr(Y^2)$ from \S\ref{ss:polynomial_Schwartz_functions} which we restrict to $\lie s_n = \mfsl_n\times \mbA^n\times (\mbA^n)^t$. It is the sum of the two forms
$$Q_0(Y) = \tr(y^2),\quad Q_1(Y) = 2wv,\quad Y = (y,v,w)\in \lie s_n.$$
on $\lie{sl}_n$ and $\bbA^n \oplus (\bbA^n)^t$ respectively. Note that $Q_0,Q_1\in \mcA$ precisely span the space of homogeneous quadratic elements. For an $n$-dimensional hermitian space $V$, the corresponding forms on $\lie t_V = \lie{su}(V) \oplus V$ are
$$Q_0(X) = -\tr(x^2),\quad Q_1(X) = 2(u,u),\quad X = (x,u)\in \lie t_V.$$
Then $Q_0$ and $Q_1$ define Weil representations as described in \S\ref{ss:Weil_rep} and \eqref{eq:Weil_Lie_algebra}:
\begin{equation}\label{eq:Omega_orb_ints}
\begin{aligned}
\Omega:\Mp_2\times \Mp_2 & \ \circlearrowright\ \mcS(\mfsl_n(\mbR)\times \mbR^n\times \mbR_n),\ \ \mcS(\mfsu(V)(\mbR)\times V)\\[1mm]
ω:\mfsl_2\times \mfsl_2 & \lr \mbC\{\mfs_n\}^{\GL_n},\ \ \mbC\{\mft_V\}^{U(V)}.
\end{aligned}
\end{equation}
\begin{corollary}\label{cor:Weil_rep_transfer}
Let $Φ \in \mcS(\mfs_n(\mbR))$ and $(Ψ_{p,q})_{p + q = n} \in \bigoplus_{p + q = n} \mcS(\mft_{p,q}(\mbR))$ be mutual transfers. Then for every $(h_1,h_2)\in (\mfsl_2\times \mfsl_2)(\mbC)$ also
\begin{equation}\label{eq:transfer_Weil_compatible}
ω(h_1, h_2)Φ\quad\text{and}\quad \big(ω(h_1,h_2)Ψ_{p,q}\big)_{p + q = n}
\end{equation}
are mutual transfers.
\end{corollary}
\begin{proof}
This is a special case of Proposition \ref{prop:inv_diff_op_isomorphic}. Namely, for every signature $(p,q)$ in question, the isomorphism $ρ:\mfs_n(\mbC)\simto \mft_{p,q}(\mbC)$ identifies the pairs of forms $(Q_0,Q_1)$ on the two spaces. Hence, by the coordinate invariance of the Weil representation as in \S\ref{ss:infinitesimal_Weil_rep_algebraic}, it also identifies the two infinitesimal Weil representations. The claim now follows from \eqref{eq:transfer_inv_diff_op}.
\end{proof}

\begin{remark}\label{rmk:use_Xue_instead} Corollary \ref{cor:Weil_rep_transfer} can also be obtained from \cite[Theorem 9.1]{Xue} by differentiating the group-theoretic Weil representation.
\end{remark}

We define $\mfA \subseteq \mbC\{\mfs_n\}^{\GL_n}$ as the subalgebra generated by the infinitesimal Weil representation and the invariant polynomials:
\begin{equation}\label{eq:def_A_fat}
\mfA := \mbC\,\big\langle ω_{\mfs_n,Q_0,Q_1}(\mfsl_2 \times \mfsl_2),\, \mcA\,\big\rangle.
\end{equation}
Using Proposition \ref{prop:inv_diff_op_isomorphic}, for every $V$, we can also naturally view $\mfA$ as the subring
\begin{equation}\label{eq:A_fat_on_unitary_side}
\mfA = \mbC\,\big\langle ω_{\mft_V,Q_0,Q_1}(\mfsl_2 \times \mfsl_2),\, \mcA\,\big\rangle \subseteq \mbC\{\mft_V\}^{U(V)}.
\end{equation}
Moreover, this identification is such that the following square commutes:
\begin{equation}\label{eq:comm_square_A_fat}
\xymatrix{
& \mbC\{\mfs_n\}^{\GL_n} \ar[rd]^-{\text{Lem. \ref{lem:restrict_diff_op}}} & \\
\mfA\ \ar[ru] \ar[rd] \ar@{-->}[rr]^{\mcD\longmapsto \mcD\vert_\mcQ} & & \mbC\{\mcQ\}.\\
& \mbC\{\mft_V\}^{U(V)} \ar[ru]_-{\text{Lem. \ref{lem:restrict_diff_op}}} &
}
\end{equation}
Differentiation preserves polynomial Schwartz functions and Proposition \ref{prop:diff_op_orb_int_generic} ensures that the $\mfA$-action descends to $\Orb(-, \mcS_n)$ and all $\Orb(-,\mcT_{p,q})$ as
$$\begin{aligned}
\Orb(-,\mcD(Φ)) & = (\mcD\vert_\mcQ)(\Orb(-,Φ))\\
\Orb(-,\mcD(Ψ)) & = (\mcD\vert_\mcQ)(\Orb(-,Ψ)).
\end{aligned}$$
In particular, if $\Orb(-,Φ) = \sum_{p+q = n} \Orb(-,Ψ_{p,q})$, then for every $\mcD\in \mfA$ also
\begin{equation}\label{eq:compatib_transfer_mfA}
\Orb(-,\mcD(Φ)) = \sum_{p+q = n} \Orb(-,\mcD(Ψ_{p,q})).
\end{equation}
Extrapolating from the cases $n = 1$ and $n = 2$, we have the following expectation.

\begin{conjecture}\label{conj:cyclic_Orb} (1) Each of the $\mfA$-modules $\Orb(-,\mcT_{p,q})$ is cyclic with generator $\Orb(-,ψ_{p,q})$.

\mn (2) Each $\psi_{p,q}$ admits a transfer, and this collection of transfers generates $\Orb(-,\mcS_n)$ as an $\mfA$-module.
\end{conjecture}
We conclude this subsection by noting that every differential operator on $\mcQ_\rs$ comes by restriction from an invariant differential operator on $(\lie s_n)_\rs$ or, alternatively, $(\lie t_V)_\rs$. In particular, applying a differential operator to some $\Orb(-,Φ)$ or $\Orb(-,Ψ)$ can always be realized by applying an invariant differential operator with meromorphic coefficients to $Φ$ resp. $Ψ$.

Define the \emph{discriminant} $Δ\in \mcA$ as in \cite[\S4]{Xue}. We do not need the precise formula but only use the property $\mcQ_\rs = \mcQ[Δ^{-1}]$. The case $n = 2$ is made explicit in \S\ref{ss:geometry_Q}.

\begin{lemma}\label{lem:restriction_surjective}
For every $n\geq 1$ and every hermitian $\mbC$-vector space $V$ of dimension $n$, the restriction maps
$$\begin{aligned}
\mbC\{\lie s_n\}[Δ^{-1}]^{\GL_n} & \lr \mbC\{\mcQ\}[Δ^{-1}]\\
\mbC\{ \lie t_V\}[Δ^{-1}]^{U(V)} & \lr \mbC\{\mcQ\}[Δ^{-1}]
\end{aligned}$$
are surjective.
\end{lemma}
The statement is similarly true for the morphisms $\mfgl_{n+1}\to \mcQ'$ and $\mfu(V\oplus \mbC)\to \mcQ'$, but we have only stated the version we will use explicitly later.
\begin{proof}
The claimed surjectivity can be checked after base extension of the involved algebraic groups and varieties to $\mbC$. The unitary setting is a form of the $\GL_n$-setting, so it suffices to consider the $\GL_n$-case. There exists a section
$$σ':\mcQ'\lr \mfgl_{n+1}$$
of the quotient map, sending a characteristic polynomial $T^n + a_{n-1}T^{n-1} + \ldots + a_0$, a tuple of values $b_j = wy^jv$ with $j = 0,\ldots,n-1$, and a scalar $d$ to the matrix
$$σ'(a_i, b_j, d) := \begin{pmatrix}
0& 1& & & & b_0\\
& 0 & 1& & & b_1 \\
& & \ddots& \ddots & & \vdots \\
& & & 0& 1 & b_{n-2}\\
-a_0 & -a_1 & \hdots & -a_{n-2} & -a_{n-1} & b_{n-1}\\
1 & 0 & & \hdots & 0 & d
\end{pmatrix}.$$
This map also restricts to a section $σ:\mcQ \to \mfsl_n\times \mbA^n\times (\mbA^n)^t$. Since stabilizers of regular semi-simple elements are trivial, this shows that the quotient map is a trivial $\GL_n$-torsor over $\mcQ_\rs$:
$$\begin{aligned}
\GL_n \times \mcQ[Δ^{-1}] & \simlr (\mfsl_n\times \mbA^n\times (\mbA^n)^t)[Δ^{-1}]\\[1mm]
(g,(a_i, b_j, 0)) & \longmapsto g\cdot σ(a_i, b_j, 0).
\end{aligned}$$
The projection map $\mbC\{\GL_n\times \mcQ[Δ^{-1}]\}^{\GL_n} \to \mbC\{\mcQ[Δ^{-1}]\}$ is trivially surjective, proving the lemma.
\end{proof}

\subsection{Lie algebra coinvariants}

Differentiating the action of $\GL_n(\mbR)$ on $\mcS(\mfs_n(\mbR))$, we obtain a $(\mfgl_n(\mbR), O(n))$-module. The subspace $\mcS_n$ is $(\mfgl_n(\mbR), O(n))$-stable and $O(n)/\SO(n)$ acts on the Lie algebra coinvariants
$$ \mcS_n \big/\langle L(Φ) \ | \ L\in \mfgl_n(\mbR),\, Φ\in \mcS_n\rangle.$$
We define
\begin{equation}\label{eq:def_bar_S_n}
\ov {\mcS_n} := \left( \mcS_n \big/\langle L(Φ) \ | \ L\in \mfgl_n(\mbR),\, Φ\in \mcS_n\rangle \right)_{(O(n),η)}
\end{equation}
as the $\eta$-isotypic quotient.
Analogously, each $\mcT_{p,q}$ is a $(\mfu(V_{p,q}), U(p)\times U(q))$-module and we set
\begin{equation}\label{eq:def_bar_T_pq}
\ov {\mcT_{p,q}} := \mcT_{p,q} \Big/\langle L(Ψ) \ | \ L\in \mfu(V_{p,q}),\, Ψ\in \mcT_{p,q}\rangle.
\end{equation}
It is clear that we have factorizations of the orbital integral functionals as
\begin{equation}
\xymatrix@R=2mm{
\mcS_n \ar@{->>}[r] & \ov{\mcS_n} \ar@{->>}[r] & \Orb(-,\mcS_n)\\
\mcT_{p,q} \ar@{->>}[r] & \ov{\mcT_{p,q}} \ar@{->>}[r] & \Orb(-,\mcT_{p,q})
}
\end{equation}
and that the $\mfA$-action descends from $\mcS_n$ to $\ov{\mcS_n}$ as well as from $\mcT_{p,q}$ to $\ov{\mcT_{p,q}}$. We make the following general conjecture.
\begin{conjecture}\label{conj:cyclic_coinvariants}  (1) Each of the $\mfA$-modules $\ov{\mcT_{p,q}}$ is cyclic and generated by the image of $ψ_{p,q}$.

\mn (2) The $\mfA$-module $\overline{\mathcal S_n}$ admits a system of $\lie A$-module generators $\{ \Phi_{p,q} \ | \ p+q = n\}$ such that the biweight of $\Phi_{p,q}$ with respect to the Weil representation coincides with that of $\psi_{p,q}$.
\end{conjecture}

The expected $n+1$ generators of $\ov{\mcS_n}$ should have a representation-theoretic description, but our proof for $n = 2$ does not per se suggest one. In line with an expected density statement for orbital integrals, one can also ask the following question.
\begin{question}\label{qu:surjection_is_iso}
Is it true that the quotient maps $\ov{\mcS_n} \twoheadrightarrow \Orb(-,\mcS_n)$ and $\ov{\mcT_{p,q}}\twoheadrightarrow \Orb(-,\mcT_{p,q})$ are in fact isomorphisms?
\end{question}
This is easy to confirm for $n = 1$, but already the case $n = 2$ does not seem to be straightforward. If the answer to Question \ref{qu:surjection_is_iso} were affirmative and if Conjecture \ref{conj:polynomial_Schwartz_SL} held, then we would obtain an isomorphism of $\mfA$-modules
\begin{equation}\label{eq:decomp_algebraic}
\ov{\mcS_n} \simlr \bigoplus_{p+q=n} \ov{\mcT_{p,q}}.
\end{equation}
This is a purely algebraic statement that is formulated independently of orbital integrals.

\begin{example}\label{ex:cylic_T_definite}
The space of coinvariants $\ov{\mcT_{n,0}}$ is a cyclic $\mfA$-module. Namely, taking $\mfu(n)$-coinvariants can be realized as the averaging map
$$\mr{vol}(U(n))^{-1} \int_{U(n)}\ dg : \mbC[\lie t_{n,0}] {\relbar\joinrel\twoheadrightarrow} \mbC[\lie t_{n,0}]^{U(n)}$$
and this quotient is already cyclic (free of rank $1$ even) for the subring $\mcA\subset \mfA$. The same argument shows that $\ov{\mcT_{0,n}}$ is cyclic as $\mfA$-module.
\end{example}

\begin{remark}
With respect to the decompositions described in \S\ref{ss:reduction_trace_0}, we have the following description of the ring $\mfA'$ from the introduction:
\begin{equation}
\begin{aligned}
\mfA' & = \mbC\{\tr(y), d\} \tensor_\mbC \mfA\\
& \iso \mbC\{\tr(x), e\} \tensor_\mbC \mfA.
\end{aligned}
\end{equation}
The factor $\mbC[\tr(y), d]\cdot e^{-2π(\tr(y)^2/n + d^2)}$ in \eqref{eq:S_n_different_notation} is a cyclic $\mbC\{\tr(y),d\}$-module and a similar observation holds on the unitary side. In this way, Conjecture \ref{conj:cyclic_coinvariants} immediately implies Conjecture \ref{conj:intro_coinvariants}.
\end{remark}

\subsection{Transposition}

We end by noting a sign symmetry of the setting. Consider the transpose map $Y\mapsto Y^t$ on $\mfgl_{n+1}$. It restricts to $(y, v, w) \mapsto (y^t, w^t, v^t)$ on $\lie s_n =\mfsl_n\times \mbA^n\times (\mbA^n)^t$. The $2n-1$ invariant polynomials from \eqref{eq:invariant_functions_SL} are invariant under this map, which means that transposition preserves individual regular semi-simple orbits. For $Y\in \mfgl_{n+1}(\mbR)_\rs$, we can hence define a sign $τ(Y)$ by
\begin{equation}
τ(Y) := η(h)
\end{equation}
where $h \in \GL_n(\bbR)$ is the (unique) element such that $Y^t = h^{-1} Y h$.
The sign $\tau(Y)$ only depends on the orbit $\GL_n(\mbR)\cdot Y$. For a Schwartz function $Φ$, we define $Φ^t(Y) := Φ(Y^t)$. The Haar measure on $\GL_n(\mbR)$ is transposition invariant, so it is easy to verify that
$$\Orb(Y, Φ) = τ(Y)\Orb(Y, Φ^t).$$
\begin{lemma}\label{lem:transposition}
Assume that $Y\in \mfgl_{n+1}(\mbR)_{\rs}$ matches to signature $(p,q)$. Then $τ(Y) = (-1)^q$.
\end{lemma}
\begin{proof}
Write $Y$ in block matrix form $\left(\begin{smallmatrix} y & v \\ w & d\end{smallmatrix}\right)$. The function $Y\mapsto τ(Y)$ is locally constant on $\mfgl_{n+1}(\mbR)_\rs$. After perturbing $Y$ slightly, we may assume that $y$ is regular semi-simple. Acting on $Y$ with $\GL_n(\mbR)$ and using that $τ(Y)$ only depends on the orbit of $Y$, we may block diagonalize $y$ and assume that $Y$ is of the form
\begin{equation}\label{eq:block_diag_general}
\begin{pmatrix}
λ_1 &    & & & & & & & v_1 \\
& \ddots & & & & & & & \\
& & λ_r    & & & & & & \vdots \\
& & & α_1 & β_1 & & & & \\
& & & -β_1 & α_1 & & & & \vdots \\
& & & & & \ddots & & & \\
& & & & & & α_s & β_s & v_{n-1} \\
& & & & & & -β_s & α_s & v_n \\
w_1 & \hdots & & & \hdots & & w_{n-1}& w_n & d
\end{pmatrix}.
\end{equation}
Here, $v_1,\ldots,v_r,w_1,\ldots,w_r$ are all non-zero because $Y$ is regular semi-simple. Similarly, all $β_1,\ldots,β_s$ are non-zero and no pair
$$(v_{r+1},v_{r+2}),\ \ldots,\ (v_{n-1},v_n),\ (w_{r+1},w_{r+2}),\ \ldots,\ (w_{n-1},w_n)$$
equals $(0,0)$. Constructing an analogously block diagonalized element on the unitary side shows that $Y$ matches to signature $(p, q) = (r_+ + s,\, r_- + s)$ where $r_{\pm}$ is the number of indices $1\leq i\leq r$ with $\mr{sign}(v_iw_i) = \pm$. The matrix $\left(\begin{smallmatrix} v_i/w_i & \\ & 1 \end{smallmatrix}\right)$ conjugates $\left(\begin{smallmatrix} λ_i & v_i \\ w_i & \end{smallmatrix}\right)^t$ back into $\left(\begin{smallmatrix} λ_i & v_i\\ w_i & \end{smallmatrix}\right)$; its determinant has sign $\mr{sign}(v_iw_i)$. An element of negative determinant is needed to conjugate the transpose
$$\begin{pmatrix}
α_i & β_i & v_{r + 2i - 1} \\
-β_i & α_i & v_{r + 2i}\\
w_{r + 2i - 1} & w_{r + 2i} &
\end{pmatrix}^t$$
back into the original $(3\times 3)$-matrix. Taking the product of all resulting signs, we find that $τ(Y) = (-1)^q$ as claimed.
\end{proof}

The two quadratic forms $Q_0, Q_1\in \mcA$ introduced in \S\ref{sec:Weil rep transfer} are transposition invariant. The two corresponding Weil representations on $\mcS(\mfs_n(\mbR))$ hence commute with transposition. In fact, every operator $\mcD\in \mfA$ satisfies
$$(\mcD Φ)^t = \mcD (Φ^t).$$
The Gaussian is transposition invariant, compare \eqref{eq:Siegel_GL}, so $\mcS_n$ is stable under $Φ\mapsto Φ^t$. Taking eigenspaces, we obtain decompositions
\begin{equation}
\begin{aligned}
\mcS_n & = \mcS_n^{t\,=\, \mr{id}} \oplus \mcS_n^{t \,=\, -1}\\
\ov {\mcS_n} & = \ov{\mcS_n}^{t \,=\, \mr{id}} \oplus \ov{\mcS_n}^{t \,=\, -1}\\
\Orb(-,\mcS_n) & = \Orb(-,\mcS_n^{t\,=\,\mr{id}}) \oplus \Orb(-,\mcS_n^{t \,=\, -1}).
\end{aligned}
\end{equation}
Lemma \ref{lem:transposition} explained that $\Orb(-,\mcS_n^{t = \pm 1})$ is supported on invariants of signature $(p,q)$ with $(-1)^q = \pm 1$.

\newpage

\part{Structure of $\Orb(-,\mcS_2)$ and $\Orb(-,\mcT_{1,1})$}

In this part, we show that $\ov{\mcS_2}$ is generated by the images of three functions $Φ_{2,0}$, $Φ_{1,1}$ and $Φ_{0,2}$ as $\mfA$-module. Here, $Φ_{2,0}$ and $Φ_{0,2}$ are from Theorem \ref{thm:pos_def}, while $Φ_{1,1}$ will be specified in \eqref{eq:def_Phi_11}. We similarly show that $\ov{\mcT_{1,1}}$ is generated by the image of $ψ_{1,1}$.

We begin with a brief discussion of the case $n = 1$, which will serve as a (slightly oversimplified) model for the subsequent sections.

\section{The case $n = 1$}\label{s:n_equal_1}
In this section, we prove Conjectures \ref{conj:polynomial_Schwartz_SL} and \ref{conj:cyclic_coinvariants} for $n=1$. We begin by fixing coordinates on the general linear side. As $\lie{sl}_1(\mathbb R) = \{0\}$, the relevant quadratic space is
\[
\lie s_1 = \mathbb R \times \mathbb R
\]
equipped with the quadratic form $Q(v,w) = 2vw$. Following the conventions in \S\ref{ss:infinitesimal_Weil_rep}, we choose the standard basis
\[
y_1 := \frac1{ 2 } (1,\,1), \qquad y_2 := \frac{1}{ 2} (1,\,-1),
\]
and note that $y_1$ is positive, while $y_2$ is negative. Let $x_1, x_2 \in \bbC[\lie s_1]$ denote the coordinate functions with respect to the basis $\{ y_1, y_2\}$.

We pass to the Fock model; although this is not strictly necessary, it will simplify our computations  and help motivate our discussion for the case $n=2$ below. Let $\varphi \in \mcS(\mfs_1)$ denote the Siegel Gaussian function, as in  \eqref{eq:Siegel Gaussian general}, which is concretely given by the formula $\varphi(x_1 y_1 + x_2 y_2) = e^{- \pi (x_1^2 + x_2^2)}$. The polynomial Schwartz space is $\mathcal S_1 = \mathbb C[\lie{s_1}] \varphi$. Let
\[
\mathcal F = \mathcal F[ \lie s_1] = \bbC[z_1, z_2]
\]
denote the corresponding Fock model, cf.\ \Cref{def:Fock}.
An element $t \in \GL_1(\mathbb R) = \mathbb R^{\times}$ acts on $ (v,w) \in \lie s_1$ by
\[ t\cdot (v,w) = (tv, t^{-1}w).
\]
Differentiating the corresponding action $\GL_1(\mathbb R) \to \Aut(\mcS(\lie s_1))$ induces a homomorphism $\lie{gl}_1(\mathbb R) = \mathbb R  \to \bbC\{ \lie s_1\}$; letting  $X = - 1 \in \lie{gl}_1$ denote a basis vector, a direct computation shows that the image of $X$ in $\bbC\{\lie s_1\}$ is
\[
x_1 \frac{\partial}{\partial x_2} + x_2 \frac{\partial}{\partial x_1}.
\]
Applying the intertwiner $\iota \colon \bbC\{\lie s_1\} \isomto \calF\{ \lie s_1\}  = \mathbb C\{ z_1, z_2\}$ from \eqref{eq:Fock_model_intertwiner}, we obtain
\begin{align}\label{eq:n1_Lie_alg_operator}
\iota(X)
&= - \frac{1}{4 \pi } z_1 z_2  + 4 \pi \frac{\partial^2}{\partial z_1 \partial z_2}.
\end{align}

Finally, recall that we have defined the algebra $\mathfrak A = \mbC\langle ω(\mfsl_2 \times \mfsl_2), \mathcal A\rangle$ acting on $\bbC[\lie s_1] \simeq \calF[\lie s_1]$, where in this case, we have  $\mathcal A = \mathbb C[\lie s_1]^{\GL_1(\mathbb R)} = \mathbb C[ x_1^2 - x_2^2]$. The first factor of $\lie{sl}_2$ acts trivially. As for the second factor we record the following explicit formulas, cf.\ \eqref{eq:LRH_Fock}, for the action of $(0,L), (0,R) \in \lie{sl}_2 \times \lie{sl}_2$ on $\mathcal F$:
\[
\omega(0,L ) = 2 \pi \frac{\partial^2}{\partial z_1^2} - \frac{1}{8 \pi} z_2^2, \qquad \omega(0,R) = - \frac{1}{8 \pi} z_1^2 + 2 \pi \frac{\partial^2}{\partial z_2^2}.
\]

\begin{proposition} \label{prop:struct_of_S1}
As an $\lie A$-module, the quotient $\ov{\mcS_1}$ is generated by the classes of the two functions
\[
\Phi_{1,0} := x_1 \varphi, \qquad \Phi_{0,1} := x_2 \varphi,
\]
where $\varphi $ is the Siegel Gaussian on $\lie s_1$. Consequently, the orbital integrals $\Orb(-,\Phi_{1,0})$ and $\Orb(-,\Phi_{0,1})$ generate $\Orb(-,\mcS_1)$ as $\mfA$-module.

\begin{proof}
We translate the problem to the Fock model. For $n \geq 0$, let
\[ \calF^n\subset \calF = \calF[\lie s_1] \]
denote the subspace of polynomials of degree $\leq n$. Using the intertwiner $\iota \colon \mathcal S_1 \simto \calF$, we consider $\GL_1(\mbR)$- and $(\{\pm 1\}, η)$-coinvariants on $\mcF$. Set
\[
R := \lie{A} \cdot z_1 +  \lie{A} \cdot z_2 + \langle L(Φ) \mid L\in \mfgl_1(\mbR),\ Φ\in \mcF\rangle + \mcF^{O(1)}.
\]
Here, we note that $-1\in O(1) = \{\pm 1\}$ acts as $-1$ on $\mfs_1$. The monomial-wise even degree polynomials $\mcF^{O(1)}$ provide the kernel for the $(O(1), η)$-coinvariants in \eqref{eq:def_bar_S_n}.
We also define
\[
R^n =  \begin{cases} \lie A \cdot \calF^{n-1} + R, & n > 0 \\ R, & n=0 \end{cases}
\]
Noting that  $4\pi \iota(\Phi_{1,0}) =  z_1$ and $ 4\pi \iota(\Phi_{0,1}) =  z_2$,  the proposition is equivalent to the equality $\calF = R$; for this, it will suffice to show that $\calF^n  \subset  R^n$ for all $n$, and we proceed by induction on $n$.

The claim is clear for $n = 0$ because $φ\leftrightarrow 1 \in \mcF$ lies in $\mcF^{O(1)}$. It is also clear for $n = 1$ because $z_1,z_2\in R^1$ by construction. For the induction step,  suppose $ f= z_1^a z_2^b \in \calF^{n}$ is a monomial with $a+b=n \geq 2$. If $ab > 0$, then \eqref{eq:n1_Lie_alg_operator} allows  us to write
\[
f =  - 4 \pi \iota(X) ( z_1^{a-1} z_2^{b-1} ) +   (\text{lower degree terms}) \ \in R^n.
\]
Hence, we may assume that $f = z_1^a$ or $f= z_2^b$. If, for example, $a \geq 2$, we have
\[
z_1^{a} = -8 \pi \omega(0,R)(z_1^{a-2}) + (\text{lower degree terms}) \in R^n
\]
and similarly, if $b \geq 2$, then
\[
z_2^b = -8 \pi \omega(0,L)(z_2^{b-2}) +(\text{lower degree terms}) \in R^n.
\]
This concludes the induction.
\end{proof}
\end{proposition}

We can immediately draw our final conclusion for $n = 1$:

\begin{theorem} \Cref{conj:polynomial_Schwartz_SL} holds for $n=1$, i.e. we have an $\lie A$-module equality
\begin{equation}\label{eq:transfer_n_1}
\Orb(-,\mathcal S_1) =  \Orb(-,\mathcal T_{1,0}) \oplus \Orb(-,\mathcal T_{0,1}).
\end{equation}

\begin{proof}
 After an appropriate normalization of Haar measures, \cite[Example 4.16]{MSY} (which is the specialization of \Cref{thm:pos_def} for $n=1$) states that the functions $\Phi_{1,0}$ and $\Phi_{0,1}$ defined above are transfers of $\psi_{1,0}$ and $\psi_{0,1}$ respectively.
By Proposition \ref{prop:struct_of_S1}, $\Orb(-,Φ_{1,0})$ and $\Orb(-,Φ_{0,1})$ generate $\Orb(-,\mcS_1)$ as $\mfA$-module. By Example \ref{ex:cylic_T_definite}, $\Orb(-,ψ_{1,0})$ generates $\Orb(-,\mcT_{1,0})$ and $\Orb(-,ψ_{0,1})$ generates $\Orb(-,\mcT_{0,1})$. Since transfer is compatible with $\mfA$-actions, we obtain \eqref{eq:transfer_n_1}.
\end{proof}
\end{theorem}

Noting that the biweights of $Φ_{1,0}$ and $Φ_{0,1}$ equal those of $ψ_{1,0}$ and $ψ_{0,1}$, Proposition \ref{prop:struct_of_S1} also proves Conjecture \ref{conj:cyclic_coinvariants} for $n = 1$.

\section{Structure of $\Orb(-,{\mcS_2})$} \label{sec:Orb(S2)}
We now turn to the case $n=2$. Our first step is to elucidate the $\lie A$-module structure of $\Orb(-,\mathcal S_2)$.
To begin, we recall the notation from previous sections. Let
\[
\lie{s}_2 := \lie{sl}_2 \times \mathbb R^2 \times \mathbb R_2 = \left.\left\{  \begin{pmatrix} y & v \\ w & 0 \end{pmatrix} \ \right\vert \ \tr(y) = 0  \right\}
\]
with quadratic form $Q(y,v,w) = \tr(y^2) + 2 wv$.  We choose standard basis vectors as follows: for the $\lie{sl}_2$ factor, we take the basis vectors
\[
y_1 = \frac{1}{ 2} \begin{pmatrix} 1 & \\ & -1 \end{pmatrix} , \qquad y_2  = \frac{1}{ 2} \begin{pmatrix} &1  \\ 1 &  \end{pmatrix} \qquad y_3  = \frac{1}{ 2} \begin{pmatrix} & 1\\-1 &  \end{pmatrix}.
\]
and for the $\mathbb R_2 \times \mathbb R^2$ factors, we have basis vectors	
\begin{equation}\label{eq:standard_basis_R_2}	
\begin{aligned}
y'_1 &= \frac1{{2}} \left[ \begin{pmatrix}1 & 0 \end{pmatrix}, \begin{pmatrix} 1 \\ 0 \end{pmatrix} \right],  && y'_2 = \frac1{{2}} \left[ \begin{pmatrix}0 & 1 \end{pmatrix}, \begin{pmatrix} 0 \\ 1 \end{pmatrix} \right]   \\
y'_3 &= \frac1{{2}} \left[ \begin{pmatrix}1 & 0 \end{pmatrix}, \begin{pmatrix} -1 \\ 0 \end{pmatrix} \right], \ 	 && y'_4= \frac1{{2}} \left[ \begin{pmatrix}0 & 1 \end{pmatrix}, \begin{pmatrix} 0 \\ -1 \end{pmatrix} \right].
\end{aligned}
\end{equation}
 Let $x_1, x_2, x_3, x_1', x_2', x_3', x_4' \in \bbC[\lie s_2]$ denote the coordinate functions on $\lie s_2$ with respect to the basis $\{ y_1, y_2, y_3, y_1', y_2', y_3', y_4' \}$.
Note that $y_1, y_2, y_1'$ and $y_2'$ are positive vectors, while the remainder are negative.

Attached to this basis is the Siegel Gaussian $\varphi$, as in \eqref{eq:Siegel Gaussian general}, and the polynomial Schwartz space $\mathcal S_2 = \mathbb C[\lie{s}_2] \varphi$. 	This basis also gives rise to a set of generators for the polynomial Fock space
\[
\mathcal F := \mathcal F[\lie s_2] = \mathbb C[z_1, z_2, \zeta, z_1', z_2', \zeta_3', \zeta_4'].
\]
Here we have written $\zeta, \zeta'_3, \zeta_4'$ instead of $z_3, z_3'$ and $z_4'$ as a mnemonic device to remind us that these variables correspond to negative vectors.
It will be convenient to change coordinates by writing
\begin{equation}\label{eq:new_Fock_vars}
z_{\pm} := z_1 \pm i z_2, \qquad z'_{\pm} = z'_1 \pm i z_2', \qquad \zeta'_{\pm} = \zeta_3' \pm i \zeta_4'
\end{equation}
because these are eigenvectors for the maximal compact $SO(2)\subset \GL_2(\mbR)$. We set
\begin{equation} \label{eqn:GL2 Fock new coords}
\calF_0 := \bbC[ z_+, z_-, \zeta], \qquad \calF_1 := \bbC[z_+', z_-', \zeta_+', \zeta_-']
\end{equation}
so that
\[
\calF = \calF_0 \otimes_{\bbC} \calF_1.
\]
As before, differentiating the action of $\GL_2(\mathbb R)$ on $\mathcal S_2$ and applying the intertwiner $\iota \colon \bbC\{ \lie s_2\} \simto \calF\{ \lie s_2\}$ from \eqref{eq:Fock_model_intertwiner} gives a Lie algebra homomorphism $\lie{gl}_2 \to \calF\{\lie s_2\}$.

We make the $\lie{gl}_2$ action explicit in the following lemma; in the sequel, this will be useful when paired with the observation that $\Orb(X\cdot Φ) = 0$ for all $X \in \lie{gl}_2$ and $Φ \in \mathcal S_2$.

\begin{lemma} \label{lem:orth ops gl2}
Extend the map $\iota \colon \lie{gl}_2 \to \calF\{ \lie s_2\}$ described above to a complex linear map $\iota \colon \lie{gl}_2(\mathbb C) \to \calF\{ \lie s_2\}$. Then we have the following explicit operators for a basis of $\lie{gl}_2(\bbC)$:
\begin{align*}
\xi_1& := \iota (\mathrm{Id}) = - \frac{1}{8 \pi} \left( z_+' \zeta_-' + z_-' \zeta_+' \right)   +  8 \pi \left( \frac{\partial^2}{\partial z_+' \partial \zeta_-'} + \frac{\partial^2}{\partial z_-' \zeta_+'} \right) \\
\xi_2 &:= \iota \begin{psmallmatrix} & -i \\ i & \end{psmallmatrix}  =   2 \left( z_+ \frac{\partial}{\partial {z_+}}-  z_- \frac{\partial}{\partial {z_-}} \right)	+  z'_+ \frac{\partial}{\partial {z_+'}} -  z_-' \frac{\partial}{\partial {z_-'}} +  \zeta_+' \frac{\partial}{\partial {\zeta_+'}} -  \zeta_-' \frac{\partial}{\partial {\zeta_-'}} \\
\xi_3 &:= \iota \begin{psmallmatrix} 1&-i\\-i & -1 \end{psmallmatrix} = \frac{i}{2 \pi} z_- \zeta - \frac{1}{4 \pi }z_-'\zeta_-' - 16 \pi i \frac{\partial^2}{\partial z_+ \partial \zeta}+ 16 \pi \frac{\partial^2}{\partial z_+'  \partial \zeta_+'} \\
\xi_4 & := \iota  \begin{psmallmatrix} 1&i\\i & -1 \end{psmallmatrix}  =  \frac{-i}{2 \pi} z_+ \zeta - \frac{1}{4 \pi }z_+'\zeta_+' + 16 \pi i \frac{\partial^2}{\partial z_- \partial \zeta}+ 16 \pi \frac{\partial^2}{\partial z_-'  \partial \zeta_-'}
\end{align*}

	\begin{proof} Let $X_1, X_2, X_3, X_4 \in \bbC\{ \lie{s}_2\}$ denote the images of the basis vectors
\[
\begin{pmatrix} 1 & \\ & 1 \end{pmatrix} , \begin{pmatrix}1 & \\ & -1 \end{pmatrix} , \, \begin{pmatrix} & 1 \\ & \end{pmatrix} , \, \text{ and } \begin{pmatrix} & \\ 1 & \end{pmatrix}
\]
under the map $\lie{gl}_2(\bbR) \to \bbC\{ \lie s_2\}$. Then, considering $X_1$ for example, setting $c(t) = (t+t^{-1})/2$ and $s(t) = (t - t^{-1})/2$ for $t \in \bbR$, we note that
\begin{multline*}
\left(\left(\begin{smallmatrix} t & \\ & t\end{smallmatrix}\right) Φ \right)(x_1, x_2, x_3, x_1', x_2', x_3', x_4')  = Φ(\left(\begin{smallmatrix} t & \\ & t\end{smallmatrix}\right)^{-1}(x_1, x_2, x_3, x_1', x_2', x_3', x_4'))\\
= Φ(x_1, x_2, x_3, c(t) x_1' + s(t) x_3', c(t) x_2' + s(t) x_4' , s(t)x_1' + c(t) x_3', \, s(t) x_2' + c(t) x_4').
\end{multline*}
for every $Φ \in \mathcal S_2$. Differentiating and setting $t = 1$, we find
\[
X_1 =  x_3' \frac{\partial}{\partial x_1'} + x_4' \frac{\partial}{\partial x_2' } + x_1' \frac{\partial }{\partial x_3'} + x_2' \frac{\partial }{\partial x_4'}
\]
In exactly the same way, we calculate
\begin{small}
\begin{align*}
X_2 &= -2 x_3 \frac{\partial}{\partial x_2} -2 x_2 \frac{\partial}{\partial x_3} + x_3' \frac{\partial}{\partial x_1'} -  x_4' \frac{\partial}{\partial x_2'} +  x_1' \frac{\partial}{\partial x_3'} -   x_2' \frac{\partial}{\partial x_4'} \\
X_3 &= (x_3 - x_2) \frac{\partial}{\partial x_1} + x_1 \left( \frac{\partial}{\partial x_2} + \frac{\partial}{\partial x_3} \right) + \frac12 \left( x_2 '- x_4' \right) \left( \frac{\partial}{\partial x_3'} - \frac{\partial}{\partial x_1'} \right) + \frac12 \left( x_1' + x_3' \right)\left( \frac{\partial}{\partial x_2' } + \frac{\partial}{\partial x_4'}\right) \\
X_4 &= (x_2 + x_3) \frac{\partial}{\partial x_1} - x_1 \left(\frac{\partial}{\partial x_2} - \frac{\partial}{\partial x_3} \right)  + \frac12 \left( x_2' + x_4' \right) \left( \frac{\partial}{\partial x_1'} + \frac{\partial}{\partial x_3'} \right) - \frac12(x_1' - x_3') \left(\frac{\partial}{\partial x_2'} - \frac{\partial}{\partial x_4'} \right).
\end{align*}
\end{small}
The stated formulas for $ξ_1,\ldots,ξ_4$ follow by taking linear combinations of the $X_i$ and substituting \eqref{eq:Fock_model_intertwiner} as well as \eqref{eq:new_Fock_vars}.
\end{proof}
\end{lemma}

Next, we express the action of the infinitesimal Weil representation of $\omega\colon \lie{sl_2} \times \lie{sl}_2 \to \calF\{ \lie s_2\}$ in our coordinates. Concretely, this means starting from \eqref{eq:LRH_Fock} for $z_1, z_2, ζ$ as well as $z_1', z_2', ζ_3', ζ_4'$ and substituting \eqref{eq:new_Fock_vars}. The result is
\begin{equation}\label{eq:Weil_rep_S2}
\begin{aligned}
\omega(L,0) &= 8 \pi \frac{\partial^2}{\partial z_+ \partial z_-} - \frac{1}{8 \pi} \zeta^2 & \omega(0,L)&= 8 \pi \frac{\partial^2}{\partial z_+' \partial z_-'} - \frac1{8\pi} \zeta_+'\zeta_-' \\
\omega(R,0) &=  2 \pi \frac{\partial^2}{\partial \zeta^2} - \frac1{8 \pi} z_+ z_- & \omega(0,R) &=  8 \pi \frac{\partial^2}{\partial \zeta'_+ \partial \zeta'_-}  - \frac{1}{8 \pi} z'_+ z'_- \\
\omega(H,0) &= z_+ \frac{\partial}{ \partial z_+} + z_- \frac{\partial}{\partial z_-} - \zeta \frac{\partial}{\partial \zeta} + \frac12
& \omega(0,H) &=  z'_+ \frac{\partial}{ \partial z'_+} + z'_- \frac{\partial}{\partial z'_-} - \zeta'_+ \frac{\partial}{\partial \zeta'_+}   - \zeta'_- \frac{\partial}{\partial \zeta'_-} .
\end{aligned}
\end{equation}
Finally, recall the algebra $ \lie{A} = \mbC\langle ω(\lie{sl}_2 \times \lie{sl}_2), \mcA\rangle$ defined in \eqref{eq:def_A_fat}, where, for $n=2$, the ring of invariant polynomials is
\[
\mathcal A = \bbC[\lie s_2]^{\GL_2(\bbR)} = \bbC[ \det y, wv, wyv ], \qquad  \lie{s}_2 = \left.\left\{  \begin{pmatrix} y & v \\ w & 0 \end{pmatrix} \ \right\vert \ \tr(y) = 0  \right\}.
\]
We will later need the explicit expression for the action of the invariant polynomial
\[p(y,v,w) := 16 wyv \in \calA
\]
in Fock model coordinates (here the normalizing  factor 16 is just for convenience). In the coordinates $\{ x_i, x_j'\}$, we have
\begin{align*}
p  &= 2(x_1' + x_3', x_2' + x_4') \begin{pmatrix} x_1 & x_2 + x_3 \\ x_2 - x_3 & - x_1 \end{pmatrix} \begin{pmatrix} x_1' - x_3' \\ x_2' - x_4' \end{pmatrix}  \\
&=     \left( (x_1' + i x_2')^2 - (x_3'+ix_4')^2 \right) (x_1 - i x_2) +  \left(  (x_1' - i x_2' )^2 - (x_3' - i x_4')^2 \right)   (x_1 + i x_2)  \\
& \qquad +  4 x_3(-x_1'x_4'+ x_2'x_3').
\end{align*}
Applying the intertwiner  $\iota \colon \bbC\{ \lie{s_2} \} \to \calF\{\lie s_2\}$ on Weyl algebras, a brief calculation yields
\begin{equation} \label{eqn:iota_p_explicit}
\iota(p) =  A + B + C
\end{equation}
with
\begin{equation} \label{eqn:p inv Fock decomp}
\begin{aligned}
A &= \left( z_- / 4 \pi + 2 \partial_{z_+} \right) \left( \left(z'_+/4 \pi + 2 \partial_{z'_-} \right)^2  - \left( \zeta'_+/4 \pi + 2 \partial_{\zeta'_-} \right)^2 \right)   \\
B &=  \left( z_+ / 4 \pi + 2 \partial_{z_-} \right)	  \left( \left(z'_-/4 \pi + 2 \partial_{z'_+} \right)^2  - \left( \zeta'_-/4 \pi + 2 \partial_{\zeta'_+} \right)^2 \right) \\
C&= \frac{-i}{ \pi}\left( \zeta/4\pi + \partial_{\zeta} \right) \Big[ \frac{1}{8 \pi}(z'_+ \zeta'_- - z'_- \zeta'_+) + \left( z'_+ \partial_{\zeta'_+}  - z'_- \partial_{\zeta'_-} - \zeta'_+ \partial_{z'_+} + \zeta'_- \partial_{z'_-} \right) \\
& \qquad\qquad\qquad\qquad\qquad - 8 \pi ( \partial_{z'_+} \partial_{\zeta'_-} - \partial_{z'_-} \partial_{\zeta'_+}) \Big].
\end{aligned}
\end{equation}

\begin{theorem}\label{thm:generators_S2}
As an $\lie A$-module, the coinvariant space $\ov{\mathcal S_2}$ (cf.\ \eqref{eq:def_bar_S_n}) is generated by the inverse images of the  three elements $$\zeta, \  z_+ (z_-')^2 \text{ and } z_+ (\zeta_-')^2 $$  under the isomorphism $\iota \colon \mathcal S_2 \simeq \calF(\lie s_2)$.

\begin{proof}

For an integer $n \geq 0$, let $\calF^{n} \subset \calF(\mfs_2)$ denote the subspace spanned by polynomials of degree $\leq n$. Set
\[
R := \lie A \cdot \zeta + \lie A \cdot z_+(z'_-)^2 + \lie A \cdot z_+(\zeta'_-)^2 + \ker\big(\mcS_2 \twoheadrightarrow \ov{\mcS_2}\big)
\]
and, for $n>0$, define
\[
R^n := \lie A \cdot \calF^{n-1} + R.
\]
By convention, we set $R^0 = R$.
To prove the theorem it suffices to show by induction that
\begin{equation} \label{eqn:gl2 proof induction statement}
\mathcal F^n \subset R^n
\end{equation}
for all $n$. The case $n = 0$ is straightforward: the constant polynomial $\mathbf 1 \in \calF$ corresponds to the Siegel Gaussian $φ\in \mcS_2$ which is $O(2)$-invariant and hence maps to zero in $\ov{\mcS_2}$ because the definition involves taking $(O(2), η)$-coinvariants, see \eqref{eq:def_bar_S_n}. Assume $n \geq 1$ from now on and that $\mcF^{n-1} \subseteq R^{n-1}$ is already established.

Let $f  \in \calF$ be a polynomial of degree $n$. We may assume that $f = f_0 f_1$, where $f_0 \in \calF_0 = \bbC[z_{+}, z_-, \zeta]$, and where
\[
f_1 = (z'_+)^{a_+}(z'_-)^{a_-}(\zeta'_+)^{b_+} (\zeta'_-)^{b_-}	 \in \mathcal F_1 = \bbC[ z'_+, z'_-, \zeta'_+, \zeta'_-]	\]
is a monomial.

\mn   \emph{$(\star)$ Claim: If $a_+a_- \geq 1$ or $b_+b_-\geq 1$, then $f\in R^n$:}  If, for example, we have $a_+ a_- \geq 1$, then we can write
\[
f = - 8 \pi  \, \omega(0,R)  \Big( {f}/{ z'_+ z'_- }\Big)  + (\text{lower degree terms}).
\]
The first summand lies in $\mfA\cdot \mcF^{n-2}$ and the terms of lower degree lie in $R^n$ by induction, so $f$ lies in $R^n$. In exactly the same way, we may use $\omega(0,L)$ to conclude that $f \in R^n$ when $b_+ b_- \geq 1$.

\mn  \emph{$(\star\star)$ Claim: we may reduce to the case $a_-b_- = 0$ and $a_+b_+ = 0$:} Assume that $a_- b_-\geq 1$ and set $g = f/z_-' \zeta_-'$. Then
\[
f = - 4 \pi \xi_3(g) + 2i (z_- \zeta g)+ (\text{lower degree terms})
\]
where $\xi_3$ is the operator from Lemma \ref{lem:orth ops gl2}. Note that $\xi_3(g) \in \ker(\mcS_2 \twoheadrightarrow \ov{\mcS_2}) \subset R^n$. Moreover, the lower degree terms lie in $R^n$ by induction. Thus the claim $f\in R^n$ is reduced to $z_-ζg\in R^n$. 	Applying this operation repeatedly, we have reduced to $a_- b_- = 0$. A similar argument with $\xi_4$ reduces to the case $a_+ b_+ = 0$.

After $(\star)$ and $(\star\star)$, we need to show $f\in R^n$, for all $f$ of degree $n$ which are of the form $f = f_0f_1$ with
\begin{equation} \label{eqn:f1 dichotomy}
f_1 = (z_+')^{a_+} (\zeta_-')^{b_-} \quad\ \ \text{or}\ \ \ \! \quad f_1 = (z_-')^{a_-}(\zeta_+')^{b_+}.
\end{equation}
Assume $f_1$ is of the first form with $a_+b_-\geq 2$, set $g = f/z_+' \zeta_-' $, and write
\begin{align*}
f =  - 8 \pi \xi_1(g) -  z_-' \zeta_+' g + (\text{lower degree terms}).
\end{align*}
The element
\[
z_-' \zeta_+' g = f_0 \cdot \left( (z_+')^{a_+ - 1}z_-' \zeta_+' (\zeta_-')^{b_--1} \right)
\]
lies in $R^n$ by $(\star)$ and by our assumption $a_+b_- \geq 2$. Lower degree terms lie in $R^n$ by induction hypothesis. The second case of \eqref{eqn:f1 dichotomy} is handled similarly, showing that we can reduce \eqref{eqn:f1 dichotomy} to all cases with $a_+b_- \leq 1$ and $a_-b_+\leq 1$.

The operator $ξ_1$ also allows to deduce the case $f_1 = z'_-ζ'_+$ from the case $f_1 = z'_+ζ'_-$. Thus, to summarize the argument so far, it is enough to consider $f$ of degree $n$ of the form $f= f_0 f_1$, where
\begin{equation} \label{eqn:f1 opts}
f_1 \in \{1,\, (z_+')^{a},\, (z_-')^{a},\, (\zeta_+')^a,\, (\zeta_-')^a,\, z_+' \zeta_-'\ | \ a \geq 1 \}.
\end{equation}
We may also assume that $f_0 = z_+^{\alpha_+} z_-^{\alpha_-} \zeta^{\beta}$ is a monomial. Using the operators $\omega(L,0)$ and $\omega(R,0)$, and arguing as in $(\star)$, it suffices to consider the case $\alpha_+ \alpha_- = 0$ and $\beta \leq 1$, i.e. we may assume
\begin{equation} \label{eqn:f0 opts}
f_0 \in \{  \zeta^{\beta},  \, z_+^{\alpha}\zeta^{\beta}, \,   z_-^{\alpha} \zeta^{\beta} \ | \ \alpha \geq 1, \beta \leq 1\}
\end{equation}

Next, we consider the action of $\xi_2$, and observe that if a monomial $f$ has non-zero $ξ_2$-eigenvalue, then $f \in \ker(\mcS_2 \twoheadrightarrow \ov{\mcS_2})$. Note that $(z_{\pm}, ζ, z'_{\pm}, ζ'_{\pm})$ are $ξ_2$-eigenvectors with respective eigenvalues $(\pm 2, 0, \pm 1, \pm 1)$. If we take $f=f_0f_1$ with $f_0$ and $f_1$ in \eqref{eqn:f0 opts} and \eqref{eqn:f1 opts} as above, the only products with vanishing eigenvalue are
\begin{equation} \label{eqn:f opts}
f \in  \{ \zeta^{\beta},  \, \zeta^{\beta} z_+'\zeta_-',  \, \zeta^{\beta}z_+^a(z'_-)^{2a},\zeta^{\beta}z_+^a(\zeta'_-)^{2a}, \zeta^{\beta} z_-^a(z'_+)^{2a},\, \zeta^{\beta} z_-^a(\zeta'_+)^{2a} \ | \ a \geq 1, \beta \leq 1 \}.
\end{equation}
Assume $f$ is divisible by $ζz_+$, for example $f = ζ(z_+)^a(ζ'_-)^{2a}$. Using $\xi_4$, we write
$$f = {2 \pi i} \xi_4( z_+^{a-1} (\zeta'_-)^{2a}) + \frac{i}2 z_+^{a-1} z'_+ \zeta'_+ (\zeta'_-)^{2a}.$$
The second summand lies in $R^n$ by $(\star)$, so $f\in R^n$. Similar arguments (always using $ξ_3$ or $ξ_4$) show that $f\in R^n$ whenever $f$ from \eqref{eqn:f opts} is divisible by $ζz_+$ or $ζz_-$. In this way, we have reduced to
\begin{equation} \label{eqn:f opts better}
f \in  \{1,\, \zeta,\, z_+'\zeta_-',\, ζ z_+'\zeta_-',\, z_+^a(z'_-)^{2a},\, z_+^a(\zeta'_-)^{2a},\, z_-^a(z'_+)^{2a},\, z_-^a(\zeta'_+)^{2a} \ | \ a \geq 1\}.
\end{equation}

At this point, we have exhausted the structure provided by the Weil representation and the $\mfgl_2$-action. To proceed, we use the cubic invariant polynomial from \eqref{eqn:iota_p_explicit}.

Recall that the weight operators $(H,0), (0,H) \in \lie{sl}_2 \oplus \lie{sl}_2$ decompose $\mcF(\mfs_2)$ into weight spaces as per \eqref{eq:weight_monomial}. Monomial differential operators from $ι(\mbC\{\mfs_2\})$ act homogeneously for this decomposition, making $\mbC\{\mfs_2\}$ into a bigraded algebra. Concretely, the biweight of its generators is given as follows.\footnote{The mnemonic is that $z$-variables raise the weight by $1$ and $ζ$-variables lower it by $1$.}
\begin{center}
\def\arraystretch{1.2}
\begin{minipage}{3cm}
\begin{tabular}{|c|l|}
	\hline
	$z_\pm$ & $(1, 0)$\\
	$ζ$ & $(-1,0)$\\
	$z'_{\pm}$ & $(0,1)$\\
	$ζ'_{\pm}$ & $(0,-1)$\\
	\hline
\end{tabular}
\end{minipage}
\begin{minipage}{2.5cm}
\begin{tabular}{|c|l|}
	\hline
	$\partial_{z_\pm}$ & $(-1, 0)$\\
	$\partial_ζ$ & $(1,0)$\\
	$\partial_{z'_{\pm}}$ & $(0,-1)$\\
	$\partial_{ζ'_{\pm}}$ & $(0,1)$\\
	\hline
\end{tabular}
\end{minipage}
\end{center}
Since the weight operators $(H, 0)$ and $(0,H)$ lie in $\mfA$, $\mfA \subset \mbC\{\mfs_2\}$ is a graded subalgebra. In particular, the invariant polynomial $p(y,v,w) = 16wyv$ from \eqref{eqn:iota_p_explicit} has a decomposition
\[
p  = \sum_{ \substack{i \in \{ -1, 1\} \\ j \in \{ -2, 0, 2\} }} p^{(i,j)}
\]
into weight components, with each component $p^{(i,j)}$ again lying in $\mfA$. The components $ι(p^{(i,j)})$ can be read off directly from \eqref{eqn:p inv Fock decomp}; for example,
\[
\iota\left(p^{(1, 2)}\right) = \underset{A^{(1,2)}}{\underbrace{\frac{z_-}{4π} \left(  (z'_+)^2/16 \pi ^2 - 4 \partial^2_{\zeta'_-}\right)}}
+ \underset{B^{(1,2)}}{\underbrace{\frac{z_+}{4π} \left( (z'_-)^2/16\pi^2 - 4 \partial^2_{\zeta'_+}\right)}}
+ \underset{C^{(1,2)}}{\underbrace{\frac{-i}{\pi} \partial_{\zeta} { (z'_+ \partial_{\zeta'_+} - z'_- \partial_{\zeta'_-})}}}.
\]
Now suppose that $f = z_+^a (z'_-)^{2a}$ with $a \geq 2$. We then have
$$
z_+^{a} (z'_-)^{2a} = 64 \pi^3 \ \iota(p^{(1,2)}) \left( z_+^{a-1} (z'_-)^{2a - 2} \right)  - z_+^{a-1} z_- (z'_+)^2 (z'_-)^{2a-2}.
$$
The last summand lies in $R^n$ by the analog of $(\star)$ for \eqref{eqn:f0 opts}, so $f\in R^n$.
Also using the operators $\iota(p ^{(1, -2)})$, the same kind of argument shows that if $a \geq 2$, then
\[
z_+^{a}(\zeta'_-)^{2a}, \ z_-^a(z'_+)^{2a}, \ z_-^a(\zeta'_+)^{2a} \ \in \ R^n.
\]
We also observe that
\begin{align*}
\iota(p^{(1,2)})(\mathbf 1) &= \frac{1}{64 \pi^3} ( z_+(z'_-)^2 + z_-(z'_+)^2),  & \iota(p^{(1,-2)}) (\mathbf 1) &= \frac{-1}{64 \pi^3} ( z_+(\zeta'_-)^2 + z_-(\zeta'_+)^2).
\end{align*}
Thus $z_+(z_-')^2$ lies in $R^1$ if and only if $z_-(z_+')^2$ lies in $R^1$, and similarly for $z_+(ζ'_-)^2 \leftrightarrow z_-(ζ'_+)^2$. At this point, the proof of the theorem is reduced to checking that $f\in R^n$ (where $n = \deg(f)$) for the finite list
\[
f \in \{  \zeta, z_+'\zeta'_-, \zeta z_+' \zeta'_-, z_+(z'_-)^2, z_+(\zeta'_-)^2 \}.
\]

The three elements $\zeta,  z_+(z'_-)^2$ and $z_+(\zeta'_-)^2 $ lie in $R^n$ by definition. For $z'_+ζ'_-$ and $ζz'_+ζ'_-$, we note the two identities
\begin{align*}
-8πξ_1(ζ^k) & = ζ^k(z'_+ζ'_- + z'_-ζ'_+)\\
8iπ^2 ι(p^{(1, 0)})(ζ^{k+1}) & = (k+1)ζ^k(z'_+ζ'_- - z'_-ζ'_+).
\end{align*}
They imply $ζ^kz'_+ζ'_- \in R^{k+2}$ and the proof is complete.
\end{proof}
\end{theorem}

Given $f\in \mcF$, we write $\Orb(-,f) := \Orb(-,ι^{-1}(f))$ for simplicity. Our next goal is to determine $\Orb(-,ζ)$, $\Orb(-,z_+(z'_-)^2)$ and $\Orb(-,z_+(ζ_-')^2)$. We denote the natural coordinates on $\mfsl_2(\mbR)\times \mbR^2\times \mbR_2$ by
\begin{equation}\label{eq:nat_coords}
\begin{pmatrix}
y_{11} & y_{12} & v_1\\
y_{21} & -y_{11} & v_2\\
w_1 & w_2 &
\end{pmatrix}.
\end{equation}
Concerning $ζ$, we calculate
\begin{equation}\label{eq:zeta_explicit}
\begin{aligned}
ι^{-1}(ζ\cdot 1) & = (2πx_3 - \partial_{x_3})φ\\
& = 4πx_3φ\\
& = 4π(y_{12} - y_{21}) φ.
\end{aligned}
\end{equation}
We define
\begin{equation}\label{eq:def_Phi_11}
Φ_{1,1} = (y_{12} - y_{21})φ.
\end{equation}
Then $Φ_{1,1}^t = -Φ_{1,1}$, so $\Orb(-,Φ_{1,1})$ is supported on elements matching to signature $(1,1)$ (Lemma \ref{lem:transposition}). We will show in Part III that $Φ_{1,1}$ is a transfer of $ψ_{1,1}$ (up to constant).

We next study $z_+(z'_-)^2$ and decompose it into real and imaginary part:
$$
\begin{aligned}
z_+(z'_-)^2 & = z_1\big((z_1')^2 - (z_2')^2\big) + 2 z_2 z_1' z_2'\\
& \ + i\big[z_2 \big((z_1')^2 - (z_2')^2\big) - 2z_1z'_1z'_2\big].
\end{aligned}
$$
Consider the element $σ = \left(\begin{smallmatrix} & 1 \\ 1 & \end{smallmatrix}\right)\in O(2)$ and note the relations
$$σ(z_1) = -z_1,\ \ σ(z_2) = z_2,\ \ σ(z_i') = z_{3-i}'.$$
The real part of $z_+(z'_-)^2$ is $σ$-invariant and hence lies in $\ker(\mcS_2\twoheadrightarrow \ov{\mcS_2})$. We obtain
\begin{equation*}
\Orb\big(-,z_+(z'_-)^2\big) = i\cdot \Orb\big(-,\underset{Z_1}{\underbrace{z_2\big((z_1')^2 - (z_2')^2\big)}} - 2 \underset{Z_2}{\underbrace{z_1z'_1z_2'}}\big).
\end{equation*}
We have
$$\begin{aligned}
ι^{-1}(Z_1)(φ) & = 64π^3 x_2\big((x'_1)^2 - (x'_2)^2\big)φ\\
ι^{-1}(Z_2)(φ) & = 64π^3 x_1x_1'x_2' φ.
\end{aligned}$$
Translating to the coordinates from \eqref{eq:nat_coords}, we finally get
\begin{equation}\label{eq:Orb_20}
\frac{1}{64π^3i} \Orb(-, z_+(z'_-)^2) = \Orb(-,Φ_{2,0})
\end{equation}
with
\begin{equation}\label{eq:Phi_KM}
Φ_{2,0} = \big[(y_{12}+y_{21})\big((v_1 + w_1)^2 - (v_2 + w_2)^2\big) - 4 y_{11} (v_1 + w_1)(v_2+w_2)\big]\,φ.
\end{equation}
This agrees with the Schwartz function provided in \cite[Example 4.17]{MSY} (up to constant), so with suitable Haar measure choices, $Φ_{2,0}$ is a transfer of $ψ_{2,0}$ as in Theorem \ref{thm:pos_def}.

Finally, consider again the map $c(y,v,w) = (y, v, -w)$ and define $Φ_{0,2} := c^*(Φ_{2,0})$. This is a transfer of $ψ_{0,2}$ (up to constant), see Theorem \ref{thm:pos_def}. We also have $c^*(z'_\pm) = - ζ'_\pm$, so $c^*(z_+(z'_-)^2) = (z_+(ζ'_-)^2)$ and hence
\begin{equation}\label{eq:Orb_02}
\frac{1}{64 π^3i}\Orb(-, z_+(ζ'_-)^2) = \Orb(-,Φ_{0,2}).
\end{equation}

\begin{corollary}\label{cor:main_S_2}
As an $\mfA$-module, $\Orb(-, \mcS_2)$ is generated by $\Orb(-,Φ_{2,0})$, $\Orb(-,Φ_{1,1})$ and $\Orb(-,Φ_{0,2})$.
\end{corollary}
\begin{proof}
This is clear from Theorem \ref{thm:generators_S2}, the fact that $\Orb(-,-)$ factors over $\mcS_2\twoheadrightarrow \ov{\mcS_2}$, and identities \eqref{eq:def_Phi_11}, \eqref{eq:Orb_20}, \eqref{eq:Orb_02}.
\end{proof}

\begin{remark}
We remark that since $ζ$, $z_+(z'_-)^2$ and $z_+(ζ'_-)^2$ have different biweight for the $\mfsl_2\times \mfsl_2$-action, $\ov{\mcS_2}$ is itself a cyclic $\mfA$-module generated by $ζ + z_+(z'_-)^2 + z_+(ζ'_-)^2$.
\end{remark}

\section{Structure of $\Orb(-,{\mcT_{1,1}})$}
Our approach in this section mirrors that of \Cref{sec:Orb(S2)}. Let $V = V_{1,1}  $ be a complex Hermitian space of signature $(1,1)$ with Hermitian form $\langle \cdot, \cdot \rangle$. We fix a basis $\{ e,f\}$ of $V$ with 
\[
\langle e,e\rangle = \frac14, \qquad \langle f,f\rangle = -\frac14, \qquad \langle e,f \rangle = 0.
\]
With respect to this basis, we have
\[
\lie u(V) \simeq \left\{   \begin{pmatrix} a & b \\ \overline b & d \end{pmatrix} \Big| \ a + \overline a = d + \overline d = 0  , \ b \in \mathbb C \right\}.
\]
As described in \Cref{ss:reduction_trace_0}, the relevant quadratic space is
\[
\lie{t} = \lie{t}_{1,1} := \lie{su}(V_{1,1}) \oplus  V_{1,1} = \left.\left\{  \begin{pmatrix} x & u \\- u^* \end{pmatrix} \ \right\vert \ \tr(x)= 0 \right\}
\]
with quadratic form $Q(x,u) = -\tr(x^2) + 2 \langle u, u \rangle$  .  We take the standard basis
\[
X_1 = \frac{1}{ 2} \begin{pmatrix} i & \\ & -i \end{pmatrix}, \qquad X_2 = \frac{1}{ 2} \begin{pmatrix} 0 & 1 \\ 1 & 0 \end{pmatrix}, \qquad X_3 = \frac{1}{ 2} \begin{pmatrix} 0 & i \\ -i & 0 \end{pmatrix}
\]
for $\lie {su}(V_{1,1})$, and the standard (real) basis
\begin{equation*}
u_1' = e , \qquad   u_2' = ie , \qquad u_3' = f , \qquad u_4' = if.
\end{equation*}
for $V_{1,1}$. Note that the vectors $X_1, u_1'$ and $u_2'$ are positive, while the remainder are negative. This choice of basis yields generators for the Fock model
\[
\calG :=   \calF[\lie t] = \bbC[ t_1, \tau_2, \tau_3, t_1', t_2', \tau_3', \tau_4']
\]
where as before, we label  variables $τ_2, \tau_3, \tau_3', \tau_4'$ so as to remind us that they correspond to negative basis vectors. For future use, we introduce new coordinates
\[
t = t_1, \qquad \tau_{\pm} = \tau_2 \pm i \tau_3, \qquad t'_{\pm} = t'_1 \pm i t'_2, \qquad \tau'_{\pm} = \tau'_3 \pm i \tau'_4.
\]
These are eigenvectors for the diagonal $U(1)\times U(1)\subset U(V_{1,1})$.

The Lie algebra $\mfu(V)$ and the maximal compact $U(1)\times U(1)\subset U(V)$ act on the polynomial Schwartz space $\mathcal T_{1,1}$, where the Siegel Gaussian $ψ_{1,1}$ is defined by the hermitian form $H$ with $H(e,e) = H(f,f) = 1/4$, $H(e,f) = 0$. We obtain a map $\lie{u}(V) \to \calF\{ \lie t  \} $ by applying the intertwiner $\iota \colon  \bbC\{ \lie t \} \simeq \calF\{ \lie t\} $. We extend it $\mbC$-linearly to a complex linear map $\lie{u}(V) \otimes \bbC \to \calF\{ \lie t  \} \subset \End(\mathcal G)$. The following lemma makes this map explicit, and is proved in the same way as \Cref{lem:orth ops gl2}.
\begin{lemma}
We have the following explicit formulas for the action of a set of  basis vectors of $\lie{u}(V) \otimes_{\bbR} \bbC$:
\begin{align*}
\mu_1 &:= \iota \left(  \begin{psmallmatrix}
i & \\ & i  \end{psmallmatrix} \otimes 1 \right) =  -i \left( t'_+ \frac{\partial}{\partial t'_+} - t'_- \frac{\partial}{\partial t'_-} + \tau'_+ \frac{\partial}{\partial \tau'_+} - \tau'_- \frac{\partial}{\partial \tau'_-}  \right)\\
\mu_2 &:= \iota \left( \begin{psmallmatrix} i & \\ & -i \end{psmallmatrix} \otimes1 \right)   = -i \left( 2 \tau_+ \frac{\partial}{\partial \tau_+} - 2 \tau_- \frac{\partial}{\partial  \tau_-}  +  t'_+ \frac{\partial}{\partial t'_+} - t'_- \frac{\partial}{\partial t'_-} - \tau'_+ \frac{\partial}{\partial \tau'_+} + \tau'_- \frac{\partial}{\partial \tau'_-}  \right) \\
\mu_3 &:= \iota \left(  \begin{psmallmatrix} & i \\ -i & \end{psmallmatrix} \otimes 1 + \begin{psmallmatrix} & 1 \\ 1 & \end{psmallmatrix} \otimes i \right)
=  \frac1{2 \pi }t  \tau_- - 16 \pi \frac{\partial^2}{\partial t \partial \tau_+}   + \frac{i}{4 \pi}\, t'_- \tau'_+ - 16 \pi i \frac{\partial^2}{\partial t'_+ \partial \tau'_- } \\
\mu_4 &:= \iota \left(  \begin{psmallmatrix} & i \\ -i & \end{psmallmatrix} \otimes 1 - \begin{psmallmatrix} & 1 \\ 1 & \end{psmallmatrix} \otimes i \right)
=  \frac1{2 \pi }  t \tau_+ - 16 \pi \frac{\partial^2}{\partial t \partial \tau_-} - \frac{i}{4 \pi }\, t'_+ \tau'_- + 16 \pi i \frac{\partial^2}{\partial t'_- \partial \tau'_+}.
\end{align*}
\qed
\end{lemma}

The action of the Weil representation in our coordinates is as follows:
\begin{equation}\label{eq:Weil_rep_T11}
\begin{aligned}
\omega(L,0) &= 2 \pi \frac{\partial^2}{\partial t^2} - \frac{1}{8 \pi} \tau_+\tau_- & \omega(0,L) &= 8 \pi \frac{\partial^2}{\partial t'_+ \partial t'_-} - \frac{1}{8 \pi} \tau'_+ \tau'_- \\
\omega(R,0) &= -\frac{1}{8 \pi} t^2 + 8 \pi \frac{\partial^2}{\partial \tau_+ \partial \tau_-} &
\omega(0,R) &= - \frac{1}{8 \pi} t'_+ t'_- +  8 \pi \frac{\partial^2}{\partial \tau'_+ \partial \tau'_-} \\
\omega(H,0) &=  t\frac{ \partial }{\partial t} - \tau_+ \frac{\partial }{\partial \tau_+} - \tau_-\frac{\partial}{\partial \tau_-} - \frac12
&
\omega(0,H) &= t'_+\frac{ \partial }{\partial t'_+} + t'_-\frac{ \partial }{\partial t'_-} - \tau'_+ \frac{\partial}{\partial \tau'_+} - \tau'_- \frac{\partial}{\partial \tau'_-}.
\end{aligned}
\end{equation}
Finally, we have the algebra $\lie A \subset \mbC\{\mft_V\}^{U(V)}$ from \eqref{eq:A_fat_on_unitary_side}. We consider the invariant polynomial
 \begin{align*}
r := 8 \langle u, xu \rangle.
\end{align*}
In terms of the coordinates $x_i, x_j'$ with respect to the basis   $X_1, X_2, X_3, u'_1, \dots u'_4$ above, we have
$$\begin{aligned}
r
= i x_1 \left( (x'_1)^2 +( x'_2)^2 + (x'_3)^2 + (x'_4)^2 \right) & + (x_2 - i x_3)(x_1' + i x_2') (x_3' - i x_4')\\ & - (x_2 + i x_3)(x_1'-ix_2')(x_3'+ix_4').
\end{aligned}$$
Applying the intertwiner $\iota \colon \bbC\{ \lie t\} \to \calF\{ \lie t \}$, we have the explicit expression
\begin{equation} \label{eqn:iota r explicit}
\iota(r)  = A + B  + C
\end{equation}
where
\begin{align*}
A &= -i \left( \frac1{4\pi} t + \partial_t \right) \times\\
&\quad \times \left( \frac{ t'_+ t'_- + \tau'_+\tau'_-}{16 \pi^2}  + \frac{1}{\pi}   + \frac{1}{2 \pi} \left( t'_+ \partial_{t'_+} + t'_- \partial_{t'_-} + \tau'_+ \partial_{\tau'_+} + \tau'_- \partial_{\tau'_-} \right)  + 4 \left( \partial_{t'_+} \partial_{t'_-} + \partial_{\tau'_+} \partial_{\tau'_-}\right)\right) \\[1mm]
B &= \left( \frac{\tau_-}{4 \pi} + 2 \partial_{\tau_+} \right)  \left( \frac{t'_+}{4 \pi} + 2 \partial_{t'_-} \right)  \left( \frac{\tau'_-}{4 \pi} + 2 \partial_{\tau'_+} \right)  \\[1mm]
C &= -  \left( \frac{\tau_+}{4 \pi} + 2 \partial_{\tau_-} \right)  \left( \frac{t'_-}{4 \pi} + 2 \partial_{t'_+} \right)  \left( \frac{\tau'_+}{4 \pi} + 2 \partial_{\tau'_-} \right).
\end{align*}

\begin{theorem}\label{thm:generator_T11}
The coinvariant space $\ov{\mcT_{1,1}}$, cf. \eqref{eq:def_bar_T_pq}, is a cyclic $\lie A$-module, generated by the image of the Gaussian $ψ_{1,1}$.
\begin{proof}
The Gaussian corresponds to $\mathbf 1 \in \mcG$. For an integer $n \geq 0$, let $\calG^n$ denote the subspace of $\calG$ consisting of polynomials of degree $\leq n$. Set
\[
R := \mfA \cdot \mathbf 1 + \ker\big(\mcT_{1,1}\twoheadrightarrow \ov{\mcT_{1,1}}\big) \subset \calG.
\]
and
\[
R^n := \mfA \cdot \calG^{n-1} + R
\]
with $R^0 = R$.
It suffices to show
\[
\calG^n \subset R^n
\]
for all $n$. We proceed by induction, noting that the case $n=0$ is immediate.

Let $\calG_0 := \bbC[ t, \tau_+, \tau_-]$ and $\calG_1 = \bbC[t'_+, t'_-, \tau'_+, \tau'_-]$, so that $\calG = \calG_0 \otimes \calG_1$. Now suppose $g  = g_0 g_1$ is of degree $n$ with $g_i \in \calG_i$, and further assume that
\[
g_1 = (t'_+)^{a_+} (t'_-)^{a_-} (\tau'_+)^{b_+} (\tau'_-)^{b_-}
\]
is a monomial. If $a_+ a_- \geq 1$, then
\[
g = -8 \pi \omega(0,R) \left( g/t'_+t'_- \right) + (\text{lower degree terms}) \in R^n.
\]
Similarly, if $b_+ b_- \geq 1$, then we may use $\omega(0,L)$ to conclude $g \in R^n$. We may therefore assume $a_+ a_- = b_+ b_- = 0$.

If $a_- b_+ \geq 1$, then we may write
\[
g = \frac{4 \pi}{i} \mu_3 (g/ t'_- \tau'_+) + 2i t \tau_- g/t'_- \tau'_+ + (\text{lower degree terms}).
\]
Thus, at the cost of replacing $g_0$ by a different polynomial, we may assume  that $a_- b_+ = 0$. A similar argument with $\mu_4$ allows us to further assume $a_+ b_- = 0$.

To summarize so far, we may assume that $g = g_0 g_1 $ with
\[
g_1 \in \{ (t'_+)^a (\tau'_+)^b, \ (t'_-)^c(\tau'_-)^d  \ | \ a,b,c,d \geq 0\}.
\]
Now consider the action of $\mu_1$, and note that if $g$ is an eigenvector with non-zero eigenvalue, then $g \in \ker(\mcT_{1,1}\twoheadrightarrow \ov{\mcT_{1,1}})$. If $g = g_0 g_1$ with $g_1 $ from the above list, the only choice yielding $\mu_1(g) = 0$ is the case $g_1 = 1$.

In other words, we have reduced to the case
$$g = g_0 \in \calG_0,$$
and we may as well assume that $g_0 = t^{\alpha} \tau_+^{\beta_+}\tau_-^{\beta_-}$ is itself a monomial. Using the action of $\omega(R,0)$ and $\omega(L,0)$ as above, we may reduce to the case $\alpha \leq 1$ and $\beta_+ \beta_- = 0$, so that it suffices to consider
\[
g \in \{  1, \  t, \ \tau_+^{\alpha}, \ \tau_-^{\alpha}, \ t \tau_+^{\alpha}, \ t \tau_-^{\alpha} \ | \ \alpha \geq 1  \}.
\]
All these elements are eigenvectors for $\mu_2$ and the only ones with zero eigenvalue are
\[
g= \mathbf 1, \qquad g = t.
\]
To conclude the proof, it will suffice to show that $t \in R^1$. To this end, let $r^{(1,0)} \in ω(\lie A)$ denote the component of the invariant polynomial $r = -8 \langle u, xu \rangle$ of weight $(1,0)$ with respect to the action of $\omega(H, 0)$ and $\omega(0, H)$ respectively.  Using \eqref{eqn:iota r explicit}, we see that
\[
\iota(r^{(1,0)}) = i \frac{t}{4\pi }   \left\{   \frac{1}{\pi} + \frac{1}{2 \pi} \left( t'_+ \partial_{t'_+} + t'_- \partial_{t'_-} + \tau'_+ \partial_{\tau'_+} + \tau'_- \partial_{\tau'_-} \right) \right\}  + \partial_{\tau_+} \left( \cdots \right) + \partial_{\tau_-}\left( \cdots \right).
\]
In particular,
\[
\left( \frac{i}{4 \pi^2}\right)t = \iota(r^{(1,0)}) (\mathbf 1)
\]
and hence $t \in R^1$ as required.
\end{proof}
\end{theorem}

\begin{corollary}\label{cor:main_T_11}
The space $\Orb(-,\mcT_{1,1})$ is generated by $\Orb(-,ψ_{1,1})$ as $\mfA$-module.
\end{corollary}
\begin{proof}
The map $\mcT_{1,1}\twoheadrightarrow \Orb(-,\mcT_{1,1})$ factors over $\ov{\mcT_{1,1}}$, and the latter is cyclic generated by $ψ_{1,1}$ by Theorem \ref{thm:generator_T11}.
\end{proof}

\begin{corollary} \label{cor:cyclic coinv n=2}
  \Cref{conj:cyclic_coinvariants} holds when $n = 2$.
  \begin{proof}
  	The fact that $\overline{\mathcal T_{p,q}}$ is generated by $\psi_{p,q}$ follows from \Cref{ex:cylic_T_definite} for $(p,q) = (2,0) $ or $(0,2)$, and from \Cref{thm:generator_T11} for $(p,q) = (1,1)$. As for $\overline{\mathcal S_2}$, the three generators $\Phi_{2,0}, \Phi_{1,1}, \Phi_{0,2}$ are described in \Cref{thm:generators_S2} (and the discussion thereafter), and it is easily seen from the formulas in \Cref{ss:Fock model} that the biweights of $\Phi_{p,q}$, with respect to the Weil representation, are equal to those of $\psi_{p,q}$. They are equal to $(3/2, 2)$ for $Φ_{2,0}$ and $ψ_{2,0}$, equal to $(-1/2, 0)$ for $Φ_{1,1}$ and $ψ_{1,1}$, and equal to $(3/2, -2)$ for $Φ_{0,2}$ and $ψ_{0,2}$. 
  \end{proof}
\end{corollary}

\part{Transfer for $n = 2$}
\section{Main result}
\label{s:transfer_main}
The proof of the following theorem will be given in the next two sections, see \Cref{ss:proof main thm 11} below.
\begin{theorem}\label{thm:main_signature_11}
Let $
\Phi_{1,1} :=	 (y_{12}-y_{21})φ\in \mcS_2$ be as in \eqref{eq:def_Phi_11} and let  $ ψ_{1,1}\in \mcT_{1,1}$ be the Siegel Gaussian, as before. Given the normalization of Haar measures as in \S\ref{s:matching_beta_0} below, we have
\[
\Orb( - , \Phi_{1,1}) = \sqrt 2 \Orb(-, \psi_{1,1}).
\]
\end{theorem}

As an immediate corollary, we conclude:
\begin{theorem}\label{thm:main_direct_sum}
Conjectures \ref{conj:polynomial_Schwartz_SL} and \ref{conj:cyclic_Orb}  hold when $n = 2$.
\end{theorem}
\begin{proof}
As in the proof of \Cref{conj:cyclic_coinvariants}, cf.\ \Cref{cor:cyclic coinv n=2}, we have explicit elements $\Phi_{2,0}$, $\Phi_{1,1}$, $\Phi_{0,2}$ generating the coinvariant space $\overline{\mathcal S_2}$. By Theorems \ref{thm:pos_def} and \ref{thm:main_signature_11}, $Φ_{p,q}$ and $ψ_{p,q}$ are mutual transfers (up to constant).
These assertions establish Conjecture \ref{conj:cyclic_Orb} for $n=2$.
Since transfer is compatible with $\mfA$-actions in the sense of \eqref{eq:compatib_transfer_mfA}, the three explicit transfer identities extend to the identity
$$\Orb(-,\mcS_2) = \bigoplus_{p+q = 2} \Orb(-,\mcT_{p,q}),$$
which proves Conjecture \ref{conj:polynomial_Schwartz_SL} when $n = 2$ as desired.
\end{proof}

It is hence left to prove the signature $(1,1)$ transfer identity which is the content of the rest of this article. Our proof proceeds in two steps. First, in the next section, we show that $\Orb(-,\Phi_{1,1})$ and $\Orb(-,\psi_{1,1})$ coincide (up to normalizing constant) on a hyperplane $\{\beta = 0\}$ in $\calQ_\rs(\bbR)$. Next, we show that these orbital integrals are characterized as the unique solutions to a differential equation given by the weight operator $ω(0,H)$; the restriction $\{ \beta =0\}$ serves as a boundary condition in this differential equation, forcing the orbital integrals to be equal everywhere.

\section{Matching along $\beta = 0$}
\label{s:matching_beta_0}
\subsection{Description of $\mcQ(\mbR)_\rs$ by signature}
\label{ss:geometry_Q}
We continue to use the natural coordinates as in \eqref{eq:nat_coords}. The quotient variety $\mcQ$ is identified with $\mbA^3$ via \eqref{eq:invariant_functions_SL} and we denote its three coordinates by $α$, $β$ and $γ$. The quotient map is given by
$$\mr{inv}:(y, v, w)\ \longmapsto\ (\det(y), wv, wyv).$$
We define the \emph{discriminant} of $(α, β, γ)\in \mcQ(\mbR)$ by
$$Δ(α, β, γ) := αβ^2 + γ^2.$$
We denote by $\mcQ_{\rs,(p,q)}\subset \mcQ(\mbR)$ the subset of those regular semi-simple invariants that come from $(\mfsu(V_{p,q})\times V_{p,q})(\mbR)$.

\begin{lemma}\label{lem:Q11_connected}
An invariant $(α, β, γ) \in \mcQ(\mbR)$ is regular semi-simple if and only if $Δ(α, β, γ) \neq 0$. It lies in $\mcQ_{\rs, (1,1)}$ if and only if $Δ(α, β, γ) > 0$.
\end{lemma}
\begin{proof}
Consider an invariant $(α, β, γ) \in \mcQ(\mbR)$. It has representative
\begin{equation}\label{eq:rep_for_abc}
\begin{pmatrix}
y & v\\ w &
\end{pmatrix} = \begin{pmatrix}
& 1 & β\\
-α & & γ\\
1 & &
\end{pmatrix} \in \mfgl_3(\mbR).
\end{equation}
The triple $(y, v, w)$ is regular semi-simple if and only if $\{w, wy\}$ spans $\mbR_2$ and $\{v, yv\}$ spans $\mbR^2$. The first always holds, the second if and only if $Δ = αβ^2 + γ^2\neq 0$.

We next claim that $Δ(\inv(y, v, w)) < 0$ if and only if $(y,v,w)$ is regular semi-simple and matches to signature $(2,0)$ or $(0,2)$. For the 'if' direction, we may proceed as around \eqref{eq:block_diag_general} and assume that $(y,v,w)$ is of the form
\begin{equation}\label{eq:lambda_mu_form}
\begin{pmatrix}
λ & & μ_1\\
& -λ & μ_2\\
1 & 1 &
\end{pmatrix}
\end{equation}
with $μ_1 > 0$ and $μ_2 > 0$ (signature $(2,0)$ case) or $μ_1 < 0$ and $μ_2 < 0$ (signature $(0,2)$ case). We find that
\begin{equation}\label{eq:ineq_lambda_mu}
Δ(\inv(y,v,w)) = -λ^2(μ_1 + μ_2)^2 + λ^2(μ_1-μ_2)^2 = -4λ^2μ_1μ_2 < 0.
\end{equation}
For the converse direction, assume that $Δ(\inv(y,v,w)) < 0$. This is only possible if $α = \det(y) < 0$. Hence $y$ has two real eigenvalues, so we may again assume that $(y,v,w)$ is of the form shown in \eqref{eq:lambda_mu_form}. Then \eqref{eq:ineq_lambda_mu} shows that $μ_1$ and $μ_2$ have the same sign, meaning that \eqref{eq:lambda_mu_form} comes from signature $(2,0)$ or $(0,2)$. This proves the claim and it follows that $\mcQ_{\rs, (1,1)} = \{Δ > 0\}$.
\end{proof}

\subsection{The function $\Orb(-,\Phi_{1,1})$}
We begin by explicitly writing out the relevant orbital integral on the $GL_2$ side.
Given $(\alpha,\beta,\gamma) \in \mathcal Q(\bbR)$, let
\[ Y  = Y(\alpha,\beta,\gamma)= \begin{pmatrix} y & v \\ w &  \end{pmatrix} \]
be given by \eqref{eq:rep_for_abc}. In particular, for a Schwartz function $\Phi \in \mathcal S_2$,  and a choice of Haar measure $dg$, the orbital integral is given by
\[
\Orb(\Phi)(\alpha,\beta,\gamma)  = \tilde \eta(Y) \int_{\GL_2(\bbR)}  \Phi(g^{-1} \cdot Y) \eta(g) dg = \tilde \eta(Y) \int_{\GL_2(\bbR)} \Phi \left( g^{-1} y g, \, g^{-1} v, wg \right) \eta(g) dg.
\]
with $Y= Y(\alpha,\beta,\gamma)$; here the transfer factor $\tilde \eta(Y) $ is given by  \eqref{eq:transfer_factor}.

Note that by \Cref{lem:Q11_connected}, we have
\begin{equation} \label{eqn:transfer factor gl2 explicit}
\tilde \eta (Y) = \mathrm{sgn} \det \begin{pmatrix} \beta & \gamma \\ \gamma & - \alpha \beta \end{pmatrix} =  \mathrm{sgn}( - \Delta) = \begin{cases} 1, & (\alpha, \beta, \gamma) \in \calQ_{rs, (2,0)} \text{ or } \calQ_{rs, (0,2)}  \\
-1, & (\alpha,\beta,\gamma) \in \calQ_{rs, (1,1)}
\end{cases}
\end{equation}
where $Y = Y(\alpha,\beta,\gamma)$.

Now we specialize to the relevant Schwartz function $\Phi_{1,1} \in \mathcal S_2$
appearing in \eqref{eq:def_Phi_11}. Concretely, given a general element $ Y = (y,v,w) \in \lie{s_2}$ with
\[
y = \begin{pmatrix} y_{11} & y_{12} \\ y_{21} & - y_{11} \end{pmatrix}, \qquad v = \begin{pmatrix} v_1 \\ v_2 \end{pmatrix} , \qquad w = ( w_1, w_2),
\]
we have
\[
\Phi_{1,1}(y,v,w) =   (y_{12}- y_{21})\exp \left( - 2 \pi \left[ 2 y_{11}^2 + y_{12}^2 + y_{21}^2 + v_1^2 + v_2^2 + w_1^2 + w_2^2  \right] \right).
\]
To make the orbital integral explicit, we employ the decomposition
\[
\GL_2(\mathbb R) = A N K
\]
where
\[
A = \left\{ a(t_1, t_2) = \begin{pmatrix} t_1 & \\ & t_2 \end{pmatrix} \  | \ t_1, t_2 > 0 \right\}, \qquad N = \left\{ n(b) = \begin{pmatrix} 1 & b \\ & 1 \end{pmatrix} \right\}, \qquad K = O(2).
\]
The Haar measure $dg$ decomposes as
\[
dg = \frac{dt_1}{t_1} \frac{dt_2}{t_2} \, db \, dk
\]
where we normalize $dk$ to have volume 1. Noting that $\Phi_{1,1}$ is $(O(2), \eta)$-equivariant, a short computation gives
\begin{align}
&\Orb(\Phi_{1,1}) (\alpha,\beta,\gamma)  \notag\\
& \ =   (-1)^p \int_A \int_N \int_K \Phi_{1,1}( k^{-1}n(-b) a(t_1^{-1}, t_2^{-1} )\cdot Y) \eta(k)  dk db \frac{dt_1}{t_1} \frac{dt_2}{t_2} \notag  \\
& \ = (-1)^p \int_{0}^{\infty} \int_{ 0}^{\infty} \int_{ \mathbb R} \left\{  \frac{t_1}{t_2}\alpha (1+b^2) + \frac{t_2}{t_1} \right\} \notag \\
&\qquad  	\times \exp\left( - 2\pi \left[  2 \frac{t_1^2}{t_2^2} b^2 \alpha^2 + \left( \frac{t_1}{t_2} b^2 \alpha + \frac{t_2}{t_1} \right)^2  + \frac{t_1^2}{t_2^2}\alpha^2 + t_1^2(1+b^2) + (t_1^{-1} \beta - t_2^{-1} b \gamma)^2 + t_2^{-2} \gamma^2 \right] \right) \notag  \\
& \qquad \qquad db \, \frac{dt_1}{t_1} \frac{dt_2}{t_2} \label{eqn:Phi11 explicit general}
\end{align}
where $(\alpha,\beta,\gamma) \in \calQ_{rs, (p,q)}$ and $Y = Y(\alpha,\beta,\gamma)$ in the first line.

Specializing to the locus $\{ \beta = 0 \} \cap \calQ_\rs(\bbR) \subset \calQ_{rs, (1,1)}$ simplifies things considerably: setting $\beta =0$ and applying the change of variables  $t_2 \mapsto t_1 t_2$ and $t_1 \mapsto (1+b^2)^{-\frac12} t_1$, and rearranging gives
\begin{align*}
&\Orb(\Phi_{1,1})(\alpha, 0, \gamma) \\
& \ = -  \int \left\{  t_2^{-2} \alpha(1+b^2) + 1    \right\}   \exp \left( -2 \pi \left[ t_2^{-2}    (1+b^2)^2 (\alpha^2 + t_1^{-2}\gamma^2)   + t_2^2 + 2 b^2 \alpha  + t_1^2  \right] \right) \frac{dt_1}{t_1} dt_2 db.
\end{align*}

In general if $A,B>0$ then
\begin{equation} \label{eqn:Gaussian master thm}
\int_0^{\infty} e^{- 2\pi (Au^2 + Bu^{-2})}du = \frac{1}{2\sqrt{2A}} e^{-4 \pi \sqrt{AB}}, \qquad \int_0^{\infty} u^{-2} e^{-2\pi(Au^2 + B u^{-2})} du = \frac1{2 \sqrt{2B}} e^{-4 \pi \sqrt{AB}}
\end{equation}
We apply this to integrate over $t_2$  with $A = 1, B = (1+b^2)^2(\alpha^2 + t_1^{-2} \gamma^2)$, and obtain
\begin{align*}
-& 2 \sqrt 2\Orb(\Phi_{1,1})(\alpha, 0, \gamma) \\
& =  \int_{0}^{\infty} \int_{ \mathbb R}  \left\{1 + \frac{\alpha}{ \sqrt{\alpha^2 + t_1^{-2} \gamma^2 }}    \right\} \, \exp \left( -2 \pi \left[2 (1+b^2) (\alpha^2 + t_1^{-2}\gamma^2)^{1/2} + t_1^{2} + 2 b^2\alpha \right] \right)db \,  \frac{dt_1}{t_1}.
\end{align*}
Now apply the substitution $u = \sqrt{\alpha^2 + t_1^{-2} \gamma^2}$, so that
\[
du = - u^{-1} (u^2 - \alpha^2) \frac{dt_1}{t_1}.
\]
We obtain
\begin{align*}
-2 \sqrt 2 & \Orb(\Phi_{1,1})(\alpha,0,\gamma)  = \int_{|\alpha|}^{\infty} \int_{\bbR} \frac{1}{u-\alpha} \exp \left( - 2 \pi \left[ 2 (1+b^2) u + \gamma^2(u^2-\alpha^2)^{-1} + 2b^2 \alpha \right] \right) db \, du.
\end{align*}
Now the integral on $b$ is a simple Gaussian integral: evaluating it yields
\begin{equation} \label{eqn:GL2 orb b=0 final}
-4 \sqrt 2\Orb(\Phi_{1,1} ) (\alpha, 0, \gamma) =  \int_{|\alpha|}^{\infty} \frac{1}{(u-\alpha)(u+\alpha)^{\frac12}} \exp \left( - 2 \pi \left[ 2u + \gamma^2(u^2 - \alpha^2)^{-1} \right] \right) \, du.
\end{equation}

\subsection{The function $\Orb(-,\psi_{1,1})$}
Turning to the unitary side, let $V = V_{1,1}$ denote the complex Hermitian space of signature $(1,1)$, with Hermitian form
$\langle \cdot, \cdot \rangle$. It will be convenient to fix a hyperbolic basis $\{ e', f'\}$
where
\[
\langle e', e' \rangle = \langle f', f' \rangle = 0, \qquad \langle e', f' \rangle = -i.
\]
With respect to this basis, we have
\[
\lie{su}(V) = \lie{sl}_2(\bbR) = \left\{  \begin{pmatrix} a & b \\ c & -a \end{pmatrix} \ | \ a,b,c \in \mathbb R\right\}.
\]
The Siegel Gaussian $\psi_{1,1} \in \mathcal T_{1,1}$ then has the explicit expression
\begin{equation}  \label{eqn:psi11 hyberbolic}
\psi_{1,1}(x,u) = \exp\left( - 2 \pi \left[2 x_{11}^2 + x_{12}^2 + x_{21}^2  + 2|u_1|^2 + 2 |u_2|^2 \right]  \right)
\end{equation}
where we write
\[
x = \begin{pmatrix} x_{11} & x_{12} \\ x_{21} & -x_{11} \end{pmatrix} \in \lie{su}(V), \qquad  u = \begin{pmatrix} u_1 \\ u_2 \end{pmatrix} \in V
\]
with respect to the basis $\{ e', f'\}$.

Note that on the one hand, we have $\psi_{1,1}(x,u) = \psi_{1,1}(-x,u)$. On the other hand, if the invariants of $(x,u)$ are $(\alpha,\beta,\gamma)$ respectively, then the invariants of $(-x,u)$ are $(\alpha, \beta, -\gamma)$. In particular, we have
\[
\Orb(\psi_{1,1})(\alpha, \beta, \gamma) = \Orb(\psi_{1,1})(\alpha, \beta, -\gamma).
\]
and we may therefore assume without loss of generality that $\gamma \geq 0$ in the sequel.

Now suppose we are given a triple $(\alpha, \beta, \gamma) \in \calQ_{rs, (1,1)}$ of invariants, so that$\Delta := \alpha\beta^2 + \gamma^2>0$, cf.\ \Cref{lem:Q11_connected}, and assume further that $\gamma>0$. We define an element $X = X(\alpha,\beta,\gamma) \in \lie{su}(V) \oplus V$ by the formula
\begin{equation} \label{eqn:U11 section}
X :=  (x,u) = \left(\begin{pmatrix} 0 & x_2 \\ x_3 & 0 \end{pmatrix} , \  \begin{pmatrix} -i \beta / 2  \\ 1 \end{pmatrix} \right) \ \in \ \lie{su}(V) \oplus  V
\end{equation}
where
\begin{equation} \label{eqn:x2 x3 def}
x_2 :=  - \frac12 (\gamma + \sqrt{\Delta}), \qquad  x_3 := -2 \alpha/( \gamma + \sqrt{\Delta}).
\end{equation}
Specializing \eqref{eq:invariant_functions_SU} to the case at hand yields the three invariant functions $\left( - \det x, \langle u, u \rangle, -i \langle u, xu \rangle \right)$, and applying these invariants to $X = (x,u)$ as defined above recovers the triple $(\alpha,\beta,\gamma)$; in other words, \eqref{eqn:U11 section} is a section of the quotient map $\lie{su}(V) \oplus V \to \mathcal Q_{(1,1)}$ defined on the locus $\{ \gamma >0 \} \cap \calQ_{rs,(1,1)}$.

We can also make the orbital integral more explicit by considering the decomposition
\[
U(V)(\mathbb R) = A N K
\]
where (in terms of the basis $\{ e', f'\}$ as above), we set
\[
A := \left\{ a(t) =\begin{pmatrix}t & \\ & t^{-1} \end{pmatrix} \ | \ t \in \mathbb R, t>0 \right\}, \qquad N = \left\{ n(b) = \begin{pmatrix} 1 & b \\ & 1 \end{pmatrix} \ | \ b \in \mathbb R \right\}
\]
and
\[ K = \left\{  e^{i \theta} \begin{pmatrix}  \cos \phi & \sin \phi \\ - \sin \phi & \cos \phi \end{pmatrix}  \right\} \simeq U(1) \times U(1).	\]
The corresponding decomposition of measures is given by $dg = \frac{dt}t db dk$, where we normalize $dg$ such that $dk$ is the Haar measure on $K$ of total volume 1.

A brief calculation then shows that for $\gamma>0$, we have the explicit expression
\begin{align}
&\Orb(\psi_{1,1})(\alpha,\beta,\gamma) \notag \\
& \ = \int_{A} \int_{N} \int_K \psi_{1,1} \left(  k^{-1} n(-b) a(t^{-1}) \cdot X \right) \, dk \,  db \, \frac{dt}{t}  \notag \\
&\ = \int_0^{\infty} \int_{\mathbb R} 		 e^{ -2\pi [ 2 b^2 t^4 x_3^2 + (b^2 t^2 x_3 -t^{-2} x_2)^2  + t^4 x_3^2 + 2( t^2 b^2 + t^{-2} \beta^2/4 + t^2) ] }  db  \frac{dt}t  . \label{eqn: unitary int}
\end{align}

For future purposes, it will be  useful to further restrict to the case $\alpha \neq 0, \gamma >0$.
On this region, the functions $x_2$ and $x_3$ defined in \eqref{eqn:x2 x3 def} are everywhere non-zero. More precisely, if $\alpha>0$ then $x_2, x_3 <0$ and if $\alpha <0$, then $x_3> 0 > x_2$.

Rearranging the terms in the exponent in \eqref{eqn: unitary int} above and applying the substitution $u = 1+b^2$, we find
\begin{align*}
\Orb&(\psi_{1,1})(\alpha,\beta,\gamma)   \\
&= \int_0^{\infty} \int_{\mathbb R} 		 \exp \left(-2\pi [  t^4 x_3^2( 1+ b^2)^2 + 2 t^2(b^2 + 1)   + t^{-2}\beta^2 / 2 + t^{-4} x_2^2 -2 b^2 \alpha ] \right)  \, db \, \frac{dt}t  \\
&=     \int_0^{\infty} \int_{1}^{\infty} 		 \exp \left(-2\pi [  t^4 x_3^2 u^2 + 2 t^2u   + t^{-2}\beta^2 / 2 + t^{-4} x_2^2 -2 u \alpha + 2 \alpha ] \right)  \,\frac{du}{\sqrt{u-1}}\, \frac{dt}t  \\
&=   \int_0^{\infty} \int_{1}^{\infty} 		 \exp \left(-2\pi [ ( t^2 x_3 u -t^{-2} x_2)^2 + 2 t^2u   + t^{-2}\beta^2 / 2   + 2 \alpha ] \right)  \,\frac{du}{\sqrt{u-1}}\, \frac{dt}t  .
\end{align*}
Next we apply the substitution
\[ t \mapsto t \left| \frac{x_2}{x_3} \right|^{\frac14} = t \frac{(\gamma + \sqrt{\Delta})^{\frac12}}{\sqrt 2|\alpha|^{\frac14}} \]
to obtain
\begin{equation}
\Orb(\psi_{1,1})(\alpha,\beta,\gamma) =     \int_0^{\infty} \int_{1}^{\infty} e^{- 2 \pi h(t,u)}  \,\frac{du}{\sqrt{u-1}}\, \frac{dt}t \label{eqn:unitary orb int again}
\end{equation}
where
\begin{equation} \label{eqn:U11 orb int exponent}
h(t,u) = h_{\alpha,\beta,\gamma}(t,u) :=  |\alpha|( t^2  u - \mathrm{sgn}(\alpha)t^{-2} )^2 + \frac{t^2u(\gamma+ \sqrt{\Delta})}{|\alpha|^\frac12}   + \frac{t^{-2} \beta^2 |\alpha|^{\frac12}}{\gamma+ \sqrt{\Delta}}   + 2 \alpha.
\end{equation}

\subsection{Matching along $\beta = 0$}
The main result of this section is the following proposition. 
 Key steps in its proof were suggested by Google Gemini.
\begin{proposition} \label{prop:beta = 0 match}
Suppose $\gamma \neq 0$ and $\alpha \in \mathbb R$, so that $(\alpha, 0, \gamma) \in \mathcal Q_{rs,(1,1)}$. With the normalizations of Haar measures given in the previous two subsections, we have
\[
\Orb(\Phi_{1,1})(\alpha, 0, \gamma)  =  - \frac1{\sqrt{2}} \, \Orb(\psi_{1,1})(\alpha, 0, \gamma).
\]
\end{proposition}
\begin{proof}
 We have already observed that $\Orb(\psi_{1,1})$ is invariant under $\gamma \mapsto - \gamma$. Similarly, we note that $\Phi_{1,1}(-y,-v,-w) = - \Phi_{1,1}(y,v,w)$, and if $\inv(y,v,w)= (\alpha,\beta,\gamma) $, then $\inv(-y,-v,-w) =  (\alpha,\beta, -\gamma)$. Since $\tilde \eta(-Y) = - \tilde \eta(Y)$ for any $Y \in \lie{gl}_3(\bbR)$, we conclude that $\Orb(\Phi_{1,1})$ is also invariant under $\gamma \mapsto - \gamma$. Moreover, by continuity, it suffices to show the desired identity holds when $\alpha \neq 0$. We may therefore assume that $\gamma > 0$ and $\alpha \neq 0$ in the remainder of the proof.

Starting with the unitary side, we take $\beta = 0$  (which implies $\sqrt{\Delta } = \gamma$) in \eqref{eqn:unitary orb int again} to obtain
\begin{align*}
  \Orb(\psi_{1,1})(\alpha,0,\gamma) = \int_0^{\infty} \int_{1}^{\infty} 		 \exp \left(-2\pi \left[ |\alpha| ( t^2  u - sgn(\alpha)t^{-2} )^2 + \frac{t^2u(\gamma+ \sqrt{\Delta})}{|\alpha|^\frac12}  + 2 \alpha \right] \right)  \,\frac{du}{\sqrt{u-1}}\, \frac{dt}t.
\end{align*}
Suppose first that $\alpha >0$. We apply the substitution
\[
A = t^2 u + t^{-2}, \qquad B = t^2u - t^{-2},
\]
with
\[
\frac{dA \, dB}{2\sqrt{A^2 - B^2 - 4} }  =\frac{dt\, du}{t\sqrt{u-1}}
\]
to yield
\begin{align*}
2 &\Orb(\psi_{1,1})(\alpha,0,\gamma)  \\ &= \int_{-\infty}^{\infty} \int_{\sqrt{B^2 + 4}}^{\infty} \exp\left( - 2 \pi \left[ \alpha  B^2 +  \frac{\gamma}{\sqrt \alpha } (A+B) + 2 \alpha \right] \right) \frac{dA}{\sqrt{A^2-B^2 - 4}} \, dB \\
&=  \int_{-\infty}^{\infty}    \exp\left( - 2 \pi \left[ \alpha  B^2 +  \frac{\gamma}{\sqrt \alpha } B + 2 \alpha \right] \right)
\left\{  \int_{\sqrt{B^2 + 4}}^{\infty}  e^{- 2 \pi \frac{\gamma}{\sqrt{\alpha} }A} \frac{dA}{\sqrt{A^2-B^2 - 4}}  \right\}  \, dB
\end{align*}
To compute the inner integral, we write
\[
A = \frac{\sqrt{B^2 + 4} }{2}(v + v^{-1}) \implies \frac{dA}{\sqrt{A^2 - B^2 - 4}} = \frac{dv}v
\]
to obtain
\begin{align*}
\int_{\sqrt{B^2 + 4}}^{\infty}  e^{- 2 \pi \frac{\gamma}{\sqrt{\alpha} }A} \frac{dA}{\sqrt{A^2-B^2 - 4}}   &= \int_1^{\infty}e^{- \pi \frac{\gamma \sqrt{B^2 + 4}}{\sqrt \alpha}(v + v^{-1}) } \frac{dv}{v} \\
&= \frac12 \int_0^{\infty}e^{- \pi \frac{\gamma \sqrt{B^2 + 4}}{\sqrt \alpha}(v + v^{-1}) } \frac{dv}{v} \\
&= \frac12 \int_0^{\infty} e^{-  \pi \left( v + \frac{\gamma^2(B^2 + 4)}{\alpha}v^{-1} \right) } \frac{dv}v.
\end{align*}
Substituting, interchanging the order of the $B$ and $v$ integrals, we obtain
\begin{align*}
4 \, \Orb(\psi_{1,1})(\alpha, 0, \gamma) &= \int_0^{\infty} e^{- \pi (v + 4\alpha + 4 \gamma^2/ \alpha v )}  \left\{   \int_{-\infty}^{\infty}e^{- 2 \pi \left[    \left(  \alpha + \gamma^2/2\alpha v \right) B^2 + \frac{\gamma}{\sqrt \alpha} B   \right] } dB \right\} \frac{dv}v \\
&=\int_0^{\infty} e^{-  \pi (\frac{ \gamma^2}{\alpha u} + 4\alpha +4 u )}  \left\{   \int_{-\infty}^{\infty}e^{- 2 \pi \left[    \left(  \alpha +  u/2 \right) B^2 + \frac{\gamma}{\sqrt \alpha} B   \right] } dB \right\} \frac{du}u
\end{align*}
The inner integral is a standard Gaussian integral in $B$, which we evaluate to obtain
\begin{align*}
4  \,  \Orb(\psi_{1,1})(\alpha, 0, \gamma) &=   \int_0^{\infty}\frac1{u(2\alpha + u)^{\frac12}} e^{ - \pi \left( \frac{ \gamma^2}{\alpha u} + 4\alpha +4 u - \frac{\gamma^2 }{2\alpha(u/2+{\alpha})}\right)   } du \\
&= \int_0^{\infty}\frac1{u(2\alpha + u)^{\frac12}} e^{ - \pi \left(  4(\alpha + u) + \frac{2\gamma^2 }{u(u+2{\alpha})}\right)   } du.
\end{align*}
Applying the substition $u \mapsto u - \alpha$ and comparing with \eqref{eqn:GL2 orb b=0 final} yields the result.

Now suppose $\alpha <0$. Taking $A = t^2 u + t^{-2}$ and $ B = t^2u - t^{-2}$ as before, we find
\begin{align*}
2 \, &\Orb(\psi_{1,1})(\alpha,0,\gamma)  \\ &= \int_{-\infty}^{\infty} \int_{\sqrt{B^2 + 4}}^{\infty} \exp\left( - 2 \pi \left[ |\alpha|  A^2 +  \frac{\gamma}{|\alpha|^{\frac12}} (A+B) + 2 \alpha \right] \right) \frac{dA}{\sqrt{A^2-B^2 - 4}} \, dB \\
&= \int_2^{\infty}   \exp\left( - 2 \pi \left[ |\alpha|A^2 + \frac{\gamma}{|\alpha|^{\frac12}} A + 2 \alpha \right] \right) \left\{  \int_{- \sqrt{A^2 - 4 } }^{\sqrt{A^2 -4}} \exp \left( \frac{- 2 \pi \gamma B }{|\alpha|^{\frac12}}   \right) \frac{dB}{\sqrt{A^2 - B^2 - 4}} \right\} dA \\
&= \int_2^{\infty} \exp\left( - 2 \pi \left[ |\alpha|A^2 + \frac{\gamma}{|\alpha|^{\frac12}} A + 2 \alpha \right] \right) \left\{  \int_0^{\pi}  \exp \left( \frac{- 2 \pi \gamma\sqrt{A^2 - 4}}{|\alpha|^{\frac12} } \cos\theta \right) d\theta   \right\} dA \\
&= \pi \int_2^{\infty} \exp\left( - 2 \pi \left[ |\alpha|A^2 + \frac{\gamma}{|\alpha|^{\frac12}} A + 2 \alpha \right] \right) I_0 \left( \frac{ 2 \pi \gamma\sqrt{A^2 - 4}}{|\alpha|^{\frac12} }   \right)   dA
\end{align*}
where we use the integral representation for the Bessel $I$-function, cf.\ \cite[10.32.1]{NIST:DLMF}.

Applying the substitution $u = \sqrt{A^2 - 4}$, we obtain
\begin{align*}
(2  / \pi )&\Orb(\psi_{1,1})(\alpha, 0, \gamma) \\
&= \int_0^{\infty} \exp \left( - 2 \pi \left[ |\alpha|(u^2 + 4) + \frac{\gamma}{|\alpha|^{\frac12}} (u^2 + 4)^{\frac12} + 2 \alpha \right] \right) I_0 \left( \frac{ 2 \pi \gamma u}{|\alpha|^{\frac12}} \right)   \frac{u}{(u^2 + 4)^{\frac12}} du.
\end{align*}
Now we write
\[
\frac{1}{(u^2 + 4)^{\frac12}} \exp \left( - 2 \pi \gamma(u^2 + 4)^{\frac12}/ |\alpha|^{\frac12}  \right) = 2 \sqrt 2 \int_0^{\infty} e^{-2\pi \left[  (u^2 + 4)t^2 + \frac{\gamma^2}{4|\alpha| t^2} \right] } dt.
\]
Substituting this expression and rearranging , we obtain
\[
\frac{1}{\sqrt{2} \, {\pi}} \Orb(\psi_{1,1})(\alpha, 0, \gamma) =  \int_0^{\infty} e^{ - 2 \pi \left( \frac{\gamma^2}{4 |\alpha| t^2}  + 4 t^2 - 2 \alpha \right) } \left\{ \int_0^{\infty} e^{-2 \pi (t^2 + |\alpha|)  u^2} I_0 \left( \frac{ 2 \pi \gamma u}{|\alpha|^{\frac12}} \right)   u  \, du \right\}   \, dt.
\]
By \cite[(10.43.23)]{NIST:DLMF}, we have
\[
\int_0^{\infty} e^{-2 \pi (t^2 + |\alpha|)  u^2} I_0 \left( \frac{ 2 \pi \gamma u}{|\alpha|^{\frac12}} \right)   u  \, du = \frac{1}{4 \pi (t^2+|\alpha|)} \exp \left( \frac{ \pi \gamma^2}{2 |\alpha|(t^2+|\alpha|)} \right)
\]
We therefore have
\begin{align*}
2 \sqrt 2  \, \Orb(\psi_{1,1}) (\alpha, 0, \gamma) = \int_0^{\infty} (t^2 + |\alpha|)^{-1}  \exp \left( - 2 \pi \left[ 4t^2 - 2 \alpha + \frac{\gamma^2}{4|\alpha|} (t^{-2} - (t^2 + |\alpha|)^{-1}) \right] \right)  dt.
\end{align*}
Substituting $u = 2 t^2 + |\alpha|$ and comparing with \eqref{eqn:GL2 orb b=0 final} gives the result.
\end{proof}

\section{Differential equations for orbital integrals}

\subsection{Unique solutions for the differential equation defined by $ω(0,H)$}

We begin by calculating the restriction $ω(0,H)\vert_\mcQ$.

\begin{lemma}\label{lem:restriction_H_2}
The restriction $ω(0,H)\vert_{\mcQ}$ is given by
\begin{equation}\label{eq:H2_restriction}
\frac{1}{4π}\left(16π^2β - 2\partial_β - β \partial_β^2 + αβ\,\partial_γ^2 - 2γ\,\partial_β\partial_γ\right).
\end{equation}
\end{lemma}
\begin{proof}
We use the coordinates from \eqref{eq:nat_coords}. The operator $ω(0,H)$ is defined by choosing the standard basis \eqref{eq:standard_basis_R_2} for $\mbR^2\times \mbR_2$ and applying the formulas in \eqref{eq:Weil_Lie_algebra}, \eqref{eq:LRH}. Substituting back to the natural coordinates $v_1$, $v_2$, $w_1$ and $w_2$, we have\footnote{As a quick check, one immediately sees that this operator indeed annihilates $(y_{12}-y_{21})φ$ which has $(\mbR^2\times \mbR_2)$-factor $e^{-2π(v_1^2 + v_2^2 + w_1^2 + w_2^2)}$, compare \eqref{eq:Siegel_GL}.}
\begin{equation}\label{eq:omega_2_natural}
ω(0,H) = \underset{=\,4πβ}{\underbrace{4π(v_1w_1 + v_2w_2)}} - \frac{1}{4π}\Big(\partial_{v_1}\partial_{w_1} + \partial_{v_2}\partial_{w_2}\Big)
\end{equation}
The formula shows that $ω(0,H)$ is a second order differential operator. By Lemma \ref{lem:restrict_diff_op}, the restriction $ω(0,H)\vert_{\mcQ}$ is of second order as well. That is, it can be written in the form
\begin{equation}
\begin{aligned}
p_1 & + p_α\partial_α + p_β \partial_β + p_γ\partial_γ\\
& + p_{αα}\partial_α^2 + p_{ββ}\partial_β^2 + p_{γγ}\partial_γ^2\\[1mm]
& + p_{αβ}\partial_α\partial_β + p_{αγ}\partial_α\partial_γ + p_{βγ}\partial_β\partial_γ
\end{aligned}
\end{equation}
for unique polynomials $p_1,\ldots,p_{βγ}\in \mbC[α, β, γ]$. These are uniquely determined by the values $ω((0,H)\vert_{\mcQ})(\mfm)$ where $\mfm$ runs through the monomials of degree $\leq 2$ in $\{α, β, γ\}$. Let $ρ$ denote the inclusion map $\mbC[α, β, γ] \hookrightarrow \mbC[\mfsl_2\times \mbA^2\times (\mbA^2)^t]$. That is,
$$ρ(α) = \det(y),\quad ρ(β) = wv,\quad ρ(γ) = wyv.$$
With $\mcD$ denoting the operator in \eqref{eq:H2_restriction}, we ultimately need to check
$$ρ(\mcD(\mfm)) = ω(0,H)(ρ(\mfm))$$
for all monomials $\mfm$ of degree $\leq 2$. This is mechanical; for example,
$$ρ(\mcD(1)) = 4π(v_1w_1 + v_2w_2) = ω(0,H)(1)$$
and
$$\begin{aligned}
ρ(\mcD(β)) & = 4π(v_1w_1 + v_2w_2)^2 - \frac{1}{2π}\\
& = ω(0,H)(v_1w_1 + v_2w_2).
\end{aligned}$$
The remaining monomials can easily be checked by computer calculation. We omit further details.
\end{proof}

Let $(α, β)\in \mbR^2$. Then $\{(α,β)\}\times \mbR$ is contained in $\mcQ_{\rs, (1,1)}$ if $αβ > 0$ or intersects it in the union of the two open rays $(-\infty, -\sqrt{|αβ^2|})\cup (\sqrt{|αβ^2|}, \infty)$ when $αβ \leq 0$. We say that a function $f\in \mcC^\infty(\mcQ_{\rs, (1,1)})$ has \emph{rapid decay in $γ$} if for all $(α,β)$ with $α\neq 0$, $f(α, β, γ)$ decays faster than the inverse of any polynomial in $γ$ as $|γ|\to \infty$.

\begin{proposition}\label{prop:unique_extension}
Let $f_0(α, 0, γ)$ be a smooth function on $\mcQ_{\rs,(1,1)}\cap \{β = 0\}$. There exists at most one function $f \in \mcC^\infty(\mcQ_{\rs,(1,1)})$ with the following properties:
\begin{itemize}
\item[(a)] $f(α,0,γ) = f_0(α, 0, γ)$ for all $α$ and $γ$,
\item[(b)] $ω(0,H)\vert_{\mcQ}(f) = 0$,
\item[(c)] for all $k\geq 1$, the iterated partial derivative $\partial_β^k(f)$ have rapid decay in $γ$, and
\item[(d)] $f(α, -, γ)$ is analytic for all fixed $α$ and $γ$ with $αγ\neq 0$.
\end{itemize}
\end{proposition}
\begin{proof}
The difference $δ = f - f'$ of two solutions $f, f'$ to the given differential equation is a smooth function that satisfies (b), (c), (d) and $δ(α, 0, γ) \equiv 0$. So in the following, we assume that $f$ satisfies (b), (c) and (d) with $f_0(α, 0, γ) \equiv 0$ and show that $f = 0$.

As $f$ is continuous by assumption, it suffices to show $f(α,-,γ) \equiv 0$ whenever $αγ \neq 0$. So we assume $αγ \neq 0$ in the following. The intersection
\begin{equation}\label{eq:interval_11}
(\{α\}\times \mbR\times \{γ\}) \cap \mcQ_{\rs, (1,1)}
\end{equation}
equals $\mbR$ if $α>0$ and $(-|γ|/\sqrt{|α|},\,\, |γ|/\sqrt{|α|})$ if $α < 0$. In particular, it is always an open interval that contains the point $(α, 0, γ)$. By assumption, the restriction of $f$ to \eqref{eq:interval_11} is analytic. In order to prove its vanishing, it hence suffices to show that its Taylor series in $β = 0$ vanishes. In other words, our task is to prove that all directional derivatives $f^{(k)}(α,γ) := \partial_β^k(f)(α,0,γ)$ vanish.

We fix $k \geq 1$ and assume by induction that we already know $f^{(i)} = 0$ when $i < k$, the statement $f^{(0)} = 0$ corresponding to our assumption $f(α, 0, γ) = 0$. Note that $\partial^m_β (βf)\vert_{\{β = 0\}} = m f^{(m-1)}$ for all $m\geq 0$, our convention being that $0\cdot f^{(-1)} = 0$. Going through \eqref{eq:H2_restriction} term by term, this gives
\begin{equation}
\begin{aligned}
0 & = 4π[\partial_β^{k-1} ω(0, H)\vert_{\mcQ}](f)\vert_{\{β = 0\}}\\[1mm]
& = 16π^2(k-1)f^{(k-2)} - 2 f^{(k)} - (k-1)f^{(k)} + (k-1) α\partial_γ^2 f^{(k-2)} - 2γ\partial_γ f^{(k)}.
\end{aligned}
\end{equation}
Since $(k-1)f^{(k-2)} = 0$ by induction hypothesis (if $k = 1$ our convention applies instead), this leads to the Euler type differential equation
$$γ\partial_γf^{(k)} = - \frac{k+1}{2}f^{(k)}$$
which implies
$$f^{(k)}(α, λγ) = λ^{-(k+1)/2} f^{(k)}(α,γ)$$
for all $λ \geq 1$. Since $\partial^k_β(f)$ has rapid decay in $γ$ by assumption (c), the only possibility is $f^{(k)} = 0$ as was to be shown.
\end{proof}

\begin{lemma}\label{lem:analytic_properties_simple}
The functions $\Orb(Φ_{1,1})$ and $\Orb(ψ_{1,1})$ satisfy assumptions (b) and (c) of Proposition \ref{prop:unique_extension}.
\end{lemma}
\begin{proof}
In general, any Siegel Gaussian on a real quadratic space of signature $(p,q)$ has weight $(p-q)/2$ for the Weil representation. On the Fock model side, this can be seen from the formula for $ω(H)$ in \eqref{eq:LRH_Fock}, for which $1$ is an eigenvector of weight $(p-q)/2$. On could also express $ω(H)$ in terms of \eqref{eq:Weil_Lie_algebra} and apply it directly to the Siegel Gaussian in \eqref{eq:Siegel Gaussian general}. In the setting of the lemma, the $V_{1,1}$-component of $ψ_{1,1}$ and the $\mbR^2\times \mbR_2$-component of $Φ_{1,1}$ are given by Siegel Gaussians. The signatures of $V_{1,1}$ (as real quadratic space) and $\mbR^2\times \mbR_2$ are $(2,2)$, so these components have weight $0$ and are annihilated by $ω(0,H)$. We then obtain (b) from Proposition \ref{prop:diff_op_orb_int_generic}.

Property (c) is completely general and follows from \cite[Lemma 5.5]{Xue}. The argument is the same on the $\GL_2$- and unitary side, so we only present the general linear case. Let $Φ$ be any Schwartz function on $\mfsl_2(\mbR)\times \mbR^2 \times \mbR_2$ and fix $α \neq 0$. In the terminology of \cite[Lemma 5.5]{Xue}, $Y(α, β, γ)$ is strongly regular semi-simple because $D(X) = \det\left(\begin{smallmatrix} & 1 \\ -α & \end{smallmatrix}\right) = α$ is non-zero. The term $\{1, |D(X)|^{-1/2}\}$ is then just a constant. Applying the cited lemma to the products $γ^kΦ$ for $k\geq 1$ and noting that $Δ\sim γ^2$ as $|γ|\to \infty$, we obtain the rapid decay in $γ$ of $\Orb(Φ)$.

Condition (c) also requires rapid decay in $γ$ for all iterated partial derivatives $\partial_β^k\Orb(Φ)$. Given $k$, by Lemma \ref{lem:restriction_surjective} there exist a power $m$ and an invariant differential operator $\mcD\in \mbC\{\mfsl_2\times \mbA^2\times (\mbA^2)^t\}^{\GL_2}$ such that $\mcD\vert_\mcQ = Δ^m\partial_β^k$. Then $Δ^m\partial_β^k \Orb(Φ) = \Orb(\mcD(Φ))$ has rapid decay in $γ$ by the previous argument, showing that also $\partial_β^k\Orb(Φ)$ has rapid decay in $γ$.
\end{proof}

\begin{proposition}\label{prop:analytic_U}
Fix $\alpha , \gamma \neq 0$. The function $\beta \mapsto \Orb(ψ_{1,1})(\alpha,\beta,\gamma)$ is real analytic in $\beta$, i.e.\ it satisfies part (d) of \Cref{prop:unique_extension}.
\begin{proof}
As in the proof of \Cref{prop:beta = 0 match}, we may assume without loss of generality that $\gamma > 0$.
Recall from \eqref{eqn:unitary orb int again}, we had
\[
\Orb(\psi_{1,1})(\alpha,\beta,\gamma) =     \int_0^{\infty} \int_{1}^{\infty} 		 \exp \left(-2\pi h(t,u) \right)  \,\frac{du}{\sqrt{u-1}}\, \frac{dt}t
\]
with
\begin{equation*}
h(t,u) = h_{\alpha,\beta,\gamma}(t,u) :=  |\alpha|( t^2  u - \mathrm{sgn}(\alpha)t^{-2} )^2 + \frac{t^2u(\gamma+ \sqrt{\Delta})}{|\alpha|^\frac12}  + \frac{t^{-2} \beta^2 |\alpha|^{\frac12}}{\gamma+ \sqrt{\Delta}}   + 2 \alpha.
\end{equation*}
Using the relation $ \beta^2 = \alpha^{-1} (\Delta - \gamma^2)$, a little rearranging gives
\[
h =  |\alpha| ( t^2  u - \mathrm{sgn}(\alpha)t^{-2} )^2  +  \frac{\sqrt{\Delta}}{|\alpha|^{\frac12}} (t^2 u + \mathrm{sgn}(\alpha) t^{-2}) + \frac{\gamma}{|\alpha|^{\frac12}} (t^2 u - \mathrm{sgn}(\alpha)t^{-2}) + 2 \alpha.
\]
Now given $\alpha, \gamma$ with $\alpha\gamma \neq 0$ and $\gamma>0$, fix a point $\beta^* \in \mathbb R$ such that $(\alpha, \beta^*, \gamma) \in \mathcal Q_{\rs, (1,1)}$.  It will suffice to show that there exists a complex ball $B_{r}(\beta^*)\subset \bbC$ of radius $r>0$ around $\beta^*$ such that $	\Orb(\psi_{1,1})(\alpha, \beta^*, \gamma) $ extends to a holomorphic function $\Orb(\psi_{1,1})(\alpha, \beta_{\mathbb C}, \gamma)$ for $\beta_{\bbC} \in B_{r}(\beta^*)$.

Noting that $\exp(- \pi h_{\alpha, \beta_{\mathbb C}, \gamma}(t,u))$ defines an entire function of $\beta_{\bbC}$ via the same formula \eqref{eqn:U11 orb int exponent}, it suffices to show that there exists $r>0$ such that $\exp(- \pi h_{\alpha, \beta_{\mathbb C}, \gamma})$ is absolutely bounded by an integrable function of $b$ and $u$, uniformly on $B_{r}(\beta^*)$.

More precisely, we consider $\Delta = \alpha \beta^2 + \gamma^2$ as a complex function of $\beta$. Since $\Delta|_{\beta = \beta^*} > 0$, the square root $\sqrt{\Delta}$ defines a holomorphic function on $B_r(\beta^*)$ for sufficiently small $r$. Moreover, we may choose $r$ such that  $B_r(\beta^*) \cap \bbR^3 \subset \calQ_{\rs, (1,1)} $ and
\[
Re(\sqrt{\Delta}) > C
\]
for some constant $C>0 $ uniformly on $B_r(\beta^*)$.

In particular, if $\alpha >0$, we write
\begin{align*}
h &=  |\alpha| ( t^2  u - \mathrm{sgn}(\alpha)t^{-2} )^2  +  \frac{\sqrt{\Delta}}{|\alpha|^{\frac12}} (t^2 u + \mathrm{sgn}(\alpha) t^{-2}) + \frac{\gamma}{|\alpha|^{\frac12}} (t^2 u - \mathrm{sgn}(\alpha)t^{-2}) + 2 \alpha \\
&= \left( \alpha^{\frac12} (t^2u - t^{-2}) + \frac{\gamma}{2\alpha } \right)^2 + 						 \frac{\sqrt{\Delta}}{\alpha^{\frac12}} (t^2 u +   t^{-2}) + C'
\end{align*}
where $C' := - \frac{\gamma^2}{4\alpha^{2}} + 2\alpha$ is an overall constant and hence
\[
Re(h) > \frac{C}{\alpha^{\frac12}} (t^2 u + t^{-2})  + C'.
\]
We then have
\[
(u-1)^{-\frac12} t^{-1} |e^{- \pi h}| = 	(u-1)^{-\frac12} t^{-1} e^{- \pi Re(h)} < (u-1)^{-\frac12} t^{-1} e^{- \pi C'} e^{- \frac{\pi C}{\sqrt \alpha} \left( t^2 u + t^{-2} \right) }
\]
uniformly on $B_r(\beta^*)$	and it is straightforward to verify the latter function is integrable on the region $\{ u > 1, t>0\} \subset \bbR^2$.

Similarly, if $\alpha<0$, then there are constants $C' ,C''\in\bbR$ such that
\[
Re(h) > |\alpha|(t^4 + t^{-4} ) + 2 |\alpha| u + C' t^{-2} + C''
\]
and we may conclude the argument as before.

\end{proof}

\end{proposition}

\subsection{Analyticity of $\Orb(\Phi_{1,1})$}
 Our proof that $\Orb(\Phi_{1,1})$ is analytic is less direct than the unitary case. We begin by recalling a form of the regularity theorem for solutions of partial differential equations. Let $U\subseteq \mbR^n$ be an open subset and let $\mcD = \sum_{I} f_I \partial^I$ be a differential operator on $U$ with smooth coefficients $f_I\in \mcC^\infty(U)$. The \emph{order} of $\mcD$ is defined as $\max\{|I| \text{ s.t. } f_I\neq 0\}$. Its \emph{principal symbol} is the polynomial
$$\mr{PS}_\mcD(ξ_1,\ldots,ξ_n) = \sum_{|I| = \mr{ord}(\mcD)} f_I ξ^I$$
in variables $ξ_i$ with coefficients in $\mcC^\infty(U)$.
\begin{theorem}[\protect{\cite[Theorem 8.6.1]{Hoermander}}]\label{thm:reg_thm_diff_eqn}
Let $\mcD_1,\ldots,\mcD_r$ be a set of differential operators on the open subset $U\subseteq \mbR^n$, and let $f \in \mcC^\infty(U)$ be a solution to $\mcD_i(f) = 0$ for all $i\in 1,\ldots,r$. Assume:
\begin{itemize}
\item The coefficients of the $\mcD_i$ are analytic functions.
\item The system of homogeneous equations defined by the $\mr{PS}_{\mcD_i}$ has no non-trivial solution. That is, for every $x\in U$, $ξ_1 = \ldots = ξ_n = 0$ is the only solution in $\mbR^n$ to
$$\mr{PS}_{\mcD_i}(x, ξ_1,\ldots,ξ_n) = 0\quad \text{for all }i = 1,\ldots.r$$
\end{itemize}
Then $f$ is an analytic function.
\end{theorem}
\begin{proof}
Considering the special case of the analytic wave front set in the cited theorem and using $\mcD_i(f) = 0$, we get (in the notation of \cite[\S8.6]{Hoermander})
$$\mr{WF}_a(f) \subseteq \mr{Char}(\mcD_i)\cup \mr{WF}_a(0)$$
for all $i = 1,\ldots,r$. We have $\mr{WF}_a(0) = \emptyset$. The condition on principal symbols precisely says $\bigcap_{i = 1}^r \mr{Char}(\mcD_i) = \emptyset$. Hence $\mr{WF}_a(f) = \emptyset$ which means that $f$ is analytic.
\end{proof}

\begin{proposition}\label{prop:analytic_GL2}
For fixed $\alpha,\gamma \neq 0$, the orbital integral $\Orb(-,Φ_{1,1})(\alpha,\beta,\gamma)$ is real analytic as a function of $\beta$, i.e.\ it satisfies part (d) of Proposition \ref{prop:unique_extension}.
\end{proposition}
\begin{proof}
\mn It suffices to show analyticity of the orbital integral of
$$Φ'_{1,1}(y, v, w) := 2\sqrt{π} Φ_{1,1}\big(y/2\sqrt{π},\, v/2\sqrt{π},\, w/2\sqrt{π}\big).$$
We prefer to work with this scaled version because this removes all powers of $π$ in the following. The function $Φ_{1,1}'$ is concretely given by
$$Φ'_{1,1}(y, v, w) = (y_{12} - y_{21})e^{-q(y, v, w)/2}$$
with
$$q(y, v, w) = 2y_{11}^2 + y_{12}^2 + y_{21}^2 + v_1^2 + v_2^2 + w_1^2 + w_2^2.$$

Our goal is to prove analyticity of $\Orb(-,Φ_{1,1}')$ by using Theorem \ref{thm:reg_thm_diff_eqn}, so we need to exhibit suitable  differential operators satisfying the hypotheses therein. We found a system of four such operators $\mcD_1,\ldots,\mcD_4 \in \mbC\{\mfsl_2\times \mbA^2\times (\mbA^2)^t\}^{\GL_2}$ via a systematic computer search in the space of operators  annihilating $\Orb(\Phi'_{1,1})$. Having found these operators, it is relatively straightforward to independently verify that they are $\GL_2$-invariant, and that they annihilate $\Orb(\Phi'_{1,1})$. This verification is certified by the Sage code provided with the paper \cite{Sage}.

The first two operators are explicitly given as follows:
\begin{equation}
\label{eq:quad_annihilating_ops}
\begin{aligned}
\mcD_1 & = 4 \partial_{y_{12}}\partial_{y_{21}} + \partial_{y_{11}}^2 - 4 y_{12}y_{21} - 4 y_{11}^2 - 2,\\[1mm]
\mcD_2 & = - \partial_{v_1}\partial_{w_1} - \partial_{v_2}\partial_{w_2} + v_1 w_1 + v_2 w_2.
\end{aligned}
\end{equation}
They agree with the weight operators $-4(ω(H,0) + 1/2)$ and $ω(0,H)$ up to our scaling of variables, and it is immediately checked that they satisfy $\mcD_1(Φ_{1,1}') = \mcD_2(Φ_{1,1}') = 0$. This vanishing reflects that $Φ_{1,1}$ has biweight $(-1/2, 0)$ for the Weil representation.

The operators $\mcD_3$ and $\mcD_4$ are of order three with cubic coefficients and have quite a few terms; the precise formulas will not be needed for the following discussion but the curious reader may find them in the provided Sage code \cite{Sage}. We have $\mcD_3(Φ_{1,1}') = 0$ while $\mcD_4$ satisfies a relation of the form
\begin{equation}\label{eq:relation_D_4}
\mcD_4(Φ_{1,1}') + E_{11}(f_{11}) + E_{12}(f_{12}) + E_{21}(f_{21}) + E_{22}(f_{22}) = 0
\end{equation}
with $f_{k\ell}\in \mbC[\mfsl_2\times \mbA^2\times (\mbA^2)^t] e^{-q(y,v,w)/2}$ and where the $E_{k\ell}$ denote the standard basis vectors of $\mfgl_2(\mbC)$. Again, we refer the reader to the provided code \cite{Sage} for precise formulas for the $f_{k\ell}$, and for a verification of all claims regarding these operators so far.

In particular, each $\mcD_i$ annihilates the image of $Φ_{1,1}'$ in the space of $\mfgl_2(\mbC)$-coinvariants, and hence annihilates $\Orb(-,Φ_{1,1}')$. Proposition \ref{prop:diff_op_orb_int_generic} applies and shows that
$$(\mcD_i\vert_\mcQ)\Orb(-,Φ_{1,1}') = 0$$
for all $i = 1,\ldots,4$. The restrictions $\mcD_i\vert_\mcQ$ as elements of $\mathbb C\{ \calQ\} = \mbC\{α, β, γ\}$ are calculated by the provided Sage code as well. For example,
\begin{equation}
\label{eq:restriction_abc_D_12}
\begin{aligned}
\mcD_1\vert_\mcQ & = β^2 \partial_γ^2 - 4 γ \partial_α\partial_γ - 4 α \partial_α^2 - 6\partial_α + 4α - 2,\\[1mm]
\mcD_2\vert_\mcQ & = αβ\partial_γ^2 - 2γ\partial_β\partial_γ - β\partial_β^2 - 2\partial_β + β.
\end{aligned}
\end{equation}
Only the principal symbols of the $\mcD_i\vert_\mcQ$ matter for the desired application of Theorem \ref{thm:reg_thm_diff_eqn}. We write $A$, $B$ and $C$ for the polynomial variables corresponding to $\partial_α$, $\partial_β$ and $\partial_γ$ and recall $Δ = αβ^2 + γ^2$ denotes the discriminant. Then the principal symbols are given as follows:
\begin{equation}
\label{eq:princ_symbols}
\begin{aligned}
σ_1 & = β^2C^2 - 4γAC - 4αA^2\\[1mm]
σ_2 & = αβC^2 - 2γBC - βB^2\\[1mm]
σ_3 & = 2αγ^2C^3 + 6αβγBC^2 + 4αβ^2B^2C - 2γ^2B^2C - 2βγB^3\\[1mm]
σ_4 & = -2αΔBC^2 - 2ΔB^3 + βΔAC^2 + 2ΔA^2B.
\end{aligned}
\end{equation}
It remains to check that for all $(α,β,γ)\in \mbR^3$ with $Δ >0$ and $αγ \neq 0$, $A = B = C = 0$ is the only solution to $σ_1 = \ldots = σ_4 = 0$.
We first assume that $β = 0$. Then \eqref{eq:princ_symbols} specializes to
\begin{equation}
\label{eq:princ_symbols_b_zero}
\begin{aligned}
σ'_1 & = - 4γAC - 4αA^2\\[1mm]
σ'_2 & = - 2γBC\\[1mm]
σ'_3 & = 2αγ^2C^3 - 2γ^2B^2C\\[1mm]
σ'_4 & = -2αΔBC^2 - 2ΔB^3 + 2ΔA^2B.
\end{aligned}
\end{equation}
If $C = 0$, then $σ_1' = 0$ implies $A = 0$ and then $σ_4' = 0$ implies $B = 0$. So any non-trivial solution will have $C \neq 0$. But then $σ_2' = 0$ implies $B = 0$ which leads to $σ'_3 \neq 0$. This shows our claim in points of the form $(α, 0, γ)$, so we can henceforth assume $β \neq 0$.

Then $σ_2 = 0$ gives rise to the relation $βB = C( \pm \sqrt{Δ}-γ)$. Substituting this into $β^2σ_3$ gives,  after a short calculation, the equation
\begin{equation}
β^2σ_3 = 4Δ^{3/2}(Δ^{1/2}  \mp γ)C^3.
\end{equation}
The coefficient is always nonzero because $αβγ \neq 0$ implies $Δ^{1/2} \neq |γ|$. So $σ_3 = 0$ implies $C = 0$. Then $σ_1 = 0$ implies $A = 0$, and finally $σ_4 = 0$ gives $B = 0$. That is, $A = B = C = 0$ is the only solution to $σ_1 = \ldots = σ_4 = 0$ for all parameters $(α, β, γ)$ in question, as was to be shown.

\end{proof}

\subsection{Proof of \Cref{thm:main_signature_11}} \label{ss:proof main thm 11}
To conclude the proof of \Cref{thm:main_signature_11}, we define $f \in C^{\infty}(\calQ_\rs(\bbR))$ by
\[
f = \Orb(-,\Phi_{1,1}) +  \frac1{\sqrt 2} \, \Orb(- , \psi_{1,1}).
\]
By \Cref{prop:analytic_GL2,prop:analytic_U,lem:analytic_properties_simple}, the function $f$ satisfies properties (b), (c), and (d) of \Cref{prop:unique_extension}. Moreover, \Cref{prop:beta = 0 match} implies that $f$ is identically zero on the locus $\{ \alpha \gamma \neq 0, \beta = 0\}$. By continuity, this vanishing extends to the locus $\{\beta = 0\} \subset \mathcal Q_{\rs,(1,1)}$.

Thus, by the uniqueness statement of  \Cref{prop:unique_extension} we find that
\[  f \equiv 0 \text{ on } \mathcal Q_{\rs,(1,1)}.\]

It remains to show that $f$ vanishes on $\mathcal Q_{\rs,(2,0)}$ and $\mathcal Q_{\rs,(0,2)}$ as well.  By definition, the unitary orbital integral $\Orb(-, \psi_{1,1})$ vanishes in this region. On the other hand, if $\tau \colon \lie s_2 \to \lie s_2$ is the transposition map $(y,v,w) \mapsto (y^t, w^t, v^t)$, we have
\[
\tau^* \Phi_{1,1} = - \Phi_{1,1}.
\]
The vanishing of $\Orb(-,\Phi_{1,1})$ on $\calQ_{\rs, (2,0)} \cup \calQ_{\rs, (0,2)}$ then follows from \Cref{lem:transposition}.  This concludes the proof.

\bibliographystyle{plain}

\end{document}